\documentclass[12pt]{article}
\usepackage{a4wide}
\usepackage[french,english]{babel}
\usepackage[utf8]{inputenc}
\usepackage[T1]{fontenc}
\usepackage{amsthm}
\usepackage{amssymb}
\usepackage{amsmath}
\usepackage{stmaryrd}
\usepackage{imakeidx}
\usepackage{dsfont}
\usepackage{graphicx}
\usepackage[all]{xy}
\usepackage{caption}
\usepackage{authblk}
\theoremstyle{plain}
\usepackage[colorlinks=true,breaklinks=true,linkcolor=black]{hyperref} 
\usepackage{amsfonts}
\usepackage{esvect}
\usepackage{comment}
\usepackage{mathabx}
\DeclareMathAlphabet\mathrsfso      {U}{rsfso}{m}{n}
\usepackage{ulem}

\newtheorem{Theorem}{Theorem}[section] 
\newtheorem{Proposition}[Theorem]{Proposition}
\newtheorem{Corollary}[Theorem]{Corollary}
\newtheorem{Definition}[Theorem]{Definition}
\newtheorem{Lemma}[Theorem]{Lemma}

\newtheorem{Remark}[Theorem]{Remark}

\newtheorem{Definition/Proposition}[Theorem]{Definition/Proposition}
\theoremstyle{definition}\newtheorem{Theorem*}{Theorem}
\theoremstyle{definition}\newtheorem{Corollary*}[Theorem*]{Corollary}
\theoremstyle{definition}\newtheorem{Proposition*}[Theorem*]{Proposition}
\date{\today}
\makeindex[columns=3]
\usepackage{mathrsfs}
\usepackage[all]{xy}

\makeatletter
\renewcommand\theequation%
{\thesection.\arabic{equation}}
\@addtoreset{equation}{section}
\makeatother

\usepackage{geometry}
\title{Tessellations of an affine apartment by affine weight polytopes}

\author[1]{Claudio \textsc{Bravo}}
\author[2]{Auguste \textsc{Hébert}}
\author[3]{Diego \textsc{Izquierdo}}
\author[4]{Benoit \textsc{Loisel}}

\affil[1]{claudio.bravo-castillo@polytechnique.edu}

\affil[2]{auguste.hebert@univ-lorraine.fr}

\affil[3]{diego.izquierdo@polytechnique.edu}

\affil[4]{benoit.loisel@math.univ-poitiers.fr}

\usepackage{mathrsfs}

\makeatletter \@addtoreset{figure}{section}\makeatother

\newcommand{\R}{\mathbb{R}}
\newcommand{\A}{\mathbb{A}}

\newcommand{\N}{\mathbb{N}}
\newcommand{\Z}{\mathbb{Z}}

\newcommand{\I}{\mathcal{I}}

\newcommand{\Id}{\mathrm{Id}}

\newcommand{\Fr}{\mathrm{Fr}}

\newcommand{\In}{\mathrm{Int}}
\newcommand{\conv}{\mathrm{conv}}

\newcommand{\supp}{\mathrm{supp}}
\newcommand{\cl}{\mathrm{cl}}

\newcommand{\vect}{\mathrm{vect}}
\newcommand{\germ}{\mathrm{germ}}

\newcommand{\CC}{\mathcal{C}}
\newcommand{\DC}{\mathcal{D}}
\newcommand{\FC}{\mathcal{F}}

\newcommand{\QC}{\mathcal{Q}}

\newcommand{\qp}{\varpi}
\newcommand{\ve}{\mathrm{vert}}

\newcommand{\Int}{\mathrm{Int}}

\newcommand{\bb}{\mathbf{b}}
\newcommand{\bt}{\mathbf{t}}
\newcommand{\cH}{\mathcal{H}}
\newcommand{\Extr}{\mathrm{Extr}}
\newcommand{\cE}{\mathcal{E}}
\newcommand{\cF}{\mathcal{F}}
\newcommand{\cM}{\mathcal{M}}
\newcommand{\cP}{\mathcal{P}}
\newcommand{\cQ}{\mathcal{Q}}
\newcommand{\sF}{\mathscr{F}}

\newcommand{\hh}{h}

\newcommand{\Alc}{\mathrm{Alc}}

\newcommand{\di}{\mathrm{dir}}

\newcommand{\cB}{\mathcal{B}}
\newcommand{\facet}{\mathrm{Facet}}

\newcommand{\bk}{\mathds{k}}
\newcommand{\bK}{\mathbb{K}}
\newcommand{\ff}{f}

\makeindex

\usepackage[left,modulo]{lineno}

\hypersetup{
pdfauthor={Bravo, Hébert, Izquierdo, Loisel},
pdftitle={Tessellation apartment},
}

\begin{document}


\maketitle

\abstract{Let $\A$ be  a finite dimensional vector space and $\Phi$ be a finite root system in  $\A$. To this data is associated an affine poly-simplicial complex. Motivated by a forthcoming construction of connectified higher buildings, we study ``affine weight polytopes'' associated to these data. We prove that these polytopes tesselate $\A$.  We also prove a kind of ``mixed'' tessellation, involving the affine weight polytopes and the poly-simplical structure on $\A$.}

\section{Introduction}

\paragraph{Tessellation of $\A$ by affine weight polytopes}  Let $\A$ be a finite dimensional vector space. A finite (reduced) root  system in $\A$ is a finite set $\Phi$ of linear forms on $\A$, called  roots, satisfying certain conditions (see \cite[VI, \S1]{bourbaki1981elements}).
The associated affine (locally finite) hyperplane arrangement is $\cH=\{\alpha^{-1}(\{k\})\mid (\alpha,k)\in \Phi\times \Z\}$. The elements of $\cH$ are called affine walls. The hyperplane arrangement $\cH$ equips $\A$ with the structure of a poly-simplicial complex, whose maximal poly-simplices  - the alcoves - are the connected components of $\A\setminus \bigcup_{H\in \cH}H$.
To each affine wall $H$ is  naturally associated a reflection $s_H$.  Let  $W^{\mathrm{aff}}$ be the affine Weyl group associated with $(\A,\cH)$, i.e., $W^{\mathrm{aff}}$ is the group generated by the $s_H$, for $H\in \cH$. This is naturally a Coxeter system and every closed alcove $\overline{C}$ is a fundamental domain for the action of $W^{\mathrm{aff}}$ on $\A$. 
Fix a ``fundamental alcove'' $C_f$ and pick one element $\bb=\bb_{C_f}$ of $C_f$. Then for every alcove $C$, the set $C\cap W^{\mathrm{aff}}.\bb$ contains exactly one element, that we denote $\bb_C$. Denote by $\ve(\A)$ the set of vertices of $(\A,\cH)$, i.e., the set of zero-dimensional faces of the associated poly-simplicial complex. For $F$   a face  of $(\A,\cH)$, we set \[\cF_F=\cF_{\bb,F}=\Int_r(\conv(W^{\mathrm{aff}}_F.\bb_C)),\] where $W^{\mathrm{aff}}_F$ is the fixator of $F$ in $W^{\mathrm{aff}}$, $C$ is any alcove whose closure contains $F$, $\conv$ denotes the convex hull, and given a subset $E \subset \A$, $\Int_r(E)$ denotes the relative interior of $E$ (that is the interior of $E$ in the affine subspace spanned by $E$). When $F$ is $0$-dimensional, i.e., when $F=\{\lambda\}$ for some $\lambda\in \ve(\A)$, we usually write $\cP(\lambda)$ instead of $\cF_{\{\lambda\}}$. 
Our first main result is the following (see Figures~\ref{f_irregular_tessellation_A2}, \ref{f_regular_tessellation_A2}, \ref{f_tessellation_B2} and \ref{f_tessellation_G2}): 

\begin{Theorem}\label{t_tessellation_intro} (see Propositions~\ref{p_description_faces_affine}, \ref{p_faces_intersections} and Theorem~\ref{t_tessellation_A})
\begin{enumerate}

\item Let  $\lambda\in \ve(\A)$.  Then the assignment $F\mapsto \cF_F$ is a bijection between the set of faces of $(\A,\cH)$ whose closure contains $\lambda$ and the set of faces of $\overline{\cP(\lambda)}$. 

\item Let $F$ be a face of $(\A,\cH)$. Then we have  $\overline{\cF_F}=\bigcap_{\lambda\in \ve(F)}\overline{\cP(\lambda)}$,   where $\ve(F)$ denotes the set of vertices of $F$. 

\item Let $\sF$ denote the set of faces of $(\A,\cH)$. Then $\A=\bigsqcup_{F\in \sF}\cF_F.$

\end{enumerate}
\end{Theorem}

\paragraph{Thickened tessellation of $\A$}  We then study a variant of the tessellation of $\A$ given by Theorem \ref{t_tessellation_intro}(3).
Let $\eta\in ]0,1]$.  Let $x\in \A$. Let $F=F_{\cP}(x)$ be the face of $(\A,\cH)$ such that $\cF_F$ contains $x$ (the existence and uniqueness of such a face is provided by Theorem~\ref{t_tessellation_intro} (3)).  Set $F_\eta(x)=(1-\eta)F+\eta x=\{(1-\eta)y+\eta x\mid y\in F\}$. For $\lambda\in \ve(\A)$, we denote by $h_{\lambda,\eta}:\A\rightarrow \A$ the homothety of $\A$ with center $\lambda$ and radius $\eta$.    Then $\overline{{F_\eta(x)}}$ is a ``copy'' of $\overline{F}$ containing exactly one point of $h_{\lambda,\eta}(\overline{\cP(\lambda)})$, namely $h_{\lambda,\eta}(x)$, for each $\lambda\in \ve(F)$. This is the convex hull of these points. Then we prove the following theorem:

\begin{Theorem}\label{t_thickened_tessellation_intro}
Let $\eta\in ]0,1]$. Then $\A=\bigsqcup_{x\in \A} \overline{F_\eta(x)}$. 
\end{Theorem}

We deduce the following corollary, which is easier to interpret (see Figures~\ref{f_explanation_t_thickened_A2} and \ref{f_explanation_t_thickened_G2}):

\begin{Corollary}\label{c_thickened_tessellation_intro}
Let $\eta\in ]0,1]$. Then we have \[\A=\bigsqcup_{\lambda\in \ve(\A)} h_{\lambda,\eta}\left(\cP\left(\lambda\right)\right)\sqcup \bigsqcup_{x\in \bigcup_{\lambda\in \ve(\A)}\Fr(\cP(\lambda))}\overline{F_\eta(x)}.\]
\end{Corollary}

Note that Theorem~\ref{t_thickened_tessellation_intro} and Corollary~\ref{c_thickened_tessellation_intro} provide some kinds of tessellations of $\A$ involving homothetic images of the $\cP(\lambda)$ and homothetic images of the faces of $(\A,\cH)$.

The polytopes $\overline{\cP(\lambda)}$ that we study already appear in the litterature, in various contexts of representation theory. If $\lambda\in \A$, then we can write $\overline{\cP(\lambda)}=\lambda+\overline{\cP^0_\lambda}$, where $\overline{\cP^0_\lambda}=\conv\{w.\bb'\mid w\in W^v_\lambda\}$, where $W^v_\lambda$ is the finite Weyl group associated with the root system $\Phi_\lambda=\{\alpha\in \Phi\mid \alpha(\lambda)\in \Z\}$ and $\bb'$ is a regular element of $\A$ (i.e., $w.\bb'\neq \bb'$, for every $w\in W^v_\lambda$). The polytope $\overline{\cP^0_\lambda}$ is called  the \textbf{weight polytope} and is studied in \cite{vinberg1991certain} and \cite{besson2023weight}. Its intersection with a closed vectorial chamber of $\A$  is studied in \cite{burrull2023dominant}.

\paragraph{Motivation from the theory of higher buildings} Let $\bK$ be a field equipped with a valuation $\omega:\bK^*\twoheadrightarrow \Z^n$, for some $n\in \Z_{\geq 1}$, where $\Z^n$ is equipped with the lexicographic order.
For example, let $\bk$ be a field and set $\bK=\bk(t_1,\ldots,t_n)$, which is contained in $ \bk(\!(t_n)\!)(\!(t_{n-1})\!)\ldots (\!(t_1)\!)$, with  \[\omega\left(\sum_{(i_1,\ldots,i_n)\in \Z^n} a_{i_1,\ldots,i_n} t_1^{i_1}\ldots t_n^{i_n}\right)=\min\{ (i_1,\ldots,i_n)\in \Z^n\mid a_{i_1,\ldots,i_n}\neq 0\},\] for $\sum_{(i_1,\ldots,i_n)\in \Z^n} a_{i_1,\ldots,i_n} t_1^{i_1}\ldots t_n^{i_n}\in \bK$. Such fields naturally appear as successive completions of function fields of $n$-dimensional algebraic varieties.

We would like to study groups of the form $G=\mathbf{G}(\bK)$, for $\mathbf{G}$ a split reductive group, as well as certain arithmetic subgroups of them (for instance their presentations, certain (co)homological groups or their (co)homological dimension). When $n=1$, then $G$ acts on a classical Bruhat-Tits building  which is connected and contractible (see \cite{bruhat1972groupes} and \cite{bruhat1984groupes}). We would like to use buildings techniques to study $G$, when $n\geq 2$. In this case Bennett and Parshin independently defined analogues of Bruhat-Tits buildings in the case where $\mathbf{G}$ is $\mathrm{SL}_m$ or $\mathrm{PGL}_m$, for $m\in \Z_{\geq 2}$ (see \cite{bennett1994affine} and \cite{parshin1994higher}), using lattices of $\bK^m$. In \cite{bennett1994affine}, Bennett develops an axiomatic definition of $\Lambda$-buildings (where $\Lambda=\omega(\bK^*)$, $\Lambda=\Z^n$ in our case). In \cite{hebert2020Lambda}, the three last named authors prove that $G$ acts on a $\Z^n$-building $\I$, generalizing (in the quasi-split case) Bruhat-Tits construction via ``parahoric'' subgroups.

When $n\geq 2$, $\I$ is a bit like a set of Russian dolls: $\I$ is a Bruhat-Tits building  whose vertices are $\Z^{n-1}$-buildings. A problem is that $\I$ is not connected since the  latter $\Z^{n-1}$-buildings  are connected components of $\I$. As the connectedness and contractibility of Bruhat-Tits buildings are important ingredients of the classical theory, we would like to embed the $\mathbb{Z}^n$-building  $\I$ in a ``connectified building $\hat{\I}$ or $\hat{\I^\eta}$, for $\eta\in ]0,1[$, on which $G$ would act and which would have better topological properties. We hope that our tessellations  are a first step in this direction.

\paragraph{Structure of the paper} The paper is organized as follows. In section~\ref{s_vectorial_Weyl_polytopes}, we study the polytopes $\cP(\lambda)$ (or more precisely translates of these polytopes) one by one. We describe  the set of faces   and  study the orthogonal projection on such a polytope.

In section~\ref{s_tessellation_A}, we study the polytopes $\overline{\cP(\lambda)}$ and prove that they tessellate $\A$.

In section~\ref{s_thickened_tessellation}, we prove Theorem~\ref{t_thickened_tessellation_intro}.

\tableofcontents

\section{Vectorial Weyl polytopes}\label{s_vectorial_Weyl_polytopes}

Let $\A$ be a finite dimensional vector space and $\Phi$ be a finite root system in $\A$. This equips $\A$ with the structure of a poly-simplicial complex. The aim of this paper is to study tesselations  of $\A$ by polytopes $\overline{\cP(\lambda)}$ associated with vertices $\lambda$ of $\A$.  Before studying the layout of the different polytopes, we begin by studying the polytopes one by one.   

In this section, we study the vectorial weight polytope $\overline{\cP^0}=\overline{\cP^0_\Psi}$ associated with a finite root system $\Psi$. In the applications, $\Psi$ will be a root system of the form $\Phi_\lambda=\{\alpha\in \Phi\mid \alpha(\lambda)\in \Z\}$, for $\lambda$ a vertex of $(\A,\Phi)$ and then $\overline{\cP(\lambda)}$ will be a translate of $\overline{\cP^0_{\Phi_\lambda}}$. To define this  polytope, we choose a regular point of $\A$ and we take the convex hull of the orbit of this point under the action of the finite Weyl group of $\Psi$.

In subsection~\ref{ss_preliminary_polytopes}, we study general facts on polyhedra.

In subsection~\ref{ss_vectorial_Weyl_polytope}, we define the vectorial weight polytope $\overline{\cP^0}$.

In subsection~\ref{ss_faces_vectorial_Weyl_polytope}, we describe the faces of $\overline{\cP^0}$.

In subsection~\ref{ss_projections_parabolic_subspaces}, we study the restriction to certain affine subspaces of $\A$ of the orthogonal projection on $\overline{\cP^0}$.

\subsection{Preliminary results on polyhedra and polytopes}\label{ss_preliminary_polytopes}

In this section, $\A$ is a finite  dimensional real vector space. We equip $\A$ with a scalar product $\langle\cdot\mid \cdot\rangle$ (for the moment, $\langle\cdot\mid \cdot\rangle$ is arbitrary, but in most of the applications, it will be invariant under the action of the vectorial or affine Weyl groups defined later).  If $E$ is a subset of $\A$, we denote by $\mathring{E}$ its interior. If $E$ is non-empty, its \textbf{support} $\supp(E)$\index{s@$\supp$} is the smallest affine subspace of $\A$ containing $E$. Its \textbf{relative interior}, denoted $\In_r(E)$\index{i@$\In_r$}, is then the interior of $E$ in its support. In particular, if $E$ is convex and closed, we have $E=\overline{\Int_r(E)}$. When $E$ is an affine subspace of $\A$, its \textbf{direction}, denoted $\di(E)$\index{d@$\di$}, is the vector subspace $\{x-x'\mid x,x'\in E\}$ of $\A$.

We mainly follow the terminology of \cite{bruns2009polytopes}.  A (closed) \textbf{polyhedron} is the intersection of finitely many closed half-spaces of $\A$. A (closed) \textbf{polytope} is a bounded polyhedron. Let $\cQ$ be a polyhedron of $\A$.   A \textbf{support hyperplane} of $\cQ$ is a hyperplane $H$ such that $\cQ$ is contained in one of the closed half-spaces delimited by $H$ and such that $\cQ\cap H\neq \emptyset$. A \textbf{closed face} of $\cQ$ is either  $\cQ$ or  a set   of the form $H\cap \cQ$, where $H$ is a support hyperplane. Contrary to \cite{bruns2009polytopes}, we do not consider $\emptyset$ as a face of $\cQ$. A \textbf{closed facet} of $\cQ$ is a face of codimension $1$ in the support of $\cQ$. By \cite[Theorem 1.10]{bruns2009polytopes}, every proper closed face of $\cQ$ is an intersection of closed facets.  We denote by $\facet(\cQ)$\index{f@$\facet(\cQ)$}   the  set of closed facets of $\cQ$. An \textbf{open face} or simply a \textbf{face} of $\cQ$ is the relative interior of a closed face of $\cQ$. Note that in \cite{bruns2009polytopes}, what we call ``closed faces'' are simply called ``faces'' and the ``open faces'' are the relative interior of faces: we make this choice in order to be coherent with the  usual definition of faces in an apartment (see subsections~\ref{ss_vectorial_Weyl_polytope} and \ref{ss_affine_weight_polytopes}). If $F_1,F_2$ are open faces of a polyhedron for which $\overline{F}_1\cap \overline{F_2}$ is non-empty, we set $F_1\wedge F_2=\Int_r(\overline{F}_1\cap \overline{F_2})$.\index{z@$\wedge$} It is a face of $\cQ$ by \cite[Proposition 1.8 and Theorem 1.10]{bruns2009polytopes}.

\subsubsection{Orthogonal projection on polytopes}

If $E$ is a convex closed subset of $\A$, we denote by $\pi_E:\A\rightarrow E$\index{p@$\pi_E$} the orthogonal projection on $E$: for every $x\in \A$, $\pi_E(x)$ is the unique point $y_x$ of $E$ satisfying $|x-y_x|=\min \{|z-x|\mid z\in E\}$. Moreover, if $y_x\in E$, then we have \begin{equation}\label{e_characterization_projection}  \left(y_x=\pi_E(x)\right)\Longleftrightarrow \left(\langle x-y_x\mid y-y_x\rangle\leq 0, \forall y\in E\right),
\end{equation}
 by \cite[III Theorem 3.1.1]{hiriart1993convex}.

In this preliminary subsection, we study the projection $\pi_{\cQ}$, for $\cQ$ a polyhedron of $\A$.  In particular, Proposition~\ref{p_description_fibers_pi} provides a description of its fibers.

For $x\in \A$ and $r\in \R_{\geq 0}$, we denote by $B(x,r)$\index{b@$B(x,r)$} the open ball of center $x$ and radius $r$.

\begin{Lemma}\label{l_local_description_polytope}
Let $\cQ$ be a polyhedron of $\A$ with non-empty interior. Let $x\in \cQ$ and let $\facet_x(\cQ)$  be the set of closed facets of $\cQ$ containing $x$. For $\overline{\cF}\in \facet_x(\cQ)$ with relative interior $\cF$, we denote by $\DC_\cF$ the closed  half-space of $\A$ delimited by $\cF$ and containing $\cQ$. Let $\tilde{\cQ}=\bigcap_{\overline{\cF}\in \facet_x(\cQ)}\DC_{\cF}$.  Then there exists $r>0$ such that  $B(x,r)\cap \cQ=B(x,r)\cap \tilde{\cQ}$.   
\end{Lemma}

\begin{proof}
    We have $\cQ\subset \DC_\cF$ for every $\overline{\cF}\in \facet_x(\cQ)$ and thus $\cQ\subset \tilde{\cQ}$.  

  For $\overline{\cF}\in \facet(\cQ)\setminus \facet_x(\cQ)$, choose $r_{\cF}>0$ such that $B(x,r_\cF)\cap \overline{\cF}=\emptyset.$ Let $r=\min \{r_\cF\mid \overline{\cF}\in \facet(\cQ)\setminus \facet_x(\cQ)\}$ (we set $r=1$ if  $\facet(\cQ)=\facet_x(\cQ)$). Let $y\in \tilde{\cQ}\cap B(x,r)$. Assume by contradiction that $y\notin \cQ$. Let $z\in \mathring{\cQ}\cap B(x,r)$. Then $[z,y]$ meets $\Fr(\cQ)$  in a point $y'$. By \cite[Corollary 1.7]{bruns2009polytopes}, we have $\Fr(\cQ)=\bigcup_{\overline{\cF}\in \facet(\cQ)} \overline{\cF}$. We have $y'\notin \bigcup_{\overline{\cF}\in \facet_x(\cQ)} \overline{\cF}$ and thus $y'\in \bigcup_{\overline{\cF}\in \facet(\cQ)\setminus \facet_x(\cQ)} \overline{\cF}$. Consequently $y'\notin B(x,r)$: a contradiction.  Therefore $\tilde{\cQ}\cap B(x,r)\subset \cQ$, and the lemma follows.
\end{proof}

Let $\cQ$ be a polyhedron. If $\overline{\cF}\in \facet(\cQ)$, we denote by $\delta_\cF$\index{d@$\delta_\cF$} the (closed) ray of $\di(\supp(\cQ))$, which is  orthogonal to the direction of the support of $\overline{\cF}$ and such that for all  $x\in \overline{\cF}$, $(x+\delta_\cF)\cap \cQ=\{x\}$.   If $\overline{\cF}\in \facet(\cQ)$, we denote by $\DC_\cF$ the half-space of $\A$ containing $\cQ$ and  delimited by the hyperplane containing $\cF$. We have \begin{equation}\label{e_description_DF}
    \DC_\cF=\{x\in \A\mid \langle x-a\mid z\rangle \leq 0,\forall z\in \delta_\cF\},
\end{equation} for any $a\in \overline{\cF}$.

\begin{Lemma}\label{l_projection_polyhedron_tilde}
Let $\cQ$ be a polyhedron with non-empty interior.    Let $x\in \A\setminus \cQ$ and $J=\{\overline{\cF}\in \facet(\cQ)\mid \pi_\cQ(x)\in \overline{\cF}\}.$ Then $\pi_\cQ(x)=\pi_{\tilde\cQ}(x)$, where $\tilde{\cQ}=\bigcap_{\overline{\cF}\in J}\DC_\cF$.
\end{Lemma}

\begin{proof}
  By Lemma~\ref{l_local_description_polytope},  there exists $r>0$ such that $\cQ\cap B(\pi_\cQ(x),r)=\tilde{\cQ}\cap B(\pi_\cQ(x),r)$. We have $d(x,y)\geq d(x,\pi_\cQ(x))$ for every $y\in \tilde{\cQ}\cap B(\pi_\cQ(x),r)$ and thus $d(x,\cdot)|_{\tilde{\cQ}}$ admits a local minimum at $\pi_\cQ(x)$. But as $d(x,\cdot)|_{\tilde{\cQ}}$ is a convex function, by \cite[IV 1.3 (c)]{hiriart1993convex}, this minimum is actually a global minimum, which proves that $\pi_{\cQ}(x)=\pi_{\tilde{\cQ}}(x)$.
\end{proof}

\begin{Proposition}\label{p_description_fibers_pi}
 Let $\cQ$ be a polyhedron with non-empty interior and let $\pi_\cQ$ be the orthogonal projection on $\cQ$.   Let $x\in \Fr(\cQ)$ and $J=\{\overline{\cF}\in \facet(\cQ)\mid x\in \overline{\cF}\}$. Then $\pi_\cQ^{-1}(\{x\})=x+\sum_{\overline{\cF}\in J} \delta_\cF$.
\end{Proposition}

\begin{proof}
 Let $\tilde{\cQ}=\bigcap_{\overline{\cF}\in J} \DC_{\cF}$. Let $y\in x+\sum_{\overline{\cF}\in J}\delta_{\cF}$. Then by Eq . \eqref{e_description_DF}, we have $\langle y-x\mid z-x\rangle\geq 0$ for all $z\in \tilde{\cQ}\supset \cQ$ and hence by Eq . \eqref{e_characterization_projection}, we have $\pi_{\cQ}(y)=x$. Therefore \[\pi_\cQ^{-1}(\{x\})\supset x+\sum_{\overline{\cF}\in J} \delta_\cF.\]

Let now $y\in \pi_{\cQ}^{-1}(\{x\})$. Then by Lemma~\ref{l_projection_polyhedron_tilde}, we have $\pi_{\cQ}(y)=\pi_{\tilde{\cQ}}(y)=x$.  For $\overline{\cF}\in J$, choose $z_\cF\in \delta_\cF\setminus\{0\}$ and respectively define $f_\cF:\A\rightarrow \R$ and $f:\A\rightarrow \R$ by $f_\cF(z)=-\langle z\mid z_\cF\rangle$ and $f(z)=\langle z\mid x-y\rangle$, for $z\in \A$. Let $z\in \A$ be such that $f_\cF(z)\geq 0$, for all $\overline{\cF}\in J$. Let $z'=z+x$. Then we have $f_\cF(z'-x)\geq 0$,  for all $\overline{\cF}\in J$ and thus by Eq . \eqref{e_description_DF}, we have $z'\in \tilde{\cQ}$. By Eq. \eqref{e_characterization_projection}, we deduce $f(z'-x)\geq 0$ and thus $f(z)\geq 0$. Using Farkas' Lemma, we deduce $f\in \sum_{\overline{\cF}\in J} \R_{\geq 0} f_\cF$, or equivalently, $y-x\in \sum_{\overline{\cF}\in J}\R_{\geq 0} z_\cF=\sum_{\overline{\cF}\in J}\delta_\cF$, which proves the proposition.
    
\end{proof}

\subsubsection{Extreme points of a polytope}

Let $\QC$ be a convex subset of $\A$. 
Recall (see \cite[Definition 2.3.1]{hiriart2012fundamentals}) that a point $M$ of $\QC$ is called an \textbf{extreme point} of  $\QC$ if for every $k\in \Z_{\geq 1}$, $M_1,\ldots,M_k\in \QC$, $t_1,\ldots,t_k\in ]0,1]$ such that $\sum_{i=1}^k t_i=1$, we have \[\sum_{i=1}^k t_i M_i =M\Rightarrow M_1=\ldots =M_k=M.\] We denote by $\Extr(\QC)$ the set of extreme points of $\QC$.

\begin{Lemma}\label{l_extreme_points_P}
Let $H$ be a finite subgroup of the set of affine automorphisms of $\A$. Let $\bb \in \A$. Then the set of extreme points of $\conv(H.\bb)$ is $H.\bb$.
\end{Lemma}

\begin{proof}
Denote by $E$ the set of extreme points of $\conv(H.\bb)$. Then $E$ is contained in $H.\bb$. Then $\conv(H.\bb)$ is compact. By Krein-Milman theorem (\cite[Theorem 2.3.4]{hiriart2012fundamentals}), $\conv(H.\bb)$ is the convex hull of $E$, and in particular, $E$ is non-empty. As $H$ stabilizes $\conv(H.\bb)$, $E$ is $H$-invariant and thus $E=H.\bb$.
\end{proof}

\begin{Lemma}\label{l_extreme_points_faces_polytope}
Let $\cQ$ be a closed polytope of $\A$.
\begin{enumerate}
\item Let $\overline{\cF}$ be a closed face of $\cQ$. Then $\Extr(\overline{\cF})=\Extr(\cQ)\cap \overline{\cF}$.

\item Let $(\overline{\cF_j})_{j\in J}$ be a family of closed faces of $\cQ$. Then $\Extr(\bigcap_{j\in J} \overline{\cF_j})=\bigcap_{j\in J} \Extr(\overline{\cF_j})$. 
\end{enumerate}
\end{Lemma}

\begin{proof}
1) By \cite[Proposition 1.13]{bruns2009polytopes}, the set of closed faces of $\cQ$ (resp. $\overline{\cF}$) is the set of extreme sets of $\cQ$ (resp. $\overline{\cF}$) and thus the set of extreme points of $\overline{\cF}$ is the set of zero-dimensional faces of $\overline{\cF}$. That is the set of zero-dimensional faces of $\cQ$, that are contained in $\overline{\cF}$,  according to \cite[Proposition 1.8]{bruns2009polytopes}. Thus $\Extr(\overline{\cF})=\Extr(\cQ)\cap \overline{\cF}$. 

2) Since the number of faces of $\cQ$ is finite (by \cite[Theorem 1.10]{bruns2009polytopes}), we may assume that $J$ is finite. Then by  \cite[Theorem 1.10]{bruns2009polytopes},  $\bigcap_{j\in J} \overline{\cF_j}$ is a face of $\cQ$.   Then, by  statement 1),  we have $\Extr(\bigcap_{j\in J} \overline{\cF_j})=\Extr(\cQ)\cap \bigcap_{j\in J}\overline{\cF_j})=\bigcap_{j\in J} (\Extr(\cQ)\cap \overline{\cF_j})=\bigcap_{j\in J}\Extr(\overline{\cF_j})$. 
\end{proof}

\subsection{Vectorial weight polytope}\label{ss_vectorial_Weyl_polytope} 

\subsubsection{Vectorial root system}

Let $\A$ be a finite dimensional real vector space and denote by $\A^*:=\mathrm{Hom}_\mathbb{R}(\mathbb{A},\mathbb{R})$ its dual space. Let $\Psi\subset \A^*$\index{P@$\Psi$} be a reduced finite root system in $\A$, in the sense of \cite[VI, \S1]{bourbaki1981elements}. In particular, $\Psi$ is a finite subset of $\A^*$ generating $\A^*$. Elements of $\Psi$ are called \textbf{roots} and to every $\alpha\in \Psi$ is associated a unique \textbf{coroot} $\alpha^\vee$. We have $\alpha(\alpha^\vee)=2$ and the map $\A^*\rightarrow \A^*$ defined by $\beta\mapsto \beta-\beta(\alpha^\vee)\alpha$ stabilizes $\Psi$. One defines $s_\alpha:\A\rightarrow \A$ by $s_\alpha(x)=x-\alpha(x)\alpha^\vee$ for $x\in \A$.  The set $\Psi^\vee=\{\alpha^\vee\mid \alpha\in \Psi\}\subset \A$ is a root system in $\A^*$ and for every $\alpha\in \Psi$, $\alpha(\Psi^\vee)\subset \Z$ and $s_\alpha$ stabilizes $\Psi^\vee$.

Let $\cH_\Psi^0=\cH^0=\{\alpha^{-1}(\{0\})\mid \alpha\in \Psi\}$\index{H@$\cH^0_\Psi$, $\cH^0$}. This is a locally finite hyperplane arrangement of $\A$. The elements of $\cH^0$ are called \textbf{walls} or \textbf{vectorial walls}.  For $x,y\in \A$, we write $x\sim_{\cH^0} y$ if for every $H\in \cH^0$, either ($x,y\in H$) or ($x$ and $y$ are strictly on the same side of $H$). This is an equivalence relation on $\A$. Its classes are called the \textbf{faces} of $(\A,\cH^0)$ or simply the vectorial faces of $\A$. We denote by $\sF^0=\sF(\cH^0)$\index{f@$\sF^0$} the set of faces of $(\A,\cH^0)$. If $E_1,E_2$ are two relatively open (i.e open in their support)   convex subsets of $\A$, we say that $E_1$ \textbf{dominates} $E_2$ if $\overline{E_1}\supset E_2$. This equips the set of relatively open convex subsets of $\A$ with an order. By restriction, this equips $\sF^0$ with an order called the  \textbf{dominance order}.
 
A \textbf{vectorial chamber} (or \textbf{open vectorial chamber}) of $(\A,\cH^0)$ is a connected component of $\A\setminus \bigcup_{H\in \cH^0} H$. The chambers are actually the faces which are not contained in an element of $\cH^0$ or equivalently, the  maximal faces for the dominance order. A \textbf{wall} of a vectorial chamber $C^v$ is an element $H$ of $\cH$ such that $H\cap \overline{C^v}$ is a closed facet of $\overline{C^v}$.

Let $C^v$ be a vectorial chamber of $\Psi$. The \textbf{associated basis of $\Psi$}\index{s@$\Sigma_{C^v}$} is the set $\Sigma_{C^v}$ of $\alpha\in \Psi$ such that $\alpha^{-1}(\{0\})$ is a wall of $C^v$ and $\alpha(C^v)>0$.  Then by \cite[VI. 1.5 Théorème 2]{bourbaki1981elements}, $\Sigma_{C^v}$ is a \textbf{basis} of $\Psi$, which means that $(\alpha)_{\alpha\in \Sigma_{C^v}}$ is a basis of $\A^*$ and every element of $\Psi$ can be written as $\sum_{\alpha\in \Sigma_{C^v}} n_\alpha \alpha$, where $(n_\alpha)\in (\epsilon\Z_{\geq 0})^{\Sigma_{C^v}}$, for some $\epsilon\in \{-,+\}$. We set $\Psi_{\epsilon,C^v}=\Psi\cap \bigoplus_{\alpha\in \Sigma_{C^v}} \epsilon\Z_{\geq 0}\alpha$, for $\epsilon\in \{-,+\}$.

The following proposition follows from \cite[\S V 3.9 Proposition 5]{bourbaki1981elements}.

\begin{Proposition}\label{p_bourbaki_vectorial_faces}

Let $C^v$ be a vectorial chamber of $(\A,\cH^0)$. Let $\Sigma_{C^v}$ be the associated basis of $\Psi$. Let $(\qp_{\alpha})_{\alpha\in \Sigma_{C^v}} \in  \A^{\Sigma_{C^v}}$ be the dual basis of $\Sigma_{C^v}$. Then \begin{enumerate}
    \item $C^v=\bigoplus_{\alpha\in \Sigma_{C^v}}\R_{>0}\qp_\alpha$.

    \item The faces of $C^v$ are exactly the $\bigoplus_{\alpha\in E}\R_{>0} \qp_\alpha$ such that $E$ is a subset of $\Sigma_{C^v}$, where, by convention, $\bigoplus_{\alpha\in \emptyset} \R_{>0}\qp_\alpha=\{0\}$.
\end{enumerate}
    
\end{Proposition}

We choose a vectorial chamber $C^v_f=C^v_{f,\Psi}$\index{c@$C^v_f$, $C^v_{f,\Psi}$} of $\A$, that we call the \textbf{fundamental vectorial chamber}. Let $\Sigma=\Sigma_{C^v_f}$\index{s@$\Sigma$} 

The \textbf{vectorial Weyl group}  of $\Psi$ is $W^v_\Psi=\langle s_\alpha\mid \alpha\in \Psi\rangle\subset \mathrm{GL}(\A)$\index{w@$W^v_\Psi$}. It is a finite group and $(W^v_\Psi,\{s_\alpha\mid \alpha\in \Sigma\})$ is a Coxeter system (see \cite[1.1]{bjorner2005combinatorics} for a definition). By \cite[V \S 3]{bourbaki1981elements}, $W^v_\Psi$ acts simply transitively on the set of vectorial chambers and $\overline{C^v}$ is a fundamental domain for the action of $W^v_\Psi$ on $\A$, for any vectorial chamber $C^v$. By \cite[VI 1.1 Proposition 3]{bourbaki1981elements}, there exists a $W^v_\Psi$-invariant scalar product $\langle  \cdot\mid \cdot\rangle$\index{z@$\langle  \cdot\mid \cdot\rangle$} on $\A$. We equip $\A$ with this scalar product and regard $\A$ as a Euclidean space. An element $x$ of $\A$ is \textbf{regular} (with respect to $\cH^0$) if it belongs to an open vectorial chamber of $\A$. Actually, $x$ is regular if and only if $w\mapsto w.x$ is injective on $W^v_\Psi$. If $F^v\in \sF(\cH^0)$ is dominated by $C^v_f$, then its fixator in $W^v_\Psi$ is 
\begin{equation}\label{fixator}
    W^v_{F^v}=\langle s_\alpha\mid \alpha\in \Sigma,\alpha(F^v)=\{0\}\rangle \index{w@$W^v_{F^v}$},
\end{equation} according to \cite[V \S 3.3]{bourbaki1981elements}.

We sometimes assume that $\Psi$ is \textbf{irreducible.} This means that $\Psi$ cannot be written as a direct sum of roots systems, see \cite[VI \S 1.2]{bourbaki1981elements}. When $\Psi$ is reducible, we   can write \begin{equation}\label{e_decomposition_Phi}
    \A=\bigoplus_{i=1}^\ell\A_i, \Psi=\bigsqcup_{i=1}^\ell \Psi_i
\end{equation}  where $\ell \in \Z_{\geq 2}$ and $\Psi_i$ is an irreducible root system in $\A_i$ for $i\in \llbracket 1,\ell\rrbracket$. We regard $\Psi_i$ as a subset of $\Psi$ by setting $\alpha(\bigoplus_{j\in \llbracket 1,\ell\rrbracket \setminus \{i\}}\A_j)=\{0\}$ for $\alpha\in \Psi_i$.  We denote by $\cH_i^0$ the hyperplane arrangement of $(\A_i,\Psi_i)$, for $i\in \llbracket 1,k\rrbracket$. Then every face $F^v$ of $\sF^0$ can be written as $F^v=F_1^v\times \ldots \times F_\ell^v$, where $F^v_i\in \sF^0( \mathcal{H}_i^0)$, for $i\in \llbracket 1,\ell\rrbracket$.

\paragraph{Cartan matrices}
Let $A=A(\Psi)=(\alpha(\beta^\vee))_{\alpha,\beta\in \Sigma}$. This is the \textbf{Cartan matrix} of $\Psi$. Up to permutation of its coefficients, it is well-defined, independently of the choice of the basis $\Sigma$ of $\Psi$, by \cite[VI, 1.5]{bourbaki1981elements}.  The Cartan matrix characterizes the root system up to isomorphism. 

A matrix of the form $A(\Phi)$, for some   root system $\Phi$ is called a \textbf{Cartan matrix}.

The Cartan matrix $A$ is called \textbf{decomposable} if there exists a  decomposition $\Sigma=\Sigma'\sqcup \Sigma''$, with $\Sigma'$ and $\Sigma''$ non-empty, such that $\alpha(\beta^\vee)=\beta(\alpha^\vee)=0$ , for every $(\alpha,\beta)\in \Sigma'\times \Sigma''$. A Cartan matrix which is not decomposable is called \textbf{indecomposable.} Then $\Psi$ is irreducible if and only if $A$ is indecomposable, and under the notation of Eq. \eqref{e_decomposition_Phi}, $A$ is the block diagonal matrix whose blocks are the $A(\Psi_i)$, for $i\in \llbracket 1,k\rrbracket$. One can use the Cartan matrices to classify finite roots systems, see \cite[Chapter 4]{kac1994infinite} for example (where the Cartan matrices are the generalized Cartan matrices of finite type).

\subsubsection{Vectorial weight polytope}

\begin{Definition}\label{d_vectorial_Weyl_polytope}
 We choose a point $\bb\in C^v_f$ and we set $\overline{\cP^0}=\overline{\cP_{\bb,\Psi}^0}=\overline{\cP^0_\bb}=\conv(W^v_\Psi.\bb)$\index{p@$\overline{\cP^0}$, $\overline{\cP^0_\Psi}$}. We also set $\cP ^0=\In(\overline{\cP^0})$\index{p@$\cP^0$}. Then $\cP^0$ (resp. $\overline{\cP^0}$) is the open (resp. closed) \textbf{vectorial Weyl polytope associated with $\Psi$ and $\bb$.} 
\end{Definition} 
 
The closed vectorial Weyl polytope is compact. Note that by   Lemma \ref{l_convex_hull_non-empty_interior}  below,  $\cP^0$ is non-empty. This polytope naturally appears in different contexts. It is called the \textbf{weight polytope} and it appears for example  in \cite{vinberg1991certain} and \cite{besson2023weight}. Its intersection with $\overline{C^v_f}$ is studied in \cite{burrull2023dominant}.
   
 \begin{figure}[h]
 \centering
 \includegraphics[scale=0.2]{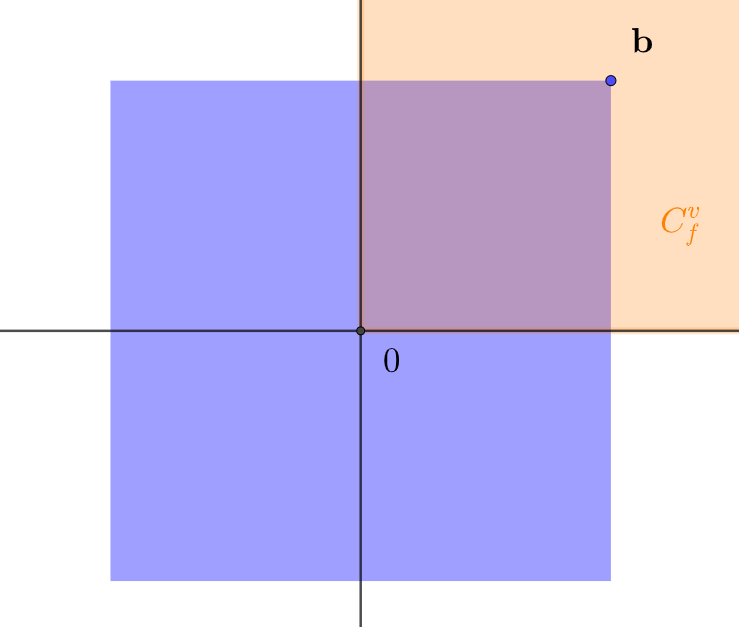}
 \caption{Polytope $\cP^0$ for the root system $\mathrm{A}_1\times \mathrm{A}_1$}
\end{figure}

 \begin{figure}[h]
 \centering
 \includegraphics[scale=0.2]{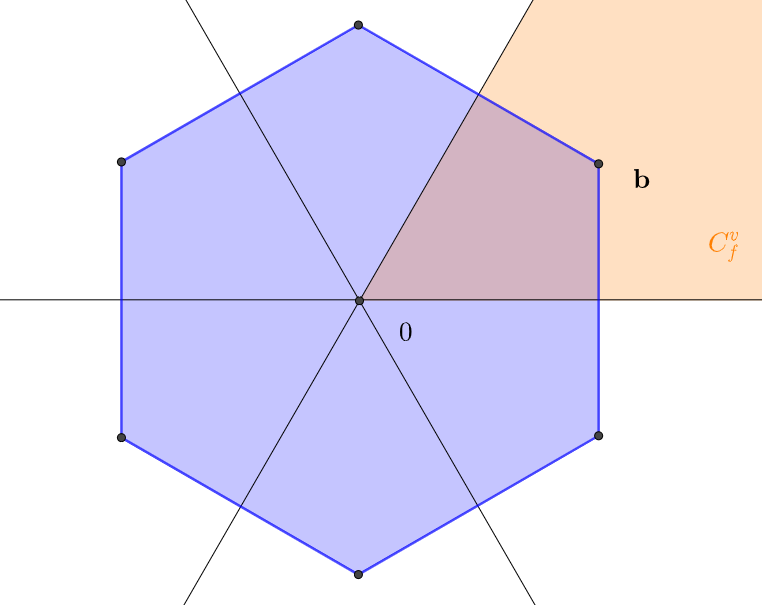}
 \caption{Polytope $\cP^0$ for the root system $\mathrm{A}_2$}
\end{figure}
    
  \begin{figure}[h]
 \centering
 \includegraphics[scale=0.2]{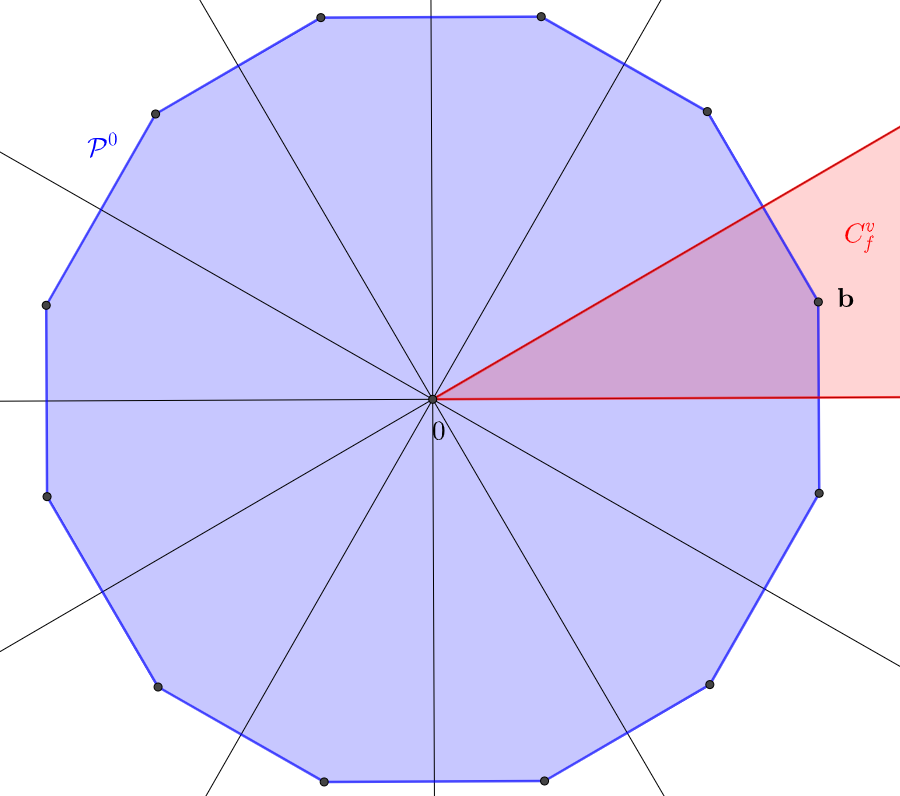}
 \caption{Polytope $\cP^0$ for the root system $\mathrm{G}_2$}
\end{figure}
 
 \begin{Lemma}\label{l_convex_hull_non-empty_interior}
\begin{enumerate}
    \item Assume that $\Psi$ is irreducible. Then for every $x\in \A\setminus\{0\}$, $\conv(W^v_\Psi.x)$ has non-empty interior. 

    \item $\cP^0$ is non-empty.
\end{enumerate} 
 \end{Lemma}
 
 \begin{proof}
The support of $W^v_\Psi.x$ is $W^v_\Psi$-invariant and thus by \cite[VI  \S 1.2 Corollaire]{bourbaki1981elements}, this support is $\A$, which proves 1). 

By 1), if $\Psi$ is irreducible, then $\cP^0$ is non-empty. Assume now that $\Psi$ is reducible and keep the notation of Eq . \eqref{e_decomposition_Phi}. Write $\bb=(\bb_1,\ldots,\bb_\ell)$. Then $\cP^0_{\bb,\Psi}=\cP_{\bb_1,\Psi_1}^0\times \ldots \times \cP_{\bb_\ell,\Psi_\ell}^0$ and thus $\cP^0_{\bb,\Psi}$ is non-empty. 
 \end{proof}

\subsection{Description of the faces of the vectorial weight polytope}\label{ss_faces_vectorial_Weyl_polytope}

\subsubsection{Faces of $\cP^0$ meeting $C^v_f$}

The main goal of this subsection is to prove Lemma \ref{p_facets_P}, which describes the faces the of $\mathcal{P}^0$ meeting $C^v_f$.

For $J\subset \Sigma$, one defines the \textbf{parabolic subgroup} $W^v_{J}$ by $W^v_J=W^v_{J,\Psi}=\langle s_\alpha\mid \alpha\in J\rangle\subset W^v_{\Psi}$. Set $Q^\vee_+=Q^\vee_{\R,+}=\bigoplus_{\alpha\in \Sigma} \R_{\geq 0}\alpha^\vee$,   $Q^\vee_{J,+}=\bigoplus_{\alpha\in J} \R_{\geq 0} \alpha^\vee$.

We equip $W^v_\Psi$ with the Bruhat order (see \cite[Chapter 2]{bjorner2005combinatorics} for a definition).

\begin{Lemma}\label{l_Bruhat_order_Qvee_order}
Let $w\in W^v_{\Psi}$ and $\alpha\in \Sigma$. Then $ ws_\alpha<w$ for the Bruhat order on  $(W^v_{\Psi},\{s_\alpha\mid \alpha\in \Sigma\})$ if and only if $w.\alpha\in \Psi_{-}$.
\end{Lemma}

\begin{proof}
This follows from  \cite[VI, \S 1 Proposition 17 (ii)]{bourbaki1981elements} and \cite[Corollary 1.4.4]{bjorner2005combinatorics}. 
\end{proof}

\begin{Lemma}\label{l_link_parabolic_subgroups_regular_coweights}
Let  $w\in W^v_{\Psi}$ and $J\subset \Sigma$. The following assertions are equivalent:
\begin{enumerate}
    \item[(i)]  we have $\bb-w.\bb\in \bigoplus_{\alpha\in J} \R \alpha^\vee$,
    \item[(ii)] we have $\bb-w.\bb\in Q^\vee_{J,+}$,
    \item[(iii)] we have $w\in W^v_{J}$.
\end{enumerate}
\end{Lemma}

\begin{proof}
By \cite[Chapitre VI, §1 Proposition 18]{bourbaki1981elements}, if $w\in    W^v_{J}$, then  $\bb-w.\bb\in Q^\vee_{J,+} \subset \bb-w.\bb\in \bigoplus_{\alpha\in J} \R \alpha^\vee$. Hence (iii) implies (ii), which itself implies (i).

It remains to prove that (i) implies (iii). To do so, take $w\in W^v_{\Psi}$ and assume that $w\in W^v_{\Psi}\setminus W^v_{J}$. Write $w=s_{\alpha_k}\ldots s_{\alpha_1}$, with $k=\ell_\Psi(w)$, $\alpha_1,\ldots,\alpha_k\in \Sigma$, where $\ell_\Psi$ denotes the length of the Coxeter system  $(W^v_{\Psi},\{s_\alpha\mid \alpha\in \Sigma\})$. Then there exists $j\in \llbracket 1,k\rrbracket $ such that $\alpha_j\notin J$. Take $j$ minimal with this property.
Then $s_{\alpha_j}.s_{\alpha_{j-1}}\ldots s_{\alpha_1}.\bb=s_{\alpha_{j-1}}\ldots s_{\alpha_{1}}.\bb-\alpha_{j}(s_{\alpha_{j-1}}\ldots s_{\alpha_1}.\bb).\alpha_{j}^\vee$.  As $\bb$ is regular,  Lemma~\ref{l_Bruhat_order_Qvee_order} implies $s_{\alpha_j}.s_{\alpha_{j-1}}\ldots s_{\alpha_1}.\bb\in    \bb+\bigoplus_{\alpha\in J} \R\alpha^\vee-\R_{>0} \alpha_{j}^\vee$.
By an induction using Lemma~\ref{l_Bruhat_order_Qvee_order}, we deduce \[w.\bb\in \bb+(\bigoplus_{\alpha\in J} \R\alpha^\vee-\R_{>0}\alpha_{j}^\vee)-\bigoplus_{\alpha\in \Sigma}\R_{\geq 0} \alpha^\vee.\]
In particular, $w.\bb-\bb\notin \bigoplus_{\alpha\in J} \R \alpha^\vee$, which completes the proof of the lemma.
\end{proof}

\begin{Lemma}\label{l_neighborhood_lambda_P}
Let $\eta=\frac{1}{|\Sigma|}\min \{\alpha(\bb)\mid \alpha\in \Sigma\}$. Then for every $(t_\alpha)\in ]0,\eta[^{\Sigma}$, $\bb-\sum_{\alpha\in \Sigma} t_\alpha \alpha^\vee \in\cP^0$.
\end{Lemma}

\begin{proof} Let $m=\min \{\alpha(\bb)\mid \alpha\in \Sigma\}$. Let $\alpha\in \Sigma$. Then $\bb-m\alpha^\vee\in [\bb,s_\alpha.\bb]\subset\cP^0$. Let  $(t_\alpha)\in ]0,\eta[^{\Sigma}$. For $\alpha\in \Sigma$, set $t_\alpha'=\frac{1}{m}t_\alpha\in ]0,\frac{1}{|\Sigma|}[$. Set $t'_0=1-\sum_{\alpha\in \Sigma} t_\alpha'\in ]0,1[$. Then \[t_0'\bb +\sum_{\alpha\in \Sigma} t_\alpha' (\bb-m\alpha^\vee)=\bb -\sum_{\alpha\in \Sigma} t_\alpha'm \alpha^\vee=\bb-\sum_{\alpha\in \Sigma} t_\alpha\alpha^\vee \in\overline{\cP^0}.\]
As $(\alpha^\vee)_{\alpha\in \Sigma}$ is a basis of $\A$, the map $(\R_{>0})^{\Sigma}\rightarrow \A$ defined by $(n_\alpha)_{\alpha\in \Sigma} \mapsto \sum_{\alpha\in \Sigma} n_\alpha\alpha^\vee$ is open, and the lemma follows.
\end{proof}

Let $\CC_\Psi=\bigoplus_{\alpha\in \Sigma}\R_{>0}\alpha^\vee$\index{c@$\CC_\Psi$}. By \cite[Chapitre VI, §1 Proposition 18]{bourbaki1981elements}, we have:

\begin{Lemma}\label{l_inclusion_cP_CC}
We have $\cP^0\subset \bb-\CC_\Psi$.
\end{Lemma}

For $\beta\in \Sigma$, set $W^v_\beta=W^v_{\Sigma\setminus\{\beta\},\Psi}$. Then $W^v_\beta$ is the fixator of the one dimensional face $\overline{C^v_f}\cap \bigcap_{\alpha\in \Sigma\setminus \{\beta\}} \ker(\alpha)$, according to \eqref{fixator}.

\begin{Lemma}\label{l_frontier_cP}
 Let $w\in W^v_\Psi$. Then $[w.\bb,\bb]\subset \Fr(\cP^0)$ if and only if $w\in \bigcup_{\beta\in \Sigma} W_\beta^v$.
\end{Lemma}

\begin{proof}
Let $w\in W^v_\Psi\setminus \bigcup_{\beta\in \Sigma} W_\beta^v$. Since $w\in W^v_\Psi$, Lemma~\ref{l_link_parabolic_subgroups_regular_coweights} shows that $u:=w.\bb-\bb \in \bigoplus_{\alpha\in \Sigma} \R_{\geq 0} \alpha^\vee$. By using again Lemma~\ref{l_link_parabolic_subgroups_regular_coweights}, since $w\not\in \bigcup_{\beta\in \Sigma} W_\beta^v$, we get that $u\in \bigoplus_{\alpha\in \Sigma} \R_{<0} \alpha^\vee$. By Lemma~\ref{l_neighborhood_lambda_P}, $\bb+tu\in \cP^0$ for $t\in [0,1]$ small enough. We have $w.\bb=\bb+u$. As $\cP^0$ is convex, either $]\bb,\bb+u[\subset \Fr(\cP^0)$ or $]\bb,\bb+u[\subset\cP^0$ and thus $]\bb,\bb+u[=]\bb,w.\bb[\subset \cP^0$.

Let $w\in \bigcup_{\beta\in \Sigma} W^v_\beta$. Then by Lemma~\ref{l_link_parabolic_subgroups_regular_coweights} and Lemma~\ref{l_inclusion_cP_CC}, $[w.\bb,\bb]\subset \A\setminus (\bb-\CC_\Psi)\subset \A\setminus \cP^0$, which proves that $[w.\bb,\bb]\subset \Fr(\cP^0)$.
\end{proof}

For $J\subset \Sigma$, set $\A_J=\bigcap_{\alpha\in J} \ker(\alpha)$. 

\begin{Lemma}\label{l_scalar_product_coroots_fundamental_weights}
Let $i,j\in I$. Then $\langle \qp_i\mid\alpha_j^\vee\rangle=\frac{1}{2}\delta_{i,j}|\alpha_i^\vee|^2$. In particular, for $J\subset \Sigma$, we have $\A_J^\perp=\bigoplus_{\alpha\in J} \R\alpha^\vee$.
\end{Lemma}

\begin{proof}
We have $\langle \qp_i\mid\alpha_i^\vee\rangle=\langle s_{\alpha_i}.\qp_i \mid s_{\alpha_i}.\alpha_i^\vee\rangle=\langle \qp_i-\alpha_i^\vee\mid-\alpha_i^\vee\rangle$, so $2\langle \qp_i\mid\alpha_i^\vee\rangle=|\alpha_i^\vee|^2$. If $j\neq i$, we have $\langle \qp_i\mid\alpha_j^\vee\rangle=\langle s_{\alpha_j}.\qp_i\mid s_{\alpha_j}.\alpha_j^\vee\rangle=\langle \qp_i\mid-\alpha_j^\vee\rangle$, and hence $\langle \qp_i\mid\alpha_j^\vee\rangle=0$. 
\end{proof}

Let $\beta\in \Sigma$. Set $\FC_\beta=\Int_r(\conv(W_\beta^v.\bb))$. Denote by $H_\beta$ the affine hyperplane orthogonal to $\bigcap_{\alpha\in \Sigma\setminus\{\beta\}} \ker(\alpha)$ and containing $\bb$. 

\begin{Lemma}\label{p_facets_P}
\begin{enumerate}
\item The facets of $\cP^0$ are the $w.\FC_\beta$ for $w\in W^v_{\Psi}$ and $\beta\in \Sigma$.

\item For $\beta\in \Sigma$, the hyperplane containing $\FC_\beta$ is $H_\beta=\bb+\bigoplus_{\alpha\in \Sigma\setminus\{\beta\}}\R\alpha^\vee$.
\end{enumerate}
\end{Lemma}

\begin{proof}
As $\cP^0$ is $W^v_\Psi$-invariant, $\Fr(\cP^0)$ is $W^v_\Psi$-invariant.  Let $\FC$ be a facet of $\cP^0$. Let $E$ be the set of extreme points of $\overline{\FC}$. By  Lemma~\ref{l_extreme_points_P} and Lemma~\ref{l_extreme_points_faces_polytope}, $E\subset  W^v_\Psi.\bb$ and thus replacing $\FC$ by $w.\FC$ for some $w\in W^v_\Psi$ if needed, we may assume $\bb\in \overline{\FC}$.

Let $x\in \overline{\FC}$. Then by Lemma~\ref{l_inclusion_cP_CC}, we can write $x=\bb-\sum_{\alpha\in \Sigma} n_\alpha \alpha^\vee$, where $(n_\alpha)\in (\R_{\geq 0})^{\Sigma}$. Set \[\supp(x)=\{\alpha\in \Sigma\mid n_\alpha\neq 0\}\subset \Sigma.\] Then for $x,y\in \overline{\FC}$, we have $\frac{1}{2}(x+y)\in \overline{\FC}$ and $\supp(\frac{1}{2}(x+y))=\supp(x)\cup \supp(y)$. Therefore there exists $x_0\in \overline{\FC}$ which has the largest
support, i.e., $\supp(x)\subseteq \supp(x_0)$ for every $x\in \overline{\FC}$.

We have $[x_0,\bb]\subset \Fr(\cP^0)$. By Lemma~\ref{l_neighborhood_lambda_P} we deduce the existence of $\beta\in \Sigma$ such that $t\bb+(1-t)x_0\in \bb+\bigoplus_{\alpha\in \Sigma\setminus\{\beta\}} \R \alpha^\vee$ for $t\in [0,1]$  large enough. But then $x_0\in \bb+\bigoplus_{\alpha\in \Sigma\setminus\{\beta\}} \R \alpha^\vee$. Therefore $\supp(x_0)\subset \Sigma\setminus \{\beta\}$ and $\supp(x)\subset \Sigma\setminus\{\beta\}$ for every $x\in \FC$. In particular, for every extreme point $Q$ of $\overline{\FC}$, we have $\supp(Q)\subset \Sigma\setminus \{\beta\}$. By Lemma~\ref{l_link_parabolic_subgroups_regular_coweights}, every extreme point of $\overline{\FC}$ is contained in $W_\beta^v.\bb$. By Krein-Milman theorem we deduce $\overline{\FC}\subset \conv(w.\bb\mid w\in W_\beta^v)= \overline{\FC_\beta}$.

As $\overline{\cP^0}\cap H_\beta$ is a $W_\beta^v$-stable set containing $\bb$, $\overline{\cP^0}\cap H_\beta\supset \FC_\beta$. Moreover, $\overline{\cP^0}\cap H_\beta$ contains $\{\bb\}\cup\{\bb-\alpha(\bb)\alpha^\vee\mid \alpha\in \Sigma\setminus\{\beta\}\}$. As the affine subspace spanned by this set is $H_\beta$, and as $\overline{\FC_\beta}\subset \Fr(\cP^0)$ (by Lemma~\ref{l_frontier_cP}), $\overline{\cP^0}\cap H_\beta$  is a closed facet of $\overline{\cP^0}$. By what we proved above we deduce that \[\overline{\cP^0}\cap H_\beta=\conv(W_\beta^v.\bb).\]

Thus, Statement (1) follows. Moreover, Statement (2) follows from Lemma~\ref{l_scalar_product_coroots_fundamental_weights}.
\end{proof}

\subsubsection{Arbitrary faces of $\cP^0$}
\begin{Definition}\label{d_b_faces}

Let $C^v$ be a vectorial chamber of $(\A,\cH_\Psi^0)$.
Since $\overline{C^v_f}$ is a fundamental domain for the action of $W^v_\Psi$ on $\A$, there exists $w\in W^v_{\Psi}$ be such that $C^v=w.C^v_f$.
Then, we write:
$$\bb_{C^v}=w.\bb.$$\index{b@$\bb_{C^v}$}
Note that $\bb_{C^v}$ does not depend on the choose of $w$, since any other $w'$ as above has the form $w'=ww_0$, where $w_0$ stabilizes $C^v_f$.

If $F^v\in \sF^0$, we choose a vectorial chamber $C^v$ dominating $F^v$, and we set:
$$\cF_{F^v}=\In_r\left(\conv(W^v_{F^v}.\bb_{C^v})\right).$$\index{f@$\cF_{F^v}$}
This does not depend on the choice of $C^v$, since $W^v_{F^v}$ acts transitively on the set of vectorial chambers dominating $F^v$, according to \cite[Chapitre V, §4 Proposition 5]{bourbaki1981elements}.
\end{Definition}

The next two lemmas will be given without proof. They can be settled in a very similar way to their affine counterparts, which will be settled later, in Lemmas~\ref{l_union_alcoves_containing_face} and \ref{l_barycenter_orbit}).

\begin{Lemma}\label{l_union_chambers_dominating_face}
Let $F^v$ be a face of $(\A,\cH^0_\Psi)$ and let $E$ be the \textbf{star} of $F^v$, i.e., the  union of the faces of $(\A,\cH^0_\Psi)$ dominating $F^v$. Let $C^v$ be a chamber of $(\A,\cH^0_\Psi)$ dominating $F^v$ and $\Sigma_{C^v}$ be the basis of $\Psi$ associated with $C^v$. Let $E'$ be the union of the faces of $C^v$ dominating $F^v$. Let $\Sigma_{F^v}=\{\alpha\in \Sigma_{C^v}\mid \alpha(F^v)=\{0\}\}$.  For $\alpha \in \Sigma_{C^v}$, set $\mathring{D}_\alpha=\alpha^{-1}(\R_{>0})$. Then \[E=W^v_{F^v}.E'=\bigcap_{w\in W^v_{F^v},\alpha\in \Sigma_{C^v}\setminus \Sigma_{F^v}} w.\mathring{D}_\alpha.\] In particular, $E$ is open and convex.
\end{Lemma}

\begin{Lemma}\label{l_barycenter_orbit_vectoriel}
Let $F^v,F_1^v$ be  faces of $(\A,\cH^0_\Psi)$ such that $F_1^v$ dominates $F^v$. Let $x\in F_1^v$ and $y=\frac{1}{|W^v_{F^v}|}\sum_{w\in W^v_{F^v}} w.x$. Then $y\in F^v$ and $y$ is the orthogonal projection of $x$ on $\vect(F^v)$.
\end{Lemma}

\begin{Lemma}\label{l_root_system_Fv}
Let $F^v$ be a face of $(\A,\cH^0_\Psi)$ which is not a chamber. Then $\Phi_{F^v}=\{\alpha\in \Phi\mid \alpha(F^v)=\{0\}\}$ is a root system in $\vect(F^v)^\perp$. Its Weyl group is $W^v_{F^v}=\langle s_{\alpha}\mid \alpha\in \Phi_{F^v}\rangle$.
\end{Lemma}

\begin{proof}
Up to changing the choice of $C^v_f$ if needed, we may assume that $C^v_f$ dominates $F^v$. Let $\Sigma_{F^v}=\Sigma\cap \Phi_{F^v}$. We have $\vect(F^v)\subset \bigcap_{\alpha\in \Sigma_{F^v}}\ker(\alpha)$. As $\Sigma$ is a basis of $\A^*$ and using a simple dimension argument based on Proposition \ref{p_bourbaki_vectorial_faces}, we deduce $\vect(F^v)=\bigcap_{\alpha\in \Sigma_{F^v}}\ker(\alpha)$. By Lemma~\ref{l_scalar_product_coroots_fundamental_weights}, $\alpha^\vee\in \vect(F^v)^\perp$ for every $\alpha\in \Sigma_{F^v}$.  Therefore $\Phi_{F^v}|_{\vect(F^v)^\perp}$ satisfies \cite[VI \S 1.1 Définition 1]{bourbaki1981elements}.
\end{proof}

We equip $\sF^0$ with the dominance order $\preceq$: for $F_1^v,F_2^v\in \sF^0$, we have $F_1^v\preceq F_2^v$ if and only if $F_1^v\subset \overline{F_2^v}$.

A less precise version of the following proposition already appears in \cite[Proposition 3.2]{vinberg1991certain}. 

\begin{Proposition}\label{p_description_faces_P0} (see Figure~\ref{f_Faces_P0})
\begin{enumerate}

\item  The assignment $f:F^v\mapsto \FC_{F^v}$ defines  an isomorphism of ordered spaces between   $(\mathscr{F}^0,\succeq)$ and the set of faces of  $(\cP^0,\subset)$.

\item Let $(F_i^v)_{i\in \llbracket 1,k\rrbracket}$ be a finite family of faces of $(\A,\cH_\Psi^0)$. Then if $\sum_{i=1}^k F_i^v\in \sF^0$ (or equivalently if there exists $F^v\in \sF^0$ dominating $\sum_{i=1}^k F_i^v$), we have $\bigcap_{i=1}^k \overline{\cF_{F_i^v}}=\overline{\cF_{\sum_{i=1}^k F_i^v}}$. Otherwise, we have $\bigcap_{i=1}^k \overline{\cF_{F_i^v}}=\emptyset$. 

\item  For  $F^v\in \sF^0$,  the  set of extreme points of $\overline{\cF_{F^v}}$ is $W^v_{F^v}.\bb_{C^v}$, for any chamber  $C^v$ dominating $F^v$. The direction of $ \supp(\cF_{F^v})$ is $(F^v)^\perp$ and in particular, $\dim(\vect(\cF_{F^v}))=\mathrm{codim}_{\A}(\vect(F^v))$. 

\item Let $\cF$ be a face of $\cP^0$ and let $y$ be the barycenter of the extreme points of $\overline{\cF}$. Then $\cF=\cF_{F^v}$, where $F^v$  is the unique face of $(\A,\cH^0_\Psi)$ containing $y$, the point $y$ is the orthogonal projection of $\cF$ on $F^v$ and $\overline{\cF_{F^v}}\cap F^v=\{y\}$. Moreover, $\overline{\cF_{F^v}}$ is contained in the star of $F^v$. 

\end{enumerate}

\end{Proposition}

  \begin{figure}[h]
 \centering
 \includegraphics[scale=0.35]{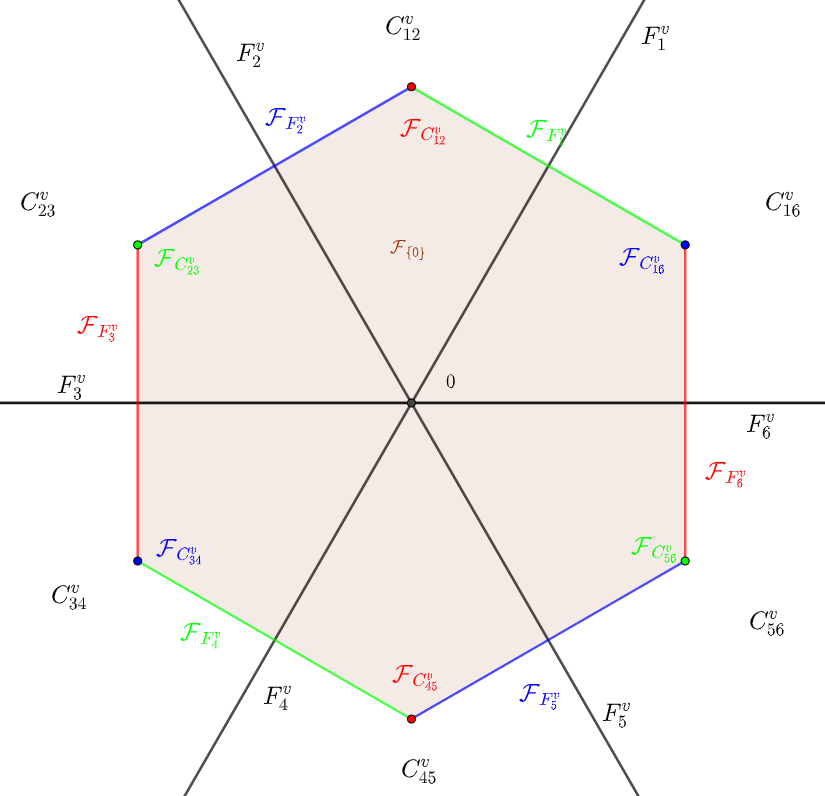}
 \caption{Illustration of Proposition~\ref{p_description_faces_P0} in the case of $\mathrm{A}_2$}
 \label{f_Faces_P0}
\end{figure}

\begin{proof}
\textit{Step 1: We first prove that the facets of $\cP^0$ are the $\cF_{F^v}$, where the $F^v$ are the one-dimensional faces of $(\A,\cH^0_\Psi)$.}
 To do so, let $\beta\in \Sigma$. Set $F^v_\beta=\beta^{-1}(\R_{>0})\cap \bigcap_{\alpha\in \Sigma\setminus\{\beta\}} \ker(\alpha).$ Then the fixator of $F^v_\beta$ in $W^v_{\Psi}$ is $W^v_\beta$ and thus $\cF_{F^v_\beta}=\cF_\beta$.  If $w\in W^v_{\Psi}$, then \[w.\overline{\cF_{F^v_\beta}}=\conv(wW^v_{F^v_\beta}.\bb)=\conv(wW^v_{F^v_\beta} w^{-1}.w.\bb)=\conv(W^v_{w.F^v_\beta}w.\bb)=\overline{\cF_{w.F^v_\beta}}.\] By Lemma~\ref{p_facets_P}, the facets of $\cP^0$ are the $w.\cF_\beta$ such that $w\in W^v_{\Psi}$ and $\beta\in \Sigma$ so, by Proposition~\ref{p_bourbaki_vectorial_faces}, the facets of $\cP^0$ are the $\cF_{F^v}$, where the $F^v$ are the one-dimensional faces of $(\A,\cH^0_\Psi)$. \\

\textit{Step 2: We now prove the second assertion of 2): if $(F_i^v)_{i\in \llbracket 1,k\rrbracket}$ is a family of faces of $(\A,\cH_\Psi^0)$ such that $\overline{\cF}:=\bigcap_{i=1}^k \overline{\cF_{F_i^v}}\neq\emptyset$, then $\sum_{i=1}^k F_i^v\in \sF^0$. } Indeed, take $x\in \overline{\cF}$. Let $F_x^v$ be the face of $(\A,\cH^0_\Psi)$ containing $x$. Then by Lemma~\ref{l_union_chambers_dominating_face}, $F_x^v$ dominates $\bigcup_{i=1}^k F_i^v$. Therefore any chamber dominating $F^v_x$ dominates $\bigcup_{i=1}^k F_i^v$. \\

\textit{Step 3: We now prove that the map $f:F^v\mapsto \FC_{F^v}$ is well-defined.} To do so, let $F^v$ be a face of $(\A,\cH^0_\Psi)$ and $F^v_1,\ldots, F^v_k$ be the one dimensional faces of $F^v$. Then $\overline{\cF}:=\bigcap_{i=1}^k \overline{\cF_{F_i^v}}$ is a closed face of $\cP^0$ by \cite[Theorem 1.10 and Proposition 1.8]{bruns2009polytopes} and it is non-empty because it contains $\overline{\cF_{F^v}}$. In particular, by Step 2, there exists a vectorial chamber $C^v$ dominating $\bigcup_{i=1}^k F_i^v$, and by Proposition~\ref{p_bourbaki_vectorial_faces}, we get $F^v=\sum_{i=1}^k F^v_i$. We then have: \begin{align*}
\Extr(\overline{\cF})&=\bigcap_{i=1}^k \Extr(\overline{\cF_{F^v_i}}) &\text{ by Lemma~\ref{l_extreme_points_faces_polytope}} \\
&=\bigcap_{i=1}^k W^v_{F_i^v}.\bb_{C^v}  &\text{ by Lemma~\ref{l_extreme_points_P} }\\
&= \left(\bigcap_{i=1}^k W^v_{F_i^v}\right).\bb_{C^v}  &\text{ since }\bb_{C^v}\text{ is regular}\\
&= W^v_{F^v}.\bb_{C^v}.
\end{align*} 
By Krein-Milman Theorem, we deduce $\overline{\cF}=\overline{\cF_{F^v}}$. Hence $\cF_{F^v}=\In_r(\overline{\cF_{F^v}})$ is a face of $\cP^0$.\\

\textit{Step 4: We now prove that the map $f:F^v\mapsto \FC_{F^v}$ of statement 1) is surjective.} Indeed, let $\cF$ be a proper face of $\cP^0$. Then by \cite[Theorem 1.10]{bruns2009polytopes}, there exist $k\in \Z_{\geq 1}$, one dimensional vectorial faces $F_1^v,\ldots,F_k^v$ of $(\A,\cH^0_\Psi)$ such that $\overline{\cF}=\bigcap_{i=1}^k \overline{\cF_{F_i^v}}$ and there exists a vectorial chamber $C^v$ dominating $\bigcup_{i=1}^k F_i^v$. By the same computations as in Step 3, we have $\overline{\cF}=\overline{\cF_{F^v}}$ with $F^v:=\sum_{i=1}^k F^v_i$. Hence $\cF=\In_r(\overline{\cF})=\In_r(\overline{\cF_{F^v}})=\cF_{F^v}$.\\

\textit{Step 5: We prove that $f$ is non-decreasing.} Let $F_1^v,F_2^v\in \sF^0$ be such that $F_1^v\subset \overline{F_2^v}$. Then as $W^v_\Psi$ acts by homeomorphisms on $\A$, we have  $W_{F_1^v}^v\supset W_{F_2^v}^v$ and any chamber dominating $F_2^v$ dominates $F_1^v$. Therefore $\overline{\cF_{F_1^v}}\supset \overline{\cF_{F_2^v}}$. \\

\textit{Step 6: We prove that, if $F_1^v,F_2^v\in \sF^0$ are such that $\overline{\cF_{F_1^v}}\supset \overline{\cF_{F_2^v}}$, then $F_2^v$ dominates $F_1^v$.} Indeed, by Step 2, there exists a vectorial chamber $C^v$ dominating $F_1^v$ and $F_2^v$. By Lemma~\ref{l_extreme_points_faces_polytope}, we have $\Extr(\overline{\cF_{F_1^v}})=W_{F_1^v}.\bb_{C^v}\supset \Extr(\overline{\cF_{F_2^v}})=W_{F_2^v}.\bb_{C^v}$. As $\bb_{C^v}$ is regular, we deduce $W^v_{F_1^v}\supset W^v_{F_2^v}$. This implies $F_1^v\subset \overline{F_2^v}$ by \cite[V \S 3.3 Proposition 1]{bourbaki1981elements}, Proposition~\ref{p_bourbaki_vectorial_faces} and \cite[Proposition 2.4.1]{bjorner2005combinatorics}.\\

\textit{Step 7: We prove assertion 1).} We have seen in Steps 4, 5 and 6 that $f$ is surjective and non-decreasing, and that if $f(F_1^v) \supset f(F_2^v)$, then $F_2^v$ dominates $F_1^v$. This implies that $f$ is an isomorphism of ordered sets.\\

\textit{Step 8: We finish the proof of 2).} Let $(F_i^v)_{i\in \llbracket 1,k\rrbracket}$ be a family of faces of $(\A,\cH_\Psi^0)$ such that $F^v:=\sum_{i=1}^k F_i^v\in \sF^0$. Exactly the same arguments as in Step 3 show that $\overline{\cF_{F^v}}=\bigcap_{i=1}^k \overline{\cF_{F_i^v}}$.\\

\textit{Step 9: We prove 3) and 4).} Indeed, let now $\cF$ be a face of $\cP^0$ and $F^v\in \sF^0$ be such that $\cF=\cF_{F^v}$. Then we proved in Lemma \ref{l_extreme_points_P} that $\Extr(\overline{\cF})=W^v_{F^v}.\bb_{C^v}$, where $C^v$ is a vectorial chamber dominating $F^v$. Let $y=\frac{1}{|W^v_{F^v}|}\sum_{w\in W^v_{F^v}}w.\bb_{C^v}$. Then $y$ is the barycenter of $\Extr(\overline{\cF_{F^v}})$.
By Lemma~\ref{l_barycenter_orbit_vectoriel},  it belongs to $F^v$. 

Let $F^v$ be a face of $(\A,\cH^0_\Psi)$ and $C^v$ be a chamber dominating $F^v$. Let $x$ be the orthogonal projection of $\bb_{C^v}$ on $\vect(F^v)$. Then $x\in F^v$ by Lemma~\ref{l_barycenter_orbit_vectoriel}. We have $W^v_{F^v}.\bb_{C^v}=x+W^v_{F^v}.(\bb_{C^v}-x)$. Therefore the direction of $\supp(\conv(W^v_{F^v}.\bb_{C^v}))$ is the vector space spanned by $W^v_{F^v}.(\bb_{C^v}-x)$.  As $\bb_{C^v}$ is regular, $|W^v_{F^v}.\bb_{C^v}|=|W^v_{F^v}|=|W^v_{F^v}.(\bb_{C^v}-x)|$. Therefore $\bb_{C^v}-x$ is a regular element of $(F^v)^\perp$ (for the root system $\Phi_{F^v}$ defined in  Lemma~\ref{l_root_system_Fv}). Using Lemma~\ref{l_convex_hull_non-empty_interior},  we deduce that $W^v_{F^v}.(\bb_C-x)$ spans  $(F^v)^\perp$. Finally, the fact that $\overline{\cF_{F^v}}$ is contained in the star of $F^v$ is a direct consequence of Lemma \ref{l_union_chambers_dominating_face}.
\end{proof}

\begin{Corollary}\label{c_faces_containing_point}
\begin{enumerate}
\item Let $F^v\in \sF^0$. Then $\cF_{F^v}=\overline{\cF_{F^v}}\setminus \bigcup_{F_2^v\in \sF^0, \overline{F_2^v}\supsetneq \overline{F^v}} \overline{\cF_{F_2^v}}$.

\item Let $x\in \overline{\cP^0}$. Then there exists a unique face $F_1^v$ of $(\A,\cH^0_\Psi)$ such that $x\in \cF_{F_1^v}$. Moreover, $\{F^v\in \sF^0\mid x\in \overline{\cF_{F^v}}\}=\{F^v\in \sF^0\mid F^v\subset \overline{F_1^v}\}$. 
\end{enumerate}
\end{Corollary}

\begin{proof}
1) Let $F^v\in \sF^0$.  Let $\cF_2$ be a face of $\cF_{F^v}$. By \cite[Proposition 1.8]{bruns2009polytopes}, $\cF_2$ is a face of $\cP^0$ and thus by Proposition~\ref{p_description_faces_P0}, there exists $F_2^v\in \sF^0$ such that $\cF_2=\cF_{F_2^v}$ and $F_2^v$ dominates $F^v$. We then conclude with \cite[Corollary 1.7]{bruns2009polytopes}.

2) Let $x\in \overline{\cP^0}$. By \cite[Theorem 1.10 d)]{bruns2009polytopes}, there exists a unique face $\cF$ of $\cP^0$ containing $x$ and $\overline{\cF}$ is the minimal closed face containing $x$. By Proposition~\ref{p_description_faces_P0}, we can write $\cF=\cF_{F_1^v}$ for some face $F_1^v$ of $(\A,\cH^0_\Psi)$ and $E=\{F^v\in \sF^0\mid x\in \overline{\cF_{F^v}}\}=\{F_2^v\in \sF^0\mid \overline{F_2^v}\subset \overline{F_1^v}\}$. 
\end{proof}

\begin{Corollary}\label{c_fibers_pi}
    Let $F^v$ be a face of $(\A,\cH_\Psi^0)$. Let $\pi=\pi_{\overline{\cP^0}}$ and $x\in \cF_{F^v}$. Then $\pi^{-1}(\{x\})=x+\overline{F^v}$. 
\end{Corollary}

\begin{proof}
  Let $G^v$ be  a one dimensional face of $(\A,\cH_\Psi^0)$. Then by Proposition~\ref{p_description_faces_P0}, the direction of $\supp(\cF_{G^v})$ is $(G^v)^\perp$. Hence we have $\delta_{G^v}=\pm\overline{G^v}$, with the notation of subsection~\ref{ss_preliminary_polytopes}. By Proposition~\ref{p_description_faces_P0}, $G^v\cap \cF_{G^v}=\{z\}$, for some $z\in G^v$. Therefore $z-G^v$ contains $0$ and thus it meets $\cP^0$, which proves that $\delta_{G^v}=\overline{G^v}$. 

Now let $k$ be the dimension of the direction of $F^v$ and let  $F_1^v,\ldots,F_k^v$ be the one-dimensional faces of $F^v$, so that $F^v=\sum_{i=1}^k F_i^v$. Then by Corollary~\ref{c_faces_containing_point} 2), the set of closed facets containing $x$ is the set of closed facets $\overline{\cF_{G^v}}$ such that  $\overline{G}^v\subset \overline{F}^v$: this is $\{F_i^v\mid i\in \llbracket 1,k\rrbracket\}$. By Proposition~\ref{p_description_fibers_pi}, $\pi^{-1}(\{x\})=x+\sum_{i=1}^k \overline{F^v_i}=x+\overline{F^v}$.
\end{proof}

\begin{Lemma}\label{l_positivity_roots_faces}
Let $F^v$ be a face of $(\A,\cH_\Psi^0)$ and $\alpha\in \Phi$ be  such that $\alpha(F^v)>0$. Then $\alpha(\overline{\cF_{F^v}})>0$.
\end{Lemma}

\begin{proof}
Let $x\in \cF_{F^v}$ and $F_1^v\in \sF^0$ be such that $x\in F_1^v$. Then by Proposition~\ref{p_description_faces_P0}  4), $\overline{F_1^v}\supset F^v$. As $F_1^v$ is a face of $(\A,\cH^0_\Psi)$, we have $\alpha(F_1^v)\in \{\R_{<0},\{0\},\R_{>0}\}$, thus $\alpha(\overline{F_1^v})\in \{\R_{\leq 0},\{0\},\R_{\geq 0}\}$.  As $\alpha(F^v)=\R_{>0}$, we have $\alpha(\overline{F_1^v})=\R_{\geq 0}$ and $\alpha(F_1^v)=\R_{>0}$, which proves the lemma.
\end{proof}

\subsection{Projections restricted to ``parabolic'' affine subspaces}\label{ss_projections_parabolic_subspaces}

In this subsection, we study the restriction of the orthogonal projection  $\pi_{\overline{\cP^0}}$ to certain affine subspaces parallel to the faces of $(\A,\cH^0_\Psi)$. We prove that under some conditions, subsets of these subspaces are stabilized by the projection (see Proposition~\ref{p_projection_preserving_affine_spaces}). The proof is a bit intricate, and one of the main tools consists in introducing a suitable auxiliary set $\overline{\tilde{\cP}}$ contained in the affine subspaces we want to study and such that the orthogonal projections on $\overline{\cP^0}$ and on $\overline{\tilde{\cP}}$ coincide on large subsets of those affine subspaces.

For $F^v$ a one dimensional face of $(\A,\cH_\Psi^0)$, we denote by  $H_{F^v}$ the hyperplane of $\A$ containing $\overline{\cF_{F^v}}$ and by $\DC_{F^v}$ the half-space delimited by $H_{F^v}$ and containing $\cP^0$.  We write $\pi$\index{p@$\pi$} instead of $\pi_{\overline{\cP^0}}$. By Proposition~\ref{p_description_faces_P0} (3), if $z\in \cF_{F^v}$, then we have: \begin{equation}\label{e_notation_half-space}
    \DC_{F^v}=\{y\in \A\mid \langle y-z\mid x\rangle\leq 0,\forall x\in F^v\}.
\end{equation}

\begin{Lemma}\label{l_projection_contained_face}
\begin{enumerate}
    \item Let $x\in \A$ and $F^v\in \sF^0$ be such that $x\in F^v$. Then $\pi(x)\in F^v$ and for every $\tilde{F}^v\in \sF^0$ such that  $\pi(x)\in \overline{\cF_{\tilde{F}^v}}$,  we have $\tilde{F}^v\subset\overline{F^v}$.

    \item Let $F^v\in \sF^0$. Then $\pi^{-1}(F^v)\subset F^v$.
\end{enumerate}    
\end{Lemma}

\begin{proof}
1) Let $F_1^v\in \sF^0$ be such that $\pi(x)\in \cF_{F_1^v}$. By Corollary~\ref{c_fibers_pi}, $x-\pi(x)\in \overline{F_1^v}$. Let $F_2^v\in \sF^0$ be such that $\pi(x)\in F_2^v$. By Proposition~\ref{p_description_faces_P0}, we have $\overline{F_2^v}\supset F_1^v$. Therefore $x\in \overline{F_1^v}+F_2^v\subset \overline{F_2^v}+F_2^v\subset F_2^v$. As the vectorial faces of $(\A,\cH^0_\Psi)$ form a partition of $\A$, we deduce $F_2^v=F^v$, which proves that $\pi(x)\in F^v$. Moreover, by Corollary~\ref{c_faces_containing_point}, for every $\tilde{F}^v$ such that $\pi(x)\in \overline{\cF_{\tilde{F}^v}}$, we have $\tilde{F}^v\subset \overline{F_1^v}\subset \overline{F_2^v}=\overline{F^v}$. 

2) Let $x\in F^v\cap \overline{\cP^0}$ and $y\in \pi^{-1}(x)$. Let $F_1^v\in \sF^0$ be such that $y\in F_1^v$. Then by 1), $x\in F_1^v$ and as the faces of $(\A,\cH^0_\Psi)$ form a partition of $\A$, we deduce that $F_1^v=F^v$. 
\end{proof}

We index the basis $\Sigma$ of $\Psi$ associated with $C^v_f$ by a finite set $I$, i.e., we write $\Sigma=\{\alpha_i\mid i\in I\}$, where $|I|=|\Sigma|$.  Let $(\qp_i)\in \A^I$  be  the dual basis of  $(\alpha_i)_{i\in I}$,  i.e., a basis satisfying $\alpha_i(\qp_j)=\delta_{i,j}$ for all $i,j\in I$.
For $i\in I$, we set $$\DC_i=\{x\in \A\mid \langle x-\bb \mid \qp_i\rangle\leq 0\}$$ and $$\DC_i^0=\{x\in \A\mid \langle x\mid\qp_i\rangle\leq 0\}$$ (we have $\DC_i=\DC_{\R_{>0}\qp_i}$  for the notation of Eq. \eqref{e_notation_half-space},  by Lemma~\ref{l_scalar_product_coroots_fundamental_weights}). For $J\subset I$, we set $V_J=\bigoplus_{j\in J} \R\qp_j$\index{v@$V_J$}. 

\begin{Lemma}\label{l_description_intersection_half_spaces}
Let $V$ be a finite dimensional vector space, $k\in \Z_{>0}$ and $f_1,\ldots,f_k\in V^*$ be a free family. For $i\in \llbracket 1,k\rrbracket$, set $E_i=f_i^{-1}(\R_{\geq 0})$ and  $L_i=E_i\cap \bigcap_{j\in \llbracket 1,k\rrbracket \setminus \{i\}} f_j^{-1}(\{0\})$. Then \[\bigcap_{i=1}^k E_i=\conv(\bigcup_{i=1}^k L_i).\]
\end{Lemma}

\begin{proof}
 Let $(f_i)_{i\in \cB}$ be a basis of $V^*$, where $\cB$ is a set containing $\llbracket 1,k\rrbracket$. For $i\in \cB$, let $e_i\in V$ be such that $f_j(e_i)=\delta_{i,j}$, for $j\in \cB$. Then $\bigcap_{i=1}^k E_i=\bigoplus_{i=1}^k \R_{\geq 0} e_i\oplus V'$, where $V'=\bigoplus_{i\in \cB\setminus \llbracket 1,k\rrbracket} \R e_i$. Let $i\in \llbracket 1,k\rrbracket$. Then $L_i=\R_{\geq 0} e_i\oplus V'$, which proves the result.
\end{proof}

\begin{Lemma}\label{l_Kac_4.3_2}
 Let $J\subset I$. For any $x\in (\bigoplus_{j\in J} \R \alpha_j^\vee)$, we have: \[(\alpha_j(x)\geq 0, \forall j \in J)\Rightarrow x\in \bigoplus_{j\in J} \R_{\geq 0} \alpha_j^\vee.\]
\end{Lemma}

\begin{proof}
For $K\subset J$, denote by $A_K$ the Cartan matrix $A_K=(\alpha_i(\alpha_j^\vee))_{i,j\in K}$.  Write $J=J_1\sqcup \ldots \sqcup J_k$, with $A_{J_i}$ indecomposable  for every $i\in \llbracket 1,k\rrbracket$.  Let $x\in \bigoplus_{j\in J}\R \alpha_j^\vee$. Write $x=\sum_{i=1}^k x_i$, with $x_i\in \bigoplus_{\ell\in J_i}\R \alpha_\ell^\vee$, for $i\in \llbracket 1,k\rrbracket$. Assume that $\alpha_j(x)\geq 0$ for all $j\in J$. Let $i\in \llbracket 1,k\rrbracket$. Then $\alpha_j(x)=\alpha_j(x_i)$ for all $j\in J_i$. By  \cite[Theorem 4.3 (Fin)]{kac1994infinite} applied to the transpose of $A_{J_i}$, we have $x_i\in \bigoplus_{j\in J_i}\R_{\geq 0} \alpha_j^\vee$, and the lemma follows.
\end{proof}

\color{black}

For $i\in I$, we set $\cF_i=\cF_{\R_{>0}\qp_i}$, with the notation of Definition~\ref{d_b_faces}. The $\cF_i$, for $i\in I$, are the open facets of $\cP^0$ meeting $C^v_f$.

\begin{Lemma}\label{l_intersection_half_spaces_0}
Let $J\subset I$. Then  we have  \[  V_J\cap \bigcap_{j\in J} \DC_j^0\subset \bigoplus_{i\in I}\R_{\leq 0}\alpha_i^\vee\subset \bigcap_{i\in I} \DC_i^0.\] \end{Lemma}

\begin{proof}
If $J$ is empty, the result is clear so we assume $J$ non-empty. For $j\in J$, denote by $H_j^0$ the hyperplane delimiting $\DC_j^0$. We have: \[H_j^0=\{x\in \A\mid \langle x\mid \qp_j\rangle=0\}\] and by Lemma~\ref{l_scalar_product_coroots_fundamental_weights},  $H_j^0=\bigoplus_{k\in I\setminus \{j\}}\R\alpha_k^\vee$. Set   $L_j=V_J\cap\bigcap_{k\in J\setminus \{j\}}H_k^0\cap \DC_j^0$. As  $(\qp_k)_{k\in J}$ is  a free family of $\A$, the family $(x\mapsto \langle x\mid\qp_k\rangle)_{k\in J}$ is a free family of $V_J^*$ and thus by Lemma~\ref{l_description_intersection_half_spaces}, \begin{equation}\label{e_descrption_polytope_lines}
 V_J\cap \bigcap_{j\in J} \DC_j^0=\conv(L_j\mid j\in J).
 \end{equation}

 Let $j\in J$. Let $V'_j=\bigoplus_{k\in (I\setminus J)\cup \{j\}} \R\alpha_k^\vee$.  As $(\alpha_i(\alpha_{i'}^\vee))_{i,i'\in (I\setminus J)\cup\{j\}}$ is a Cartan matrix (by \cite[Lemma 4.4]{kac1994infinite} for example), it is invertible and  $(\alpha_k|_{V'_{j}})_{k\in (I\setminus J)\cup \{j\}}$ is a free family of $(V_j')^*$. Therefore there exists $x_j\in V_j'$ such that $\alpha_k(x_j)=-\delta_{j,k}$, for $k\in (I\setminus J)\cup \{j\}$. By Lemma~\ref{l_Kac_4.3_2}, we have $x_j\in -\bigoplus_{k\in (I\setminus J)\cup\{j\}}\R_{\geq 0} \alpha_k^\vee$.  By Lemma~\ref{l_scalar_product_coroots_fundamental_weights}, we deduce $\langle \qp_j\mid x_j \rangle\leq 0$ and $x_j\in H_k^0$ for all $k\in J\setminus \{j\}$.  Moreover, $x_j\in V_J$ and thus $x_j\in L_j$.  For dimension reasons, $L_j=\R_{\geq 0} x_j$ and hence by Eq. \eqref{e_descrption_polytope_lines}, we have $V_J\cap \bigcap_{j\in J}\DC_j^0\subset \bigoplus_{i\in I}\R_{\leq 0}\alpha_i^\vee$. We conclude with Lemma~\ref{l_scalar_product_coroots_fundamental_weights}.
\end{proof}

As in \S \ref{ss_preliminary_polytopes}, we write $\bigwedge_{j\in J} \cF_j=\mathrm{Int}_r(\bigcap_{j \in J} \overline{\cF}_j)$, for $J\subset I$.
By Proposition~\ref{p_description_faces_P0}, we have: \begin{equation}\label{e_intersection_standard_faces}
    \bigwedge_{j\in J} \cF_j=\cF_{\sum_{j\in J}\R_{>0}\qp_j}.
\end{equation}

\begin{Lemma}\label{l_faces_affine_subspace}
Let $J\subset I$. Let $V_J=\bigoplus_{j\in J}\R\qp_j$. Let $z\in\bigwedge_{j\in J} \cF_j$. Then for all $i\in I\setminus J$, $\overline{\cP^0}\cap (z+V_J)\subset \mathring{\DC_i}$.  
\end{Lemma}

\begin{proof}
We have $\overline{\cP^0}\subset \bigcap_{i\in I} \DC_i\subset \bigcap_{j\in J}\DC_j$.  By Corollary~\ref{c_faces_containing_point} and Eq. \eqref{e_intersection_standard_faces}, we have: \begin{equation}\label{e_containing_z}
    \{i\in I\mid z\in \overline{\cF_i}\}=J.
\end{equation}

By Eq. \eqref{e_notation_half-space} we deduce $\bigcap_{j\in J} \DC_j=z+\bigcap_{j\in J} \DC_j^0$. Let $x\in (z+V_J)\cap \overline{\cP^0}$. Then we can write  $x=z+v$, with $v\in V_J$. We have $x\in \bigcap_{j\in J}\DC_j$ and thus 
$v\in V_J\cap \bigcap_{j\in J}\DC_j^0$. Let $i\in I\setminus J$. By Lemma~\ref{l_intersection_half_spaces_0}, we have $v\in \DC_i^0$. By Eq. \eqref{e_containing_z}, $z\in \mathring{\DC_i}$ and hence $x=z+v\in \mathring{\DC_i}+\DC_i^0=\mathring{\DC_i}$.
\end{proof}

Let $J\subset I$ be non-empty.  Let $z \in \bigwedge_{j\in J} \cF_j$ and  $\overline{\tilde{\cP}}=\overline{\tilde{\cP}_J(z)}=(z+V_J)\cap\bigcap_{j\in J} \DC_j$.  For $j\in J$ define $f_j\in (V_J)^*$ by $f_j(x)=\langle x\mid \qp_j\rangle$, for $x\in V_J$. Then $(f_j)_{j\in J}$ is free. Let $(a_j)_{j\in J}\subset V_J$ be a dual set of $(f_j)_{j\in J}$, i.e., a free set such that $f_i(a_j)=\delta_{i,j}$, for all $i,j\in J$. Then $\overline{\tilde{\cP}}=z-\sum_{j\in J} \R_{\geq 0} a_j$. For $J'\subset J$, set $\overline{\tilde{\cF}_{J'}}=z-\sum_{j\in J'} \R_{\geq 0} a_j$. The closed faces of $\overline{\tilde{\cP}}$ are the $\overline{\tilde{\cF}_{J'}}$, for $J'\subset J$. In particular, the closed facets of $\overline{\tilde{\cP}}$ are the $\overline{\tilde{\cF}_j}$, $j\in J$, where $\tilde{\cF}_j=\tilde{\cF}_{J\setminus \{j\}}$, for $j\in J$. Let $j\in J$. 
The  support  of $\tilde{\cF}_j$ is $\tilde{H}_j=\{x\in (z+V_J)\mid \langle x-z\mid\qp_j\rangle= 0\}=H_j\cap (z+V_J),$ where $H_j$ is the frontier of $\DC_j$.    For $x\in \overline{\tilde{\cF}_j}$, we have $(x+\R_{\geq 0} \qp_j)\cap \overline{\tilde{\cF_j}}=\{x\}$ and $\qp_j$ is orthogonal to the direction of $\overline{\tilde{\cF}_j}$, so that $\R_{\geq 0}\qp_j=\delta_{\tilde{\cF}_j}$ for the notation of subsection~\ref{ss_preliminary_polytopes} applied to  the polyhedron $\overline{\tilde{\cP}}\subset z+V_J$. Therefore by Proposition~\ref{p_description_fibers_pi}, we have: 
\begin{equation}\label{e_fibers_pi_Ptilde}
\forall J'\subset J, \forall a\in \tilde{\cF}_{J'}, \pi_{\overline{\tilde{\cP}}}^{-1}(\{a\})=a+\bigoplus_{j\in J\setminus (J')} \R_{\geq 0}\qp_j,
\end{equation}
where $\pi_{\overline{\tilde{\cP}}}:z+V_J\rightarrow \overline{\tilde{\cP}}$ is the orthogonal projection.

\begin{Lemma}\label{l_equality_polytopes_chamber}
Let $J\subset I$. Let $z\in\bigwedge_{j\in J} \cF_j\cap \overline{C^v_f}$. Let $\overline{\tilde{\cP}}=(z+V_J)\cap \bigcap_{j\in J} \DC_j$. Then $\overline{\tilde{\cP}}\cap \overline{C^v_f}=\overline{\cP^0}\cap (z+V_J)\cap \overline{C^v_f}$.
\end{Lemma}

\begin{proof}
We have  $\overline{\cP^0}\subset \bigcap_{i\in I} \DC_i\subset \bigcap_{i\in J} \DC_i$, and hence  $\overline{\tilde{\cP}}\cap \overline{C^v_f}\supset \overline{\cP^0}\cap (z+V_J)\cap \overline{C^v_f}$. Suppose that $\overline{\tilde{\cP}}\cap \overline{C^v_f}\supsetneq \overline{\cP^0}\cap (z+V_J)\cap \overline{C^v_f}$. Let $x\in (\overline{\tilde{\cP}}\cap \overline{C^v_f})\setminus \overline{\cP^0}$. Let $\gamma:[0,1]\rightarrow [x,z]$ be the affine parametrization such that $\gamma(0)=x$ and $\gamma(1)=z$. Let $t=\min\{s\in [0,1]\mid \gamma(s)\in \overline{\cP^0}\}$. Then $t>0$ and $\gamma(t)\in \Fr(\cP^0)\cap \overline{C^v_f}$.

 Let $F^v\in \sF^0$ be such that $\gamma(t)\in \cF_{F^v}$.  By Lemma~\ref{l_projection_contained_face}, $F^v\subset \overline{C^v_f}$. Let  $J'=\{i\in I\mid \qp_i\in \overline{F^v}\}$. Then by Proposition~\ref{p_bourbaki_vectorial_faces}, we have $F^v=\bigoplus_{j\in J'} \R_{>0}\qp_j$ and by Proposition~\ref{p_description_faces_P0}, we have $\cF_{F^v}=\bigwedge_{j\in J'}\cF_j=\Int_r(\bigcap_{j\in J'}\overline{\cF}_j)$. By Lemma~\ref{l_faces_affine_subspace}, we have $\gamma(t)\notin \overline{\cF_{j}}$ for all $j\in I\setminus J$ and thus  $J'\subset J$.

By Corollary~\ref{c_faces_containing_point}, the set of facets $\cF_{F_1^v}$ such that $\gamma(t)\in \overline{\cF_{F_1^v}}$ is
\[\{\cF_{F_1^v}\mid F_1^v\subset \overline{F^v}\}=\{\cF_{\R_{>0}\qp_j}\mid j\in J'\}=\{\cF_j\mid j\in J'\}.\] 
Therefore by Lemma~\ref{l_local_description_polytope}, there exists $r>0$ such that $B(\gamma(t),r)\cap \overline{\cP^0}=B(\gamma(t),r)\cap \bigcap_{j\in J'}\DC_j$.  By assumption, $z\in \bigcap_{j\in J'}\DC_j$ and thus $\gamma(t')\in \bigcap_{j\in J'}\DC_j$, for all $t'\in [0,1]$. This proves that for $t'\in [0,t]$ close enough to $t$, $\gamma(t')\in \overline{\cP^0}$, which contradicts our assumption.
\end{proof}

\begin{Lemma}\label{l_preservation_faces_projection}
Let $J\subset I$. Let $z\in \overline{C^v_f}\cap \bigwedge_{j\in J} \cF_j$. Let $\overline{\tilde{\cP}}=\overline{\tilde{\cP^0_J}(z)}=(z+V_J)\cap \bigcap_{j\in J} \DC_j$. Let $\pi_{\overline{\tilde{\cP}}}:z+V_J\rightarrow \overline{\tilde{\cP}}$ be the orthogonal projection on $\overline{\tilde{\cP}}$.  Then for all $F_1^v\in \sF^0$ such that $F_1^v\subset \overline{C^v_f}$, we have $\pi_{\overline{\tilde{\cP}}}\left((z+V_J)\cap F_1^v\right)\subset F_1^v$.  
\end{Lemma}

\begin{proof}
 Let $a\in \Fr(\tilde{\cP})$,  where we regard $\tilde{\cP}^0$ as a subset of $z+V_J$. In the notation of Eq. \eqref{e_fibers_pi_Ptilde}, let us write:
 \[J_1=\{j\in J\mid a\in \overline{\tilde{\cF_j}}\}.\] We first assume $a\notin \overline{C^v_f}$ and we prove $\pi_{\overline{\tilde{\cP}}}^{-1}(\{a\})\cap \overline{C^v_f}=\emptyset$. We  have $]a,z]\cap \overline{C^v_f}=[b,z]$, for some $b\in \overline{C^v_f}$. If for all $i\in I$ such that $\alpha_i(a)<0$, we had $\alpha_i(b)>0$, then we would have $b-\epsilon (b-a)\in \overline{C^v_f}$ for $\epsilon\in \R_{>0}$ small enough. Therefore there exists $i\in I$ such that $\alpha_i(a)<0$ and $\alpha_i(b)=0$.   We have $b\in \overline{C^v_f}\cap \overline{\tilde{\cP}}$ and thus $b\in \overline{\cP^0}$, by Lemma~\ref{l_equality_polytopes_chamber}. Let $j\in J_1$. Then $b\in \overline{\tilde{\cF_j}}\subset H_j$ , since $z\in \bigwedge_{j\in J} \cF_j$.. Therefore $b\in H_j\cap \overline{\cP^0}=\overline{\cF_j}=\overline{\cF}_{\R_{>0}\qp_j}$. Using Lemma~\ref{l_positivity_roots_faces}, we deduce $\alpha_j(b)>0$ for all $j \in J_1$. In particular, $i\notin J_1$. By Eq. \eqref{e_fibers_pi_Ptilde}, $\pi_{\overline{\tilde{\cP}}}^{-1}(\{a\})=a+\sum_{j\in J_1} \R_{\geq 0} \qp_j$ and thus $\alpha_i(\pi^{-1}_{\overline{\tilde{\cP}}}(\{a\}))=\{\alpha_i(a)\}\subset \R_{<0}$. Consequently $\pi_{\overline{\tilde{\cP}}}^{-1}(\{a\})\cap \overline{C^v_f}=\emptyset$. Therefore \[\pi_{\overline{\tilde{\cP}}}(\overline{C^v_f}\cap (z+V_J))\subset \overline{C^v_f}.\]

Let now $a\in \Fr(\overline{\tilde{\cP}})\cap \overline{C^v_f}$.  Then by Lemma~\ref{l_equality_polytopes_chamber}, $\{j\in I\mid a\in \overline{\cF_j}\}\supset J_1$. Therefore by Eq. \eqref{e_fibers_pi_Ptilde} and Corollary~\ref{c_fibers_pi}, we have $\pi_{\overline{\tilde{\cP}}}^{-1}(\{a\})=a+\sum_{j\in J_1} \R_{\geq 0} \qp_j\subset \pi_{\overline{\cP^0}}^{-1}(\{a\})$. Using Lemma~\ref{l_projection_contained_face}, we deduce $\pi_{\overline{\tilde{\cP}}}^{-1}(\{a\})\subset F^v$, if $F^v$ is the face of $(\A,\cH_\Psi^0)$ containing $a$. Consequently, if  $x\in \overline{C^v_f}\cap (z+V_J)$, then $x$ and $\pi_{\overline{\tilde{\cP}}}(x)$ belong to the same face of $(\A,\cH_\Psi^0)$. 
\end{proof}

\begin{Lemma}\label{l_Fv(z,J)}
Let $z\in \A$ and $J\subset I$. Set $V_J=\bigoplus_{j\in J} \R \qp_j$. Write $z=\sum_{i\in I} z_i \qp_i$, with $(z_i)\in \R^I$.  Then:
$$(z+V_J)\cap \overline{C^v_f}=\left\{ \begin{aligned}&\emptyset &\text{ if } \exists i\in I\setminus J\mid z_i<0\\
&\sum_{i\in I\setminus J} z_i\qp_i+\sum_{j\in J}\R_{\geq 0}\qp_j &\text{otherwise}\end{aligned}\right.$$

If $z_i\geq 0$ for all $i\in I\setminus J$, set $F^v(z,J)=\sum_{i\in I\setminus J\mid z_i>0}\R_{>0} \qp_i+\sum_{j\in J}\R_{>0} \qp_j$. Then: 
\begin{align*}
\overline{F^v(z,J)}&\supset (z+V_J)\cap \overline{C^v_f}, \\(z+V_J)\cap \overline{F^v(z,J)}&=\sum_{i\in I\setminus J}z_i\qp_i+\sum_{j\in J}\R_{\geq 0}\qp_j=(z+V_J)\cap \overline{C^v_f},\\
(z+V_J)\cap F^v(z,J)&=\sum_{i\in I\setminus J}z_i\qp_i+\sum_{j\in J}\R_{> 0}\qp_j.
\end{align*}
\end{Lemma}

\begin{Lemma}\label{l_projection_preserving_affine_space}
Let $J\subset I$. Let $z\in \overline{C^v_f}\cap \bigwedge_{j\in J} \cF_j$. Then for all $y\in (z+V_J)\cap \overline{C^v_f}$, we have $\pi(y)\in (z+V_J)$. 
\end{Lemma}

\begin{proof}
Set $\overline{\tilde{\cP}}=(z+V_J)\cap \bigcap_{j\in J} \DC_j$.  Let $F^v=F^v(z,J)$, with the notation of Lemma~\ref{l_Fv(z,J)}. Let $x\in F^v\cap (z+V_J)$. By Lemma~\ref{l_preservation_faces_projection}, we have $\pi_{\overline{\tilde{\cP}}}(x)\in F^v$. By Lemma~\ref{l_Fv(z,J)}, we can choose an open neighbourhood $U$ of $\pi_{\tilde\cP^0}(x)$ in  $z+V_J$ such that $U\subset F^v$. Then for all $y\in U\cap \overline{\tilde{\cP}}$, $d(x,y)\geq d(x,\pi_{\overline{\tilde{\cP}}}(x))$. Let $V^J=\bigoplus_{i\in I\setminus J} \R \alpha_i^\vee$. Then $V^J$ is orthogonal to $V_J$ by Lemma~\ref{l_scalar_product_coroots_fundamental_weights} and  $U+V^J$ is an open subset of $\A$. 

Let $y\in (U+V^J)\cap \overline{\cP^0}$. Write $y=u+v$, with  $u\in U$ and $v\in V^J$.   Then  $\langle u-\bb\mid \qp_j\rangle=\langle u+v-\bb\mid \qp_j\rangle=\langle y-\bb\mid \qp_j\rangle\leq 0$, for all $j\in J$ and thus $u\in U\cap \overline{\tilde{\cP}}$. Moreover, \[d(x,y)^2=|y-x|^2=|u+v-x|^2=|u-x|^2+|v|^2\geq |u-x|^2\geq d(x,\pi_{\overline{\tilde{\cP}}}(x))^2.\]

Therefore the restriction to $\overline{\cP^0}$ of $y\mapsto d(x,y)$ admits a local minimum at $\pi_{\overline{\tilde{\cP}}}(x)$. As this map is convex (by \cite[B 1.3 (c)]{hiriart2012fundamentals} for example), this minimum is global and thus $\pi_{\overline{\tilde{\cP}}}(x)=\pi_{\overline{\cP^0}}(x)$. Thus we proved: \[\forall x\in F^v\cap (z+V_J),\pi(x)\in (z+V_J).\]

Let now $x\in \overline{C^v_f}\cap (z+V_J).$ Then, by Lemma~\ref{l_Fv(z,J)}, there exists $(x_n)\in \left((z+V_J)\cap F^v\right)^\N$ such that $x_n\rightarrow x$. Then $\pi(x_n)\rightarrow \pi(x)$ (by \cite[A. (3.1.6)]{hiriart2012fundamentals}. As $\pi(x_n)\in z+V_J$ for all $n\in \N$, we deduce $\pi(x)\in z+V_J$, which completes the proof of the lemma.
\end{proof}

\begin{Proposition}\label{p_projection_preserving_affine_spaces}
 Let $F^v$ be a face of $(\A,\cH^0_\Psi)$ and $z\in \cF_{F^v}$. Let $F_1^v$ be a  face dominating $z$. In particular, $F_1^v$ dominates $F^v$ by Proposition~\ref{p_description_faces_P0}. Let $\pi$ be the orthogonal projection on $\overline{\cP^0}$. Then: \begin{enumerate}
\item $\pi\left((z+\mathrm{vect}(F^v))\cap \overline{F_1^v}\right)\subset (z+\mathrm{vect}(F^v))\cap \overline{F_1^v}.$

\item Let $y\in (z+\mathrm{vect}(F^v))\cap \overline{F_1^v}$ and let $G^v\in \sF^0$ be such that $\pi(y)\in \cF_{G^v}$. Then $G^v\subset \overline{F^v}$.
\end{enumerate}
\end{Proposition}

\begin{proof}
1) Up to changing the choice of the fundamental chamber $C^v_f$, we may assume $F_1^v\subset \overline{C^v_f}$. Let $J=\{i\in I\mid \qp_i\in \overline{F^v}\}$. Then by Proposition~\ref{p_bourbaki_vectorial_faces}, $F^v=\bigoplus_{j\in J}\R_{>0}\qp_j$ and by Proposition~\ref{p_description_faces_P0}, $z\in \bigwedge_{j\in J}\cF_j$.  We have $V_J=\bigoplus_{j\in J}\R\qp_j$ and the result follows from Lemma~\ref{l_projection_preserving_affine_space} and Lemma~\ref{l_projection_contained_face}.

2) Let $y\in(z+\mathrm{vect}(F^v))\cap \overline{F_1^v}$. Let $G^v\in \sF^0$ be such that $\pi(y)\in \cF_{G^v}$.  By  Lemma~\ref{l_projection_contained_face}, $G^v\subset \overline{F_1^v}$. In particular, $G^v\subset \overline{C^v_f}$, and hence, by Proposition~\ref{p_bourbaki_vectorial_faces}, $G^v=\sum_{j\in I\mid \qp_j\in \overline{G^v}} \R_{>0}\qp_j$. Let now $i\in I\setminus J$. By  1) and Lemma~\ref{l_faces_affine_subspace}, $\pi(y)\notin \overline{\cF_{\R_{>0}\qp_i}}$. Using Corollary~\ref{c_faces_containing_point} we deduce that $\R_{>0}\qp_i\not\subset \overline{G^v}$. This being true for all $i\in I\setminus J$, we deduce that $G^v\subset \sum_{j\in J}\R_{\geq 0} \qp_j=\overline{F^v}$.  
\end{proof}

 \section{Tessellation of $\A$ by the affine weight polytopes}\label{s_tessellation_A}

Let $\A$ be a finite dimensional vector space and let $\Phi$ be a finite root system in $\A$. In this section, we study polytopes associated with the affine poly-simplicial complex defined from $\Phi$. 

In subsection~\ref{ss_affine_apartment}, we recall and prove general facts on the   affine poly-simplicial complex associated with $\Phi$. 

In subsection~\ref{ss_affine_weight_polytopes}, we define and study the faces of the weight polytopes.

In subsection~\ref{ss_tessellation_A}, we prove that the weight polytopes tessellate $\A$.

 \subsection{Affine apartment}\label{ss_affine_apartment}
 
The main reference for this section is \cite[V]{bourbaki1981elements}. Let $\Phi$ be a (finite) root system in  $\A$.  Let  $\xi:\A\rightarrow \R$ be the constant map equal to $1$ and $\Phi^{\mathrm{aff}}=\Phi+\Z\xi$\index{p@$\Phi^{\mathrm{aff}}$}. Then by \cite[(2.1) Proposition]{macdonald1971affine}, $\Phi^{\mathrm{aff}}$ is \textbf{an affine root system} in the sense of \cite[2]{macdonald1971affine}. Let $\cH=\cH_\Phi=\{\alpha^{-1}(\{k\})\mid \alpha \in \Phi,\ k\in \Z\}$ which equals $\{\underline{\alpha}^{-1}(\{0\})\mid \underline{\alpha}\in \Phi^{\mathrm{aff}}\}$\index{h@$\cH$}. This is a locally finite hyperplane arrangement of $\A$, which means that  every bounded subset $E$ of $\A$ meets finitely many hyperplanes. An element of $\cH$ is called a \textbf{wall} or an \textbf{affine wall.} For $x,y\in \A$, we write $x\sim_\cH y$ if for every $H\in \cH$, either ($x,y\in H$) or ($x$ and $y$ are strictly on the same side of $H$). This is an equivalence relation on $\A$. Its classes are called the \textbf{faces} of $(\A,\cH)$ or simply the faces of $\A$. A \textbf{vertex} of $\A$ is an element $x$ of $\A$ such that $\{x\}$ is a face of $(\A,\cH)$. If $F$ is a face of $(\A,\cH)$, then:\begin{equation}\label{e_Property_faces_roots}
\forall \alpha\in \Phi, \exists n\in \Z\mid \alpha(F)\subset ]n,n+1[\text{ or } \alpha(F)=\{n\}.
\end{equation}

An \textbf{alcove} (or \textbf{open alcove}) of $(\A,\cH)$ is a connected component of $\A\setminus \bigcup_{H\in \cH} H$. The alcoves are actually the faces which are not contained in an element of $\cH$. A \textbf{closed alcove} is a set of the form $\overline{C}$, for some alcove $C$ of $\A$. If $\overline{C}$ is a closed alcove of $\A$, its set of vertices, denoted $\ve(\overline{C})$, is the set of vertices of $\A$ which belong to $\overline{C}$.

We denote by $\sF=\sF(\cH)$\index{f@$\sF=\sF(\cH)$} the set of faces of $(\A,\cH)$. If $F\in \sF$, we denote by $\sF_F=\sF_F(\cH)$ the set of faces dominating $F$ (i.e., the set of faces whose closure contains $F$). When $F=\{\lambda\}$, for some vertex $\lambda$, we write $\sF_\lambda(\cH)$ instead of $\sF_{\{\lambda\}}(\cH)$\index{F@$\sF_F=\sF_F(\cH)$, $\sF_\lambda=\sF_\lambda(\cH)$}.
We denote by $\Alc(\cH)$\index{a@$\Alc(\cH)$} the set of alcoves of $(\A,\cH)$ and if $F\in \sF(\cH)$, we denote by $\Alc_F(\cH)$\index{a@$\Alc(\cH),\Alc_F(\cH)$} the set of alcoves of $(\A,\cH)$ dominating $F$. The \textbf{affine poly-simplicial complex associated with $\Phi$} is $\sF(\cH)$. When $\Phi$ is irreducible, it is a simplicial complex.

 \begin{figure}[h]
 \centering
 \includegraphics[scale=0.2]{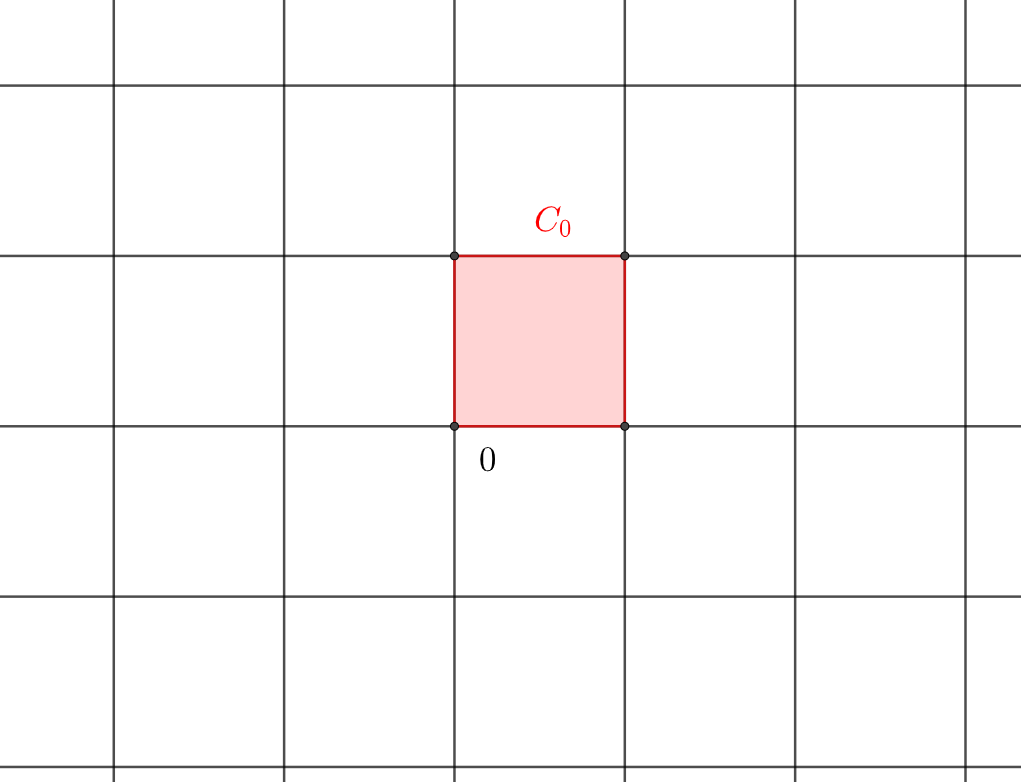}
 \caption{Affine appartment of type $\mathrm{A}_1\times \mathrm{A}_1$. The alcoves are the squares.}
\end{figure}

 \begin{figure}[h]
 \centering
 \includegraphics[scale=0.2]{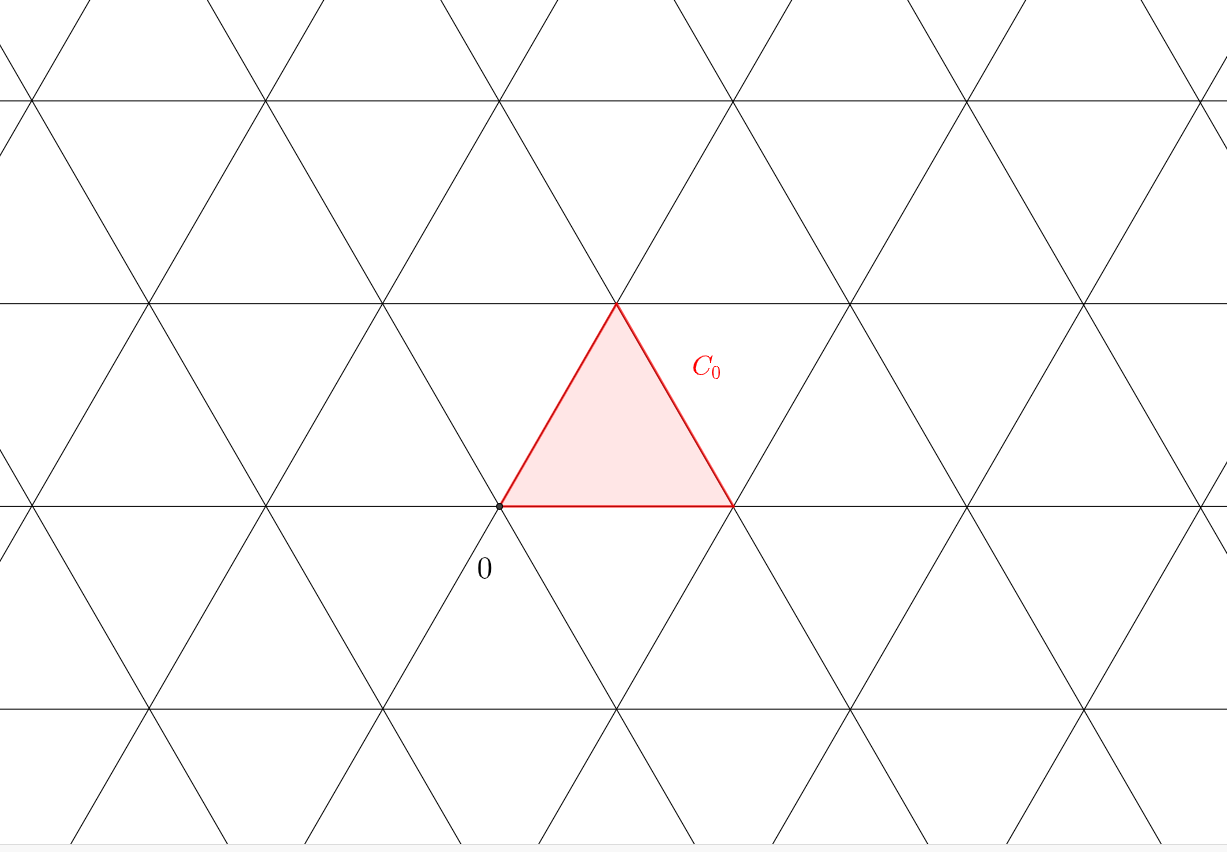}
 \caption{Affine appartment of type $\mathrm{A}_2$. The alcoves are the equilateral triangles.}
\end{figure}

 \begin{figure}[h]
 \centering
 \includegraphics[scale=0.2]{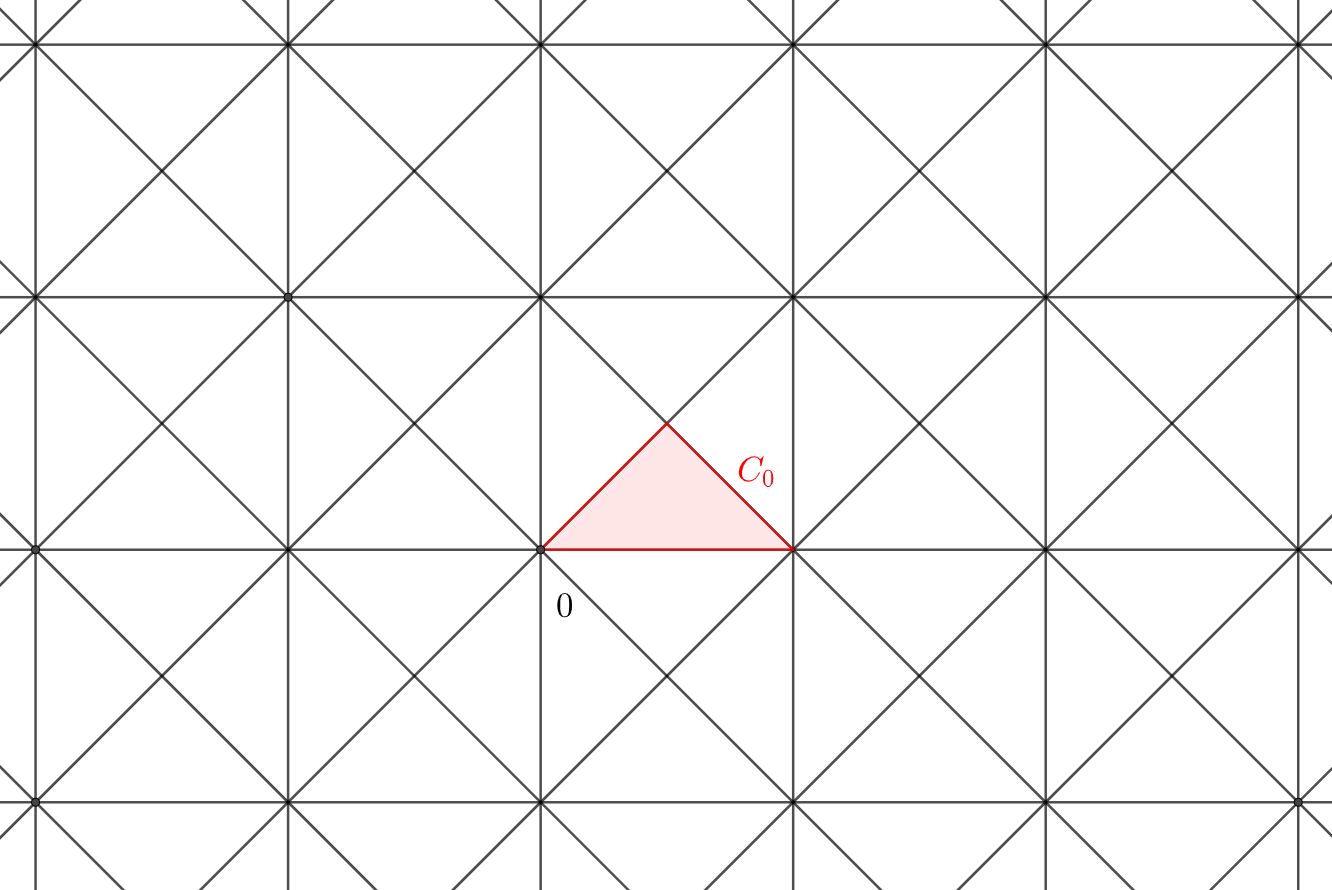}
 \caption{Affine appartment of type $\mathrm{B}_2$. The alcoves are the right triangles.}
\end{figure}

 \begin{figure}[h]
 \centering
 \includegraphics[scale=0.2]{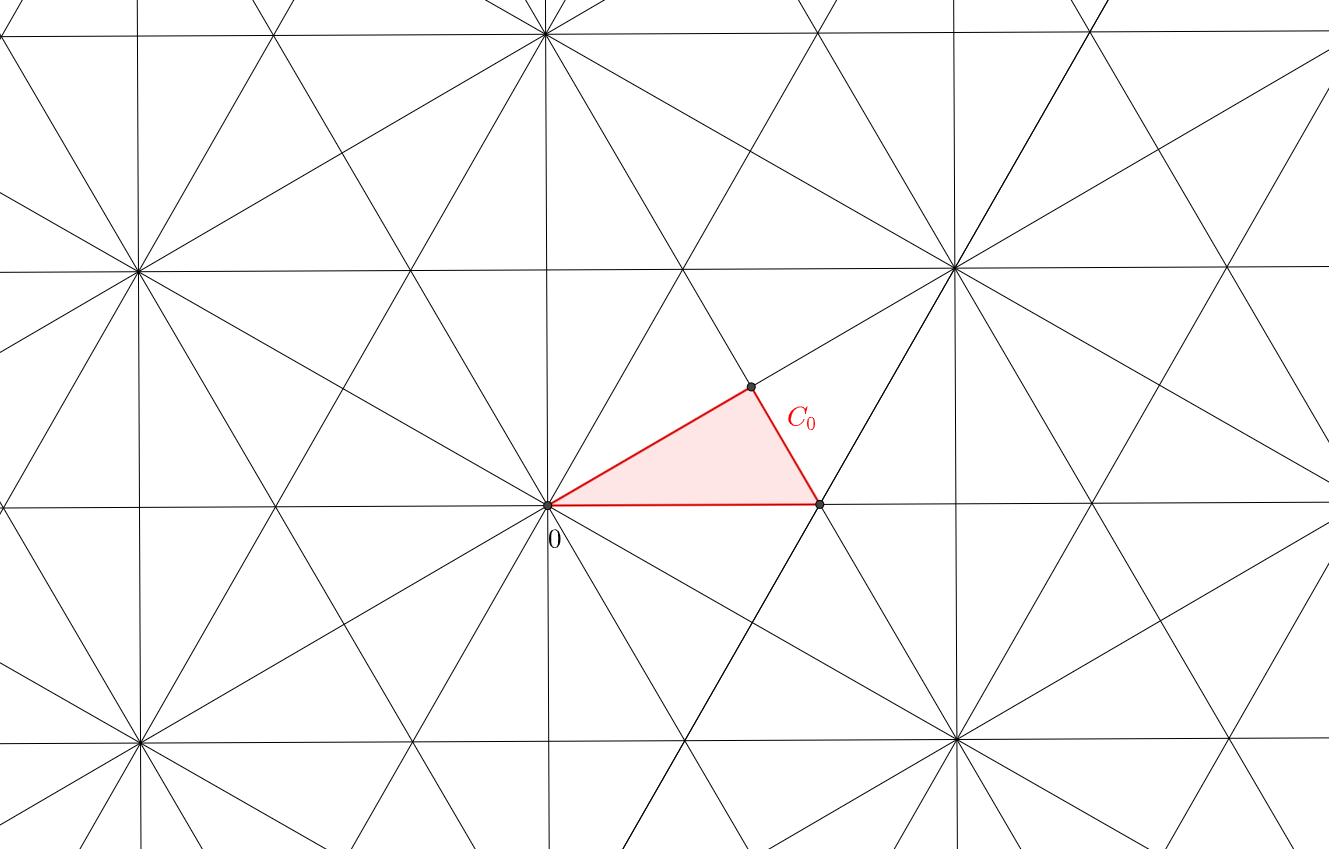}
 \caption{Affine appartment of type $\mathrm{G}_2$. The alcoves are the right triangles.}
\end{figure}

\subsubsection{Enclosure and closed faces of $(\A,\cH)$}

\begin{Definition}
A \textbf{half-apartment} of $\A$ is a closed half-space of $\A$ delimited by an element of $\cH$. If $\alpha\in \Phi$ and $k\in \Z$, one sets $D_{\alpha,k}=\{x\in \A\mid \alpha(x)+k\geq 0\}$\index{d@$D_{\alpha,k}$}

Let $\Omega\subset \A$. Its \textbf{enclosure} $\cl(\Omega)$\index{c@$\cl$} is the intersection of all the half-apartments containing $\Omega$.

If $\Omega$ is a subset of $\A$ and $\lambda\in \A$, then the \textbf{germ of $\Omega$ at $\lambda$} is the set $\germ_\lambda(\Omega)$\index{g@$\germ$} of subsets of $\A$ of the form $\Omega\cap E$ such that $E$ is an open neighbourhood of $\lambda$. A subset $\Omega'$ of $\A$ is said to \textbf{contain} the germ $\germ_\lambda(\Omega)$ if it contains $\Omega\cap E$ for some open neighbourhood $E$ of $\lambda$. The enclosure of $\germ_\lambda(\Omega)$ is the intersection of the enclosures of the $\Omega\cap E$ when $E$ describes all the open neighbourhoods of $\lambda$.
    
\end{Definition}

A \textbf{closed face of $(\A,\cH)$} is a set of the form $\overline{F}$, for some $F\in \sF(\cH)$.

\begin{Lemma}\label{l_enclosure_faces}
\begin{enumerate}
    \item Let $F\in \sF$. For $\underline{\alpha}\in \Phi^{\mathrm{aff}}$, set $E_{\underline{\alpha}}=\{0\}$ if $\underline{\alpha}(F)=\{0\}$, $E_{\underline{\alpha}}=\R_{>0}$ if $\underline{\alpha}(F)>0$ and $E_{\underline{\alpha}}=\R_{<0}$ if $\underline{\alpha}(F)<0$. Then $F=\bigcap_{\underline{\alpha}\in \Phi^{\mathrm{aff}}}\underline{\alpha}^{-1}(E_{\underline{\alpha}})$ and $\overline{F}=\bigcap_{\underline{\alpha}\in \Phi^{\mathrm{aff}}}\underline{\alpha}^{-1}(\overline{E_{\underline{\alpha}}})$. In particular, $\overline{F}$ is enclosed.

    \item Let $F\in \sF$ and $x\in F$. Then $\cl(x)=\overline{F}$. 
\end{enumerate}
\end{Lemma}

\begin{proof}
  1)  The first equality follows from the definition of faces. As the elements of $\Phi^{\mathrm{aff}}$ are continuous maps from $\A$ to $\R$, we have $\overline{F}\subset\bigcap_{\underline{\alpha}\in \Phi^{\mathrm{aff}}}\underline{\alpha}^{-1}(\overline{E_{\underline{\alpha}}})$. Let $x\in F$ and $y\in \bigcap_{\underline{\alpha}\in \Phi^{\mathrm{aff}}}\underline{\alpha}^{-1}(\overline{E_{\underline{\alpha}}})$. Then for $z\in ]x,y[$ and $\underline{\alpha}\in \Phi^{\mathrm{aff}}$, we have $\underline{\alpha}(z)\in \overline{E_{\underline{\alpha}}}+E_{\underline{\alpha}}=E_{\underline{\alpha}}$ and hence $[x,y[\subset F$, which proves that $y\in \overline{F}$ and concludes the proof of the second equality. The fact that $\overline{F}$ is enclosed then follows from the fact that if $\underline{\alpha}\in \Phi^{\mathrm{aff}}$, then $\underline{\alpha}^{-1}(\{0\})=\underline{\alpha}^{-1}(\R_{\geq 0})\cap \underline{\alpha}^{-1}(\R_{\leq 0})$.

  We now prove 2). We have $\cl(x)\subset \cl(F)=\overline{F}$. By definition of the faces, every half-apartment containing $x$ contains $F$ and thus it contains $\overline{F}$. Therefore $\cl(x)\supset\overline{F}$, which completes the proof of the lemma.
\end{proof}

\begin{Proposition}\label{p_characterization_faces}
    Let $E$ be a non-empty subset of $\A$. Then $E$ is a closed face of $(\A,\cH)$  if and only if the following two conditions are satisfied: \begin{enumerate}
        \item $E$ is enclosed,

        \item for every $\underline{\alpha}\in \Phi^{\mathrm{aff}}$, we have either $\underline{\alpha}(E)\subset \R_{\geq 0}$ or $\underline{\alpha}(E)\subset \R_{\leq 0}$. 
    \end{enumerate}
\end{Proposition}

\begin{proof}
    Let $F$ be a face of $\sF(\cH)$. Then by Lemma~\ref{l_enclosure_faces}, $\overline{F}$ is enclosed. By definition of the faces, if  $\underline{\alpha}\in \Phi^{\mathrm{aff}}$, we have $\underline{\alpha}(F)\subset \R_{>0}$, $\underline{\alpha}(F)=\{0\}$ or $\underline{\alpha}(F)\subset\R_{<0}$. As $\underline{\alpha}(\overline{F})\subset \overline{\underline{\alpha}(F)}$ we deduce that $\overline{F}$ satisfies 1. and 2.

    Let now $E$ be a non-empty subset satisfying 1. and 2. Then $E$ is convex and closed. Let $F=\In_r(E)$. Then we have $E=\overline{F}$.  Let $S$ be the support of $E$ and $x,y\in F$. Let $\underline{\alpha}\in \Phi^{\mathrm{aff}}$. If $\underline{\alpha}$ is constant on $S$, then $\underline{\alpha}(x)=\underline{\alpha}(y)$. Assume now that $\underline{\alpha}$ is non-constant on $S$. Let $U_x, U_y$ be open subsets of $S$ containing $x$ and $y$ respectively. Then  by assumption $\underline{\alpha}(U_x)\underline{\alpha}(U_y)\geq 0$ and hence $\underline{\alpha}(x)\underline{\alpha}(y)>0$. Consequently, $x \sim y$, where $\sim$ is the equivalence relation that we used to define the faces of $(\A,\cH)$. Consequently there exists $F_1\in \sF$ such that $F\subset F_1$. 

    Using 1., we write $\overline{F}=\bigcap_{\underline{\alpha}\in \cE}\underline{\alpha}^{-1}(\R_{\geq 0})$, where $\cE\subset \Phi^{\mathrm{aff}}$. Let $x\in F$. Let $\underline{\alpha}\in \cE$. Then $\underline{\alpha}(x)\geq 0$ and as $\underline{\alpha}$ has constant sign on $F_1$, we have $\underline{\alpha}(F_1)\geq 0$. Consequently $F_1\subset \overline{F}$ and thus $\overline{F_1}\subset \overline{F}$.  As $F\subset F_1$, we deduce $\overline{F}=\overline{F_1}=E$, and $E$ is a closed face of $(\A,\cH)$.
    \end{proof}
\color{black}

\begin{Corollary}\label{c_intersection_faces}
    Let $\overline{F_1}, \overline{F_2}$ be two closed faces of $(\A,\cH)$. Then either $\overline{F_1}\cap \overline{F_2}$ is empty or $\overline{F_1}\cap \overline{F_2}$ is a closed face of $(\A,\cH)$. 
\end{Corollary}

Let $C\in \Alc(\cH)$. A \textbf{wall of $C$} is a hyperplane $H$ such that  $\supp(\overline{C}\cap H)=H$. A \textbf{panel of $C$} is an element of $\sF$ contained in $\overline{C}$ and whose support is a wall of $C$.

One sets \[\Sigma^{\mathrm{aff}}_C:=\{\underline{\alpha}\in\Phi^{\mathrm{aff}}\mid \underline{\alpha}(C)>0,\underline{\alpha}^{-1}(\{0\})\text{ is a wall of }C\}.\]\index{s@$\Sigma^{\mathrm{aff}}_C$}

Assume that $\Phi$ is irreducible. Then $\Phi^{\mathrm{aff}}$ is an irreducible affine root system by \cite[(4.7)]{macdonald1971affine} and $\Sigma^{\mathrm{aff}}_C$
is a basis of $\Phi^{\mathrm{aff}}$ in the sense of \cite[4.6]{macdonald1971affine}. In particular, \begin{equation}\label{e_macdonald_4.6}
    \Phi^{\mathrm{aff}}\subset \left(\sum_{\underline{\alpha}\in \Sigma^{\mathrm{aff}}_C}\Z_{\geq 0}\underline{\alpha}\cup \sum_{\underline{\alpha}\in \Sigma^{\mathrm{aff}}_C}(-\Z_{\geq 0})\underline{\alpha}\right),
\end{equation}

by \cite[(4.6) Proposition]{macdonald1971affine}.

\begin{Lemma}\label{l_description_faces_intersections}
\begin{enumerate}
    \item Let $F\in \sF(\cH)$. Then there exists $C\in \Alc(\cH)$ such that $\overline{F}\subset\overline{C}$.  Let $\cM_F=\{\underline{\alpha}\in \Sigma^{\mathrm{aff}}_C\mid \underline{\alpha}(F)>0\}$. Then \begin{equation}\label{e_F_explicit}
       \overline{F}= \overline{C}\cap \bigcap_{\underline{\alpha}\in \Sigma^{\mathrm{aff}}_C\setminus \cM_F}\underline{\alpha}^{-1}(\{0\})=  \bigcap_{\underline{\alpha}\in \cM_F}\underline{\alpha}^{-1}(\R_{\geq 0})\cap \bigcap_{\underline{\alpha}\in \Sigma^{\mathrm{aff}}_C\setminus \cM_F}\underline{\alpha}^{-1}(\{0\}).
    \end{equation}

    \item Let $C\in \Alc(\cH)$. Let $\underline{\alpha}\in \Sigma^{\mathrm{aff}}_C$. Then $\overline{C}\cap \underline{\alpha}^{-1}(\{0\})$ is a closed panel of $C$ and all the closed panels of $C$ are of this form.

    \item Let $C\in \Alc(\cH)$. Then every   closed  face of $(\A,\cH)$ strictly contained in $\overline{C}$ can be written as an intersection of closed panels of $\overline{C}$.

\end{enumerate}

\end{Lemma}

\begin{proof}

  1)  Let $x\in F$. Then by \cite[V \S 1.4 Proposition 6]{bourbaki1981elements}, there exists $C\in \Alc(\cH)$ such that $x\in \overline{C}$. But then by Lemma~\ref{l_enclosure_faces}, $\cl(x)=\overline{F}\subset \overline{C}=\cl(C)$.  Assume for the moment that $\Phi$ is irreducible. With the same notation as in Lemma~\ref{l_enclosure_faces}, we have \begin{equation}\label{e_description_F_1}
      \overline{F}=\bigcap_{\underline{\beta}\in \Phi^{\mathrm{aff}}}\underline{\beta}^{-1}(E_{\underline{\beta}}).
  \end{equation}

  Let $\underline{\beta}\in \Phi^{\mathrm{aff}}$. Then by Eq. \eqref{e_macdonald_4.6}, we can write $\underline{\beta}=\epsilon\sum_{\underline{\alpha}\in \Sigma^{\mathrm{aff}}_C}n_{\underline{\alpha}} \underline{\alpha}$, with $\epsilon\in \{-,+\}$ and $(n_{\underline{\alpha}})\in (\Z_{\geq 0})^{\Sigma_C^{\mathrm{aff}}}$. We have $E_{\underline{\beta}}=\{0\}$ if and only if $n_{\underline{\alpha}}=0$, for all $\underline{\alpha}\in \cM_F$. Otherwise $E_{\underline{\beta}}=\epsilon \R_{\geq 0}$. In particular $\underline{\beta}^{-1}(E_{\underline{\beta}})\supset \bigcap_{\underline{\alpha}\in \Sigma^{\mathrm{aff}}_C}\underline{\alpha}^{-1}(E_{\underline{\alpha}})$. Therefore Eq.~\eqref{e_description_F_1} becomes 
  \begin{equation}\label{e_description_F_2}
      \overline{F}=\bigcap_{\underline{\alpha}\in \Sigma^{\mathrm{aff}}_C}\underline{\alpha}^{-1}(E_{\underline{\alpha}})=\bigcap_{\underline{\alpha}\in \cM_F}\underline{\alpha}^{-1}(\R_{\geq 0})\cap \bigcap_{\underline{\alpha}\in \Sigma^{\mathrm{aff}}_C\setminus \cM_F}\underline{\alpha}^{-1}(\{0\}).
  \end{equation}

When $F=C$, we have $\cM_C=\Sigma^{\mathrm{aff}}_C$, which proves Eq.~\eqref{e_F_explicit}, when $\Phi$ is irreducible.

 We now assume that $\Phi$ is reducible. Using the notation of Eq. \eqref{e_decomposition_Phi}, with $\Phi$ instead of $\Psi$, we write $\A=\bigoplus_{i=1}^\ell \A_i$ and $\Phi=\bigsqcup_{i=1}^\ell\Phi_i$, where for each $i\in \llbracket 1,\ell\rrbracket$, $\Phi_i$ is an irreducible root system in $\A_i$. If $i\in \llbracket 1,\ell\rrbracket$ and $\underline{\alpha}\in \Phi_i^{\mathrm{aff}}$, $\underline{\alpha}:\A_i\rightarrow \R$, we denote by $\tilde{\underline{\alpha}}$ the extension of $\underline{\alpha}$ to $\A$: $\tilde{\underline{\alpha}}((x_1,\ldots,x_i,\ldots,x_\ell))=\underline{\alpha}(x_i)$, for $(x_1,\ldots,x_\ell)\in \A$. Write $C=C_1\times\ldots \times C_\ell$ and $F=F_1\times \ldots \times F_\ell$. Then we have $\Sigma^{\mathrm{aff}}_C=\bigsqcup_{i=1}^\ell \{\tilde{\underline{\alpha}}\mid \underline{\alpha}\in \Sigma^{\mathrm{aff}}_{C_i}\}$ and $\cM_F=\bigsqcup_{i=1}^\ell \{\tilde{\underline{\alpha}}\mid \underline{\alpha}\in \cM_{F_i}\}$. For $i\in \llbracket 1,\ell\rrbracket$, we have:\[\bigcap_{\underline{\alpha}\in \Sigma^{\mathrm{aff}}_{C_i}\setminus \cM_{F_i}}\tilde{\underline{\alpha}}^{-1}(\{0\})\\
     =\A_1\times\ldots \times \A_{i-1}\times \bigcap_{\underline{\alpha}\in \Sigma^{\mathrm{aff}}_{C_i}\setminus \cM_{F_i}}{\underline{\alpha}}^{-1}(\{0\})\times \A_{i+1}\times \ldots \times \A_{\ell}.\]
  We thus deduce Eq.~\eqref{e_F_explicit} from the irreducible case. This proves statement 1.

2) By Proposition~\ref{p_characterization_faces}, if $\underline{\alpha}\in \Sigma^{\mathrm{aff}}_C$, then $\underline{\alpha}^{-1}(\{0\})\cap \overline{C}=\underline{\alpha}^{-1}(\R_{\geq 0})\cap \underline{\alpha}^{-1}(\R_{\leq 0})\cap \overline{C}$ is a closed face of $(\A,\cH)$. It is a panel by definition of $\Sigma^{\mathrm{aff}}_C$. By 1. and for dimension reasons, all the closed panels are of this form, which proves statement 2.

3) Let $F\in \sF$ be such that $F\subset \overline{C}$. With the same notation as above, we have $\overline{F}=\bigcap_{\underline{\alpha}\in \Sigma^{\mathrm{aff}}_C\setminus \cM_F} (\overline{C}\cap \underline{\alpha}^{-1}(\{0\}))$.

\end{proof}

\color{black}

Let $C\in \Alc(\cH)$. Then by Lemma~\ref{l_description_faces_intersections} (applied with $F=C$, $\cM_C=\Sigma^{\mathrm{aff}}_C$), $\overline{C}$ is a finite  intersection of half-apartments. Therefore it is a polytope. Moreover the closed panels of $C$ are exactly the facets of $\overline{C}$, when we regard it as a polytope, in the sense of \cite[Definition 1.3]{bruns2009polytopes}. By \cite[Theorem 1.10 a),b)]{bruns2009polytopes}, Lemma~\ref{l_description_faces_intersections} and Corollary~\ref{c_intersection_faces}, we deduce that the faces of $\overline{C}$ regarded as a polytope are exactly the closed faces of $(\A,\cH)$ which are contained in $\overline{C}$ (except the empty set, that we excluded from the set of faces).

\begin{Lemma}\label{l_enclosure_subset_alcove}
Let $C$ be an alcove of $(\A,\cH)$. Let $E\subset \overline{C}$ be non-empty. Then $\{F_1\in \sF\mid E\subset \overline{F_1}\}$ admits a minimal element $F$ (for the dominance order). We have $F\subset \overline{C}$, $\cl(E)=\overline{F}$ and if moreover, $E$ is convex, then $E$ meets $F$.    
\end{Lemma}

\begin{proof}
By Proposition~\ref{p_characterization_faces}, $\cl(E)$ is a closed face of $(\A,\cH)$: there exists $F\in \sF$ such that $\cl(E)=\overline{F}$. Moreover if $F_1\in \sF$ contains $E$, then $\overline{F_1}\supset\cl(E)=\overline{F}$.

Assume now that $E$ is convex. Let $\cE=\{F_1\in \sF\mid \overline{F_1}\subset \overline{F}, F_1\cap E\neq \emptyset\}$. Let $F_1,F_2\in \cE$, $x_1\in F_1$ and $x_2\in F_2$. Let $x_3=\frac{1}{2}(x_1+x_2)\in E\subset \overline{F}$. Let $F_3\in \sF$ be such that $x_3\in F_3$. Then $F_3\in \cE$ since $x_3\in\overline{F}$.

Moreover, $]x_1,x_2[\cap \overline{F_3}\supset \{x_3\}\neq \emptyset$. Let $C$ be an alcove dominating $F$. By Eq.~\eqref{e_F_explicit}, we have $\overline{F_3}=\bigcap_{\underline{\alpha}\in \cM_{F_3}}\underline{\alpha}^{-1}(\R_{\geq 0})\cap \bigcap_{\underline{\alpha}\in \Sigma^{\mathrm{aff}}_C\setminus \cM_{F_3}}\underline{\alpha}^{-1}(\{0\})$, with the notation of Lemma~\ref{l_description_faces_intersections}. Then if $\underline{\alpha}\in \Sigma^{\mathrm{aff}}_C\setminus \cM_{F_3}$, we have $\underline{\alpha}(x_3)=0$ and $\underline{\alpha}(x_1),\underline{\alpha}(x_2)\geq 0$ so that $\underline{\alpha}(x_1)=\underline{\alpha}(x_2)=0$. This proves that $[x_1,x_2]\subset \overline{F_3}$. Moreover, $\overline{F_3}$  is enclosed and thus by Lemma~\ref{l_enclosure_faces}, $\overline{F_3}\supset \overline{F_1}\cup \overline{F_2}$. Hence $\cE$ admits a maximal element $F_{\max}$ for the dominance order.

Let $x\in E$ and $F_x\in \sF(\cH)$ be such that $x\in F_x$. Then $F_x$ belongs to $\cE$ and thus $x\in F_x\subset \overline{F_{\max}}$. Hence $E\subset \overline{F_{\max}}$. By definition of $\overline{F}$, this implies that $\overline{F}\subset\overline{F_{\max}}$ and hence $\overline{F}=\overline{F_{\max}}$. Therefore $F$ meets $E$.
\end{proof}

\subsubsection{Properties of the affine Weyl group}

There exists a scalar product $\langle\cdot\mid \cdot\rangle$ on $\A$ which is invariant under the vectorial Weyl group $W^v_\Phi$ of $\Phi$. We equip $\A$ with such a scalar product. If $H\in \cH$, we denote by $s_H$\index{s@$s_H$} the orthogonal reflection of $\A$ fixing $H$.  Let $C\in \Alc(\cH)$. A \textbf{wall} of $C$ is an element of $\cH$ such that $\supp(\overline{C}\cap H)=H$. The \textbf{affine Weyl group of $(\A,\cH)$} is the group $W^{\mathrm{aff}}$\index{w@$W^{\mathrm{aff}}$} of affine maps on $\A$ generated by $s_H$, for $H\in \cH$. If $F\in \sF(\cH)$, we denote by $W^{\mathrm{aff}}_F$\index{w@$W^{\mathrm{aff}}_F$} the fixator of $F$ in $W^{\mathrm{aff}}$. Gathering the results of \cite[V \S 3.2, 3.3]{bourbaki1981elements}, we have the following theorem. We write it for the affine Weyl group and the faces of $(\A,\cH)$, but it is also true for the vectorial Weyl group and the vectorial faces, by the same reference. We refer to \cite{bjorner2005combinatorics} for the definition of Coxeter groups.

\begin{Theorem}\label{t_Weyl_group_Coxeter}
    \begin{enumerate}
        \item Let $C\in \Alc(\cH)$ and $S_C$ be the set of orthogonal reflections with respect to the walls of  $C$. Then $(W^{\mathrm{aff}},S_C)$ is a Coxeter group.

        \item Let $C\in \Alc(\cH)$. Then $\overline{C}$ is a fundamental domain for the action of $W^{\mathrm{aff}}$ on $\A$ and for every alcove $C'\in \Alc(\cH)$, there exists a unique $w\in W^{\mathrm{aff}}$ such that $w.C=C'$.

        \item Let $F\in \sF$ and $C\in \Alc_F(\cH)$. Then $W^{\mathrm{aff}}_F$ is the parabolic subgroup of $(W^{\mathrm{aff}},S_C)$ generated by the reflections with respect to the walls of $C$ containing $F$. Moreover, if $w\in W^{\mathrm{aff}}$,
        the following statements are equivalent: 
        \begin{enumerate}
        \item[a.] $w\in W_F^{\mathrm{aff}}$,
        \item[b.] $w$ pointwise fixes $\overline{F}$,
        \item[c.] $w.F$ meets $F$,
        \item[d.] $w$ fixes a point of $F$.
 \end{enumerate}
    \end{enumerate}
\end{Theorem}

The lemma below also has a vectorial version.

\begin{Lemma}\label{l_transitive_actions_element_fixing_face}
Let $F\in \sF(\cH)$ and $C\in \Alc_F(\cH)$. Then for every $F'\in \sF(\cH)$ such that $F'$ dominates $F$, there exists $w\in W^{\mathrm{aff}}_F$ such that $w.F'\subset \overline{C}$. In particular, $W^{\mathrm{aff}}_F$  acts transitively on $\Alc_F(\cH)$.

Equivalently, if $E$ is the union of the faces dominating $F$ and if $E'=E\cap \overline{C}$, then $E'$ is a fundamental domain for the action of  $W_F^{\mathrm{aff}}$ on $E$.      
\end{Lemma}

\begin{proof}
    Let $C'$ be an alcove of $(\A,\cH)$ dominating $F$. Write $C'=w.C$, with $w\in W^{\mathrm{aff}}$. Then $w^{-1}.F\subset \overline{C}$, so if $x\in w^{-1}.F\subset \overline{C}$, we have  $w.x\in \overline{C}$, which proves that $w.x=x$ by Theorem~\ref{t_Weyl_group_Coxeter}. Therefore $w$ fixes $w^{-1}.F$, which proves that $w$ fixes $F$. Now if $F_1$ is a face of $(\A,\cH)$ dominating $F$ and $C_1$ is an alcove dominating $F_1$, we can write $C_1=w_1.C$, with $w_1\in W^{\mathrm{aff}}$ fixing $F$. Then $w_1^{-1}.F_1\subset \overline{C}$, which proves the lemma. 
\end{proof}

\begin{Lemma}\label{l_comparison_fixators}
Let $C\in \Alc(\cH)$ and $F_1,F_2\in \sF(\cH)$ be such that $F_1\cup F_2\subset \overline{C}$. Then $W_{F_1}^{\mathrm{aff}}\subset W_{F_2}^{\mathrm{aff}}$ (resp. $W_{F_1}^{\mathrm{aff}}=W^{\mathrm{aff}}_{F_2}$) if and only if $F_2\subset \overline{F_1}$ (resp. $F_1=F_2$).
\end{Lemma}

\begin{proof}
By Theorem~\ref{t_Weyl_group_Coxeter}, if $S_C$ denotes the set of orthogonal reflections with respect to the walls of $C$, then $(W^{\mathrm{aff}},S_C)$ is a Coxeter group. Let $i\in \{1,2\}$. Set $\cM_{F_i}=\{\underline{\alpha}\in \Sigma^{\mathrm{aff}}_C\mid \underline{\alpha}(F_i)>0\}$. We have $W^{\mathrm{aff}}_{F_i}=\langle s_{\underline{\alpha}}\mid \underline{\alpha}\in \Sigma^{\mathrm{aff}}_C\setminus\cM_{F_i}\rangle$, by Theorem~\ref{t_Weyl_group_Coxeter}, where $s_{\underline{\alpha}}$ denotes the reflection of $W^{\mathrm{aff}}$ fixing $\underline{\alpha}^{-1}(\{0\})$, for $\underline{\alpha}\in \Phi^{\mathrm{aff}}$. By \cite{bjorner2005combinatorics}, if   $W^{\mathrm{aff}}_{F_1}\subset W^{\mathrm{aff}}_{F_2}$ (resp. $W^{\mathrm{aff}}_{F_1}=W^{\mathrm{aff}}_{F_2}$), then we have $\Sigma^{\mathrm{aff}}_C\setminus \cM_{F_1}\subset \Sigma^{\mathrm{aff}}_C\setminus \cM_{F_2}$ (resp. $\Sigma^{\mathrm{aff}}_C\setminus \cM_{F_1}=\Sigma^{\mathrm{aff}}_C\setminus \cM_{F_2}$).
 Using Lemma~\ref{l_description_faces_intersections} we deduce that $\overline{F_2}\subset \overline{F_1}$ (resp. $\overline{F_1}=\overline{F_2}$).

 Conversely, as $W^{\mathrm{aff}}$ acts by homeomorphisms on $\A$, we have $W^{\mathrm{aff}}_{F_i}=W^{\mathrm{aff}}_{\overline{F_i}}$ if $i\in \{1,2\}$ and thus if $F_2\subset \overline{F_1}$ (resp. $\overline{F_1}=\overline{F_2}$), we have $W^{\mathrm{aff}}_{F_2}\supset W_{F_1}^{\mathrm{aff}}$ (resp. $W^{\mathrm{aff}}_{F_1}=W^{\mathrm{aff}}_{F_2}$).
\end{proof}

\subsubsection{Properties of the faces and link between $\sF(\cH)$ and $\sF(\cH^0_\lambda)$}

For $x\in \A$, we set $\Phi_x=\{\alpha\in \Phi\mid \alpha(x)\in \Z\}$.\index{p@$\Phi_x$}

\begin{Lemma}\label{l_characterization_vertices_roots}
Let $x\in \A$. Then $x\in \ve(\A)$ if and only if $\bigcap_{\alpha\in \Phi_x}\ker(\alpha)=\{0\}$. 
\end{Lemma}

\begin{proof}
Let $E=\bigcap_{\alpha\in \Phi_x}\ker(\alpha)$. Let $y\in E$. Then for $t\in ]0,1[$ small enough, for every $\alpha\in \Phi$, if $\alpha(x)\notin \Z$, then $\alpha(x+ty)\in ]\lfloor \alpha(x)\rfloor,\lfloor \alpha(x)\rfloor+1[$ and thus $x\sim_\cH x+ty$. Therefore if $\{x\}$ is a face of $( \A,\cH)$, then $E=\{0\}$. Conversely, assume  $E=\{0\}$. Take $y\in \A\setminus \{x\}$. Then there exists $\alpha\in \Phi_x$ such that $\alpha(y)\neq \alpha(x)$.  Then $y\notin \alpha^{-1}(\{\alpha(x)\})$ and as $  \alpha^{-1}(\{\alpha(x)\})\in \cH$, $y\not\sim_H x$. Thus the class of $x$ is $\{x\}$. 
\end{proof}

Let $\lambda\in \ve(\A)$. Then by \cite[Chapitre VI, §1 Proposition 4]{bourbaki1981elements} and Lemma~\ref{l_characterization_vertices_roots},  $\Phi_\lambda$ is a root system in $\A^*$, with dual root system $\Phi_\lambda^\vee:=\{\alpha^\vee\mid \alpha\in \Phi_\lambda\}$. For $\alpha\in \Phi_\lambda$ and $x\in \R^*$, if $x\alpha\in \Phi_\lambda$, then $x\alpha\in \Phi$ and thus $x\in \{-1,1\}$, since $\Phi$ is reduced. Therefore $\Phi_\lambda$ is  a reduced root system in $\A^*$.   If $\alpha\in \Phi$, we denote by $s_\alpha\in W^v$ the associated reflection.

 Let $\cH_\lambda^0=\{\alpha^{-1 }(\{0\})\mid \alpha\in \Phi_\lambda\}$\index{h@$\cH_\lambda^0$}. The connected components of $\A \setminus \bigcup_{H \in \cH_\lambda^0} H$ are the (vectorial) chambers of $(\A,\Phi_\lambda)$. If $\lambda\in \ve(\A)$, then we denote by $\sF_\lambda(\cH)$ the set of faces of $(\A,\cH)$ dominating $\lambda$ and we denote by $\sF(\cH^0_\lambda)$ the set of faces of $(\A,\cH_\lambda^0)$.

\begin{Lemma}\label{l_open_segments_faces}
Let $\lambda\in \ve(\A)$ and $F_1,F_2\in \sF_\lambda(\cH)$. We assume that there exists a neighbourhood $\Omega$ of $\lambda$ such that $F_1\cap \Omega\subset \overline{F_2}$. Then $F_1\subset \overline{F_2}$. 
\end{Lemma}

\begin{proof}
Let $x\in \Omega\cap F_1$. Let $\alpha\in \Phi$ and $k\in \Z$ be such that $D_{\alpha,k}\supset F_2$. Then $D_{\alpha,k}\supset \overline{F_2}\ni x$ and thus $D_{\alpha,k}\supset  \cl(x)=\overline{F_1}$. As $\overline{F_2}$ is enclosed, we deduce that $\overline{F_2}\supset F_1$. 
\end{proof}

\begin{Proposition}\label{p_bijection_affine_faces_vectorial_faces_vertices}
Let  $\lambda\in \ve(\A)$. Equip $\sF_\lambda(\cH)$ and $\sF(\cH^0_\lambda)$ with the dominance order. Then the map $\Theta:\sF_\lambda(\cH)\rightarrow \sF(\cH^0_\lambda)$ defined by  $\Theta(F)= F^v(F):=\R_{>0}(F-\lambda)$, for $F\in \sF_\lambda(\cH)$ is an isomorphism of partially ordered sets. Its inverse is the map $\Xi:\sF(\cH^0_\lambda)\rightarrow \sF_\lambda(\cH)$ which associates to an element $F^v$ of $\sF(\cH^0_\lambda)$  the unique face of $\sF(\cH)$ containing $\germ_\lambda(\lambda+F^v)$. We have $\overline{\Xi(F^v)}=\cl(\germ_\lambda(\lambda+F^v))$ and $\Xi(F^v)\subset \lambda+F^v$.
\end{Proposition}

\begin{proof}
Set \[\Omega=\Omega_\lambda:=\bigcap_{\alpha\in \Phi_\lambda}\alpha^{-1}\left(]\alpha(\lambda)-1,\alpha(\lambda)+1[\right)\cap \bigcap_{\alpha\in \Phi\setminus \Phi_\lambda}\alpha^{-1}\left(]\lfloor \alpha(\lambda)\rfloor,\lceil \alpha(\lambda)\rceil[\right).\] Then $\Omega$ is a convex open set containing $\lambda$.

Let $F\in \sF_\lambda(\cH)$ and $x_1,x_2\in \R_{>0}(F-\lambda)$. Write $x_1=t_1(f_1-\lambda)$, $x_2=t_2(f_2-\lambda)$, with $t_1,t_2\in \R_{>0}$ and $f_1,f_2\in F$. Let $\alpha\in \Phi_\lambda$. Then the sign of $\alpha-\alpha(\lambda)$ (which can take the values ``$<0$'', ``$=0$'' or ``$>0$'') is constant on $\{f_1,f_2\}$ and thus $\alpha(x_1)$ and $\alpha(x_2)$ have the same sign. As this is true for every $\alpha\in \Phi_\lambda$, we deduce the existence of $F^v\in \sF(\cH^0_\lambda)$ such that $\Theta(F)\subset F^v$.

Let $x_1\in F^v$ and $f\in F$. Set $x_2=f-\lambda\in \Theta(F)\subset F^v$. Let $t_1\in \R_{>0}$ be such that $f_1:=x_1/t_1+\lambda\in \Omega$. Let $\alpha\in \Phi\setminus \Phi_\lambda$. 
By choice of $\Omega$, we have $\alpha(f_1)\in ]\lfloor \alpha(\lambda)\rfloor, \lceil \alpha(\lambda)\rceil[$. By Eq. \eqref{e_Property_faces_roots}, we have $\alpha(F)\subset ]\lfloor \alpha(\lambda)\rfloor, \lceil \alpha(\lambda)\rceil [$ and thus $\alpha-k$ has constant sign on $\{f_1,f\}$, for every $k\in \Z$. Let $\alpha\in \Phi_\lambda$. We have $\alpha(F^v)\in \{\R_{<0},\{0\},\R_{>0}\}$. If $\alpha(F^v)=\{0\}$, then $\alpha(f_1)=\alpha(\lambda)$ and $\alpha(x_2)=0$, so $\alpha(f)=\alpha(\lambda)$. Therefore $\alpha+k$ has constant sign on $\{f_1,f\}$, for every $k\in \Z$. Assume now $\alpha(F)=\R_{>0}$. Then $\alpha(f_1)>\alpha(\lambda)$ and thus $\alpha(f_1)\in ]\alpha(\lambda),\alpha(\lambda)+1[$. We have $\alpha(x_2)>0$, thus $\alpha(f)>\alpha(\lambda)$. By Eq . \eqref{e_Property_faces_roots}, $\alpha(f)\in ]\alpha(\lambda),\alpha(\lambda)+1[$. Therefore $\alpha+k$ has constant sign on $\{f,f_1\}$, for every $k\in \Z$. With the same reasoning when $\alpha(F^v)=\R_{<0}$, we deduce that $f$ and $f_1$ belong to the same face and thus $f_1\in F$. Consequently $x_1=t_1(f_1-\lambda)\in \Theta(F)$. Therefore we proved $\Theta(F)=F^v\in \sF(\cH^0_\lambda)$ and thus $\Theta$ is well-defined.

Let now $F^v\in \sF(\cH^0_\lambda)$ and $f_1,f_2\in \Omega\cap (\lambda+F^v)$. Let $\alpha\in \Phi$. Let $x_1=f_1-\lambda$, $x_2=f_2-\lambda\in F^v$. If $\alpha\notin \Phi_\lambda$, then $\alpha(f_1),\alpha(f_2)\in ]\lfloor \alpha(\lambda)\rfloor,\lceil \alpha(\lambda) \rceil[$. Therefore $\alpha+k$ has constant sign on $\{f_1,f_2\}$, for every $k\in \Z$. Let now $\alpha\in \Phi_\lambda$. If $\alpha(F^v)=\{0\}$, then $\alpha(f_1)=\alpha(f_2)=\alpha(\lambda)$ and thus $\alpha+k$ has constant sign on $\{f_1,f_2\}$ for every $k\in \Z$.  If $\alpha(F^v)=\R_{>0}$, then $\alpha(x_1),\alpha(x_2)>0$ and thus $\alpha(f_1),\alpha(f_2)\in ]\alpha(\lambda),\alpha(\lambda)+1[$, by definition of $\Omega$. With the same reasoning when $\alpha(F^v)=\R_{<0}$, we deduce that $f_1$ and $f_2$ belong the same face $F$ of $\sF_\lambda(\cH)$. Then $F$ is the unique face of $\sF_\lambda(\cH)$ containing $(\lambda+F^v)\cap \Omega\supset \germ_\lambda(\lambda+F^v)$. Consequently, $\Xi(F^v)=F$ is well-defined and we have \begin{equation}\label{e_inclusion_face_Theta}
\Omega\cap (\lambda+F^v)\subset F=\Xi(F^v). 
\end{equation} Moreover, for $x\in F$, we have $\cl(x)=\overline{F}$ and thus $\cl(\germ_\lambda(\lambda+F^v))=\overline{F}$.

Let $F\in \sF_\lambda(\cH)$ and $\tilde{F}=\Xi\circ \Theta(F)$. Then $\tilde{F}\supset \Omega\cap ]\lambda,f[$ for $f\in F$. Therefore $\tilde{F}\cap F$ is non-empty and hence $F=\tilde{F}$,  which proves that $\Xi\circ \Theta=\Id$. 

Let $F^v\in \sF(\cH^0_\lambda)$. Set $\tilde{F}^v=\Theta\circ \Xi(F^v)$. Let $f\in F^v$ be such that $\lambda+f\in \Omega$. Then $\lambda+f\in\Xi(F^v)$ by Eq. \eqref{e_inclusion_face_Theta}. Therefore $\lambda+f-\lambda=f\in \tilde{F}^v$ and hence $\tilde{F}^v\cap F^v\neq \emptyset$. Consequently $\tilde{F}^v=F^v$ and $\Theta\circ \Xi=\Id$, which proves that  $\Theta$ is the inverse of $\Xi$. 

Let $F_1,F_2\in \sF_\lambda(\cH)$ be such that $\overline{F_2}\supset F_1$. Then $\overline{\R_{>0}(F_2-\lambda)}=\R_{\geq 0}(\overline{F_2}-\lambda)\supset \R_{>0}(F_1-\lambda)$ and thus $\Theta$ is increasing. 

Let $F_1^v,F_2^v\in \sF(\cH_\lambda^0)$ be such that $\overline{F^v_2}\supset F^v_1$. We have \[\overline{\Xi(F_2^v)}\supset (\lambda+\overline{F_2^v})\cap \Omega\supset (\lambda+F_1^v)\cap \Omega \supset \Xi(F_1^v)\cap \Omega,\] by Eq. \eqref{e_inclusion_face_Theta}. Using Lemma~\ref{l_open_segments_faces}, we deduce that  $\Xi(F_1^v)\subset \overline{\Xi(F_2^v)}$, which proves that $\Xi$ is increasing. Thus $\Theta$ is an isomorphism. 

Using Eq. \eqref{e_inclusion_face_Theta} and the assetion following it, we get the last claim of the proposition. 
\end{proof}

\begin{Remark}
Note that for dimension reasons, the bijection above sends alcoves on chambers.
\end{Remark}

\subsubsection{Star of a face}

In the sequel, we write $W_x^{\mathrm{aff}}=W_{\lbrace x \rbrace}^{\mathrm{aff}}$, for the fixator of $x \in \A$ in $W^{\mathrm{aff}}$.

\begin{Lemma}\label{l_fixator_subset_alcove}
Let $C$ be an alcove of $(\A,\cH)$ and $J\subset C$ be a non-empty set. Let $F$ be the smallest face of $C$ dominating $J$ (this face is well-defined by Lemma~\ref{l_enclosure_subset_alcove}). Then $\bigcap_{x\in J} W^{\mathrm{aff}}_x=W^{\mathrm{aff}}_F$. 
\end{Lemma}

\begin{proof}
  For every $x\in J$, $W_F^{\mathrm{aff}}=W^{\mathrm{aff}}_{\overline{F}}$ fixes $x$ and thus $W^{\mathrm{aff}}_F\subset \bigcap_{x\in J} W^{\mathrm{aff}}_x$. Let $w\in \bigcap_{x\in J} W^{\mathrm{aff}}_x$. By Lemma~\ref{l_enclosure_subset_alcove}, there exists $y\in \conv(J)\cap F$. Then $w$ fixes $y$. By Theorem~\ref{t_Weyl_group_Coxeter}, $W^{\mathrm{aff}}_y=W^{\mathrm{aff}}_F$, which proves that $w\in W^{\mathrm{aff}}_F$ and completes the proof of the lemma.  
\end{proof}

\begin{Lemma}\label{l_positivity_result_affine_root_systems}
Assume that $\Phi$ is irreducible. Let $F$ be a face of $(\A,\cH)$ and let $C$ be an alcove dominating $F$. Let $\Sigma^{\mathrm{aff}}_C$ be the basis associated with $C$ and set $\cM_F=\{\underline{\alpha}\in \Sigma^{\mathrm{aff}}_C\mid \underline{\alpha}(F)>0\}$. Let $x\in \A$. Then there exists $w\in W^{\mathrm{aff}}_F$ such that $\underline{\alpha}(w.x)\geq 0$ for every $\underline{\alpha}\in \Sigma^{\mathrm{aff}}_C\setminus\cM_F$.     
\end{Lemma}

\begin{proof}
  If $F$ is an alcove, $\Sigma^{\mathrm{aff}}_C\setminus\cM_F$ is empty and there is nothing to prove. We now assume that $F$ is not an alcove.  Let $\lambda\in \ve(F)$. Let $F^v=\R_{>0}(F-\lambda)$ and $C^v=\R_{>0}(C-\lambda)$. By Proposition~\ref{p_bijection_affine_faces_vectorial_faces_vertices}, $F^v$ and $C^v$ are faces of $(\A,\cH^0_\lambda)$. Let $\Phi_{F^v}=\{\alpha\in \Phi_\lambda\mid \alpha(F^v)=\{0\}\}$ and let $\Phi_{F^v,+}=\{\alpha\in \Phi_{F^v}\mid \alpha(C^v)>0\}$. By Lemma~\ref{l_root_system_Fv}, $\Phi_{F^v}$ is a root system in $(\vect(F^v))^\perp$. By \cite[VI  \S 1.7 Corollaire 1]{bourbaki1981elements}, there exists  a chamber $C_{F^v}^v$ of $(\vect(F^v))^\perp$ such that $\Phi_{F^v,+}=\{\alpha\in \Phi_{F^v}\mid \alpha(C_{F^v}^v)>0\}$.

    We can write $x-\lambda=u+v$, where $u\in \vect(F^v)$ and $v\in \vect(F^v)^\perp$.   Let $W^v_{\Phi_{F^v}}=\langle s_\alpha\mid \alpha\in \Phi_{F^v}\rangle\subset W^v_{\Phi}$ be the vectorial Weyl group of $\Phi_{F^v}$. Then $(\vect(F^v))^\perp$ is the union of the closed vectorial chambers of $\Phi_{F^v}$ and, since  $W^v_{\Phi_{F^v}}$ acts transitively on the set of vectorial chambers of $\Phi_{F^v}$, there exists $w\in W^v_{\Phi_{F^v}}$  such that $\alpha(w.v)\geq 0$, for every $\alpha\in \Phi_{F^v,+}$. We then have $\alpha(w.(x-\lambda))\geq 0$, for all $\alpha\in \Phi_{F^v,+}$, since $W^v_{\Phi_{F^v}}$ pointwise fixes  $\vect(F^v)$.

    We have $t_\lambda W^v_{\Phi_{F^v}} t_{-\lambda}=W^{\mathrm{aff}}_F$, where $t_\mu:\A\rightarrow \A$ is the translation by the vector $\mu$, for $\mu\in \A$. Let $\underline{\alpha}\in \Sigma^{\mathrm{aff}}_C\setminus\cM_F$. Set $\underline{w}=t_{\lambda}wt_{-\lambda
    }\in W^{\mathrm{aff}}_F$. Write $\underline{\alpha}=\alpha+k$, where $k\in \Z$. Then $\underline{\alpha}(\lambda)=0$ and thus $k=-\alpha(\lambda)$. Moreover, $\underline{\alpha}(C)>0$ and thus $\alpha\in \Phi_{F^v,+}$. Then $\underline{\alpha}(\underline{w}.x)=\alpha(\lambda+w.(x-\lambda))-\alpha(\lambda)=\alpha(w.(x-\lambda))\geq 0$. This being true for every $\underline{\alpha}\in \Sigma^{\mathrm{aff}}_C\setminus\cM_F$, the lemma follows.
\end{proof}

\begin{Definition}
    Let $F$ be a face of $(\A,\cH)$. The \textbf{star of $F$} is the union of the faces of $(\A,\cH)$ dominating $F$.
\end{Definition}

\begin{Lemma}\label{l_union_alcoves_containing_face}
Let $F$ be a face of $(\A,\cH)$ and let $E$ be the star of $F$. Let $C$ be an alcove of $(\A,\cH)$ dominating $F$ and $H_1,\ldots,H_k\in \cH$ be the walls of $C$. Up to renumbering, we may assume that $H_1,\ldots,H_{k'}\supset F$ and $H_{k'+1},\ldots,H_{k}$ do not contain $F$, for some $k'\in \llbracket 1,k\rrbracket$. For $i\in \llbracket 1,k\rrbracket$, denote by $\mathring{D}_i$ the open half-space of $\A$ delimited by $H_i$ and containing $C$. Then $E=\bigcap_{w\in W_F^{\mathrm{aff}},i\in \llbracket k'+1,k\rrbracket} w.\mathring{D}_i$. In particular, $E$ is open and convex.
\end{Lemma}

\begin{proof}

Let $E'=E\cap \overline{C}$.   Then by Lemma~\ref{l_transitive_actions_element_fixing_face}, $E=W_F.E'$.

We first assume that $\Phi$ is irreducible. Let   $\Sigma^{\mathrm{aff}}_C$ be the basis of $\Phi^{\mathrm{aff}}$ defined by $C$, i.e., $\Sigma^{\mathrm{aff}}_C$ is the set of affine roots $\underline{\alpha}$ such that $\underline{\alpha}^{-1}(\{0\})$ is  a wall of $C$ and $\underline{\alpha}(C)>0$. By Lemma~\ref{l_description_faces_intersections}, we have $C=\{x\in \A\mid \underline{\alpha}(x)>0, \forall \underline{\alpha}\in \Sigma^{\mathrm{aff}}_C\}$. Let $\cM_F=\{\underline{\alpha}\in \Sigma^{\mathrm{aff}}_C\mid \underline{\alpha}(F)>0\}$. If $F'$ is a face of $C$, then by Lemma~\ref{l_description_faces_intersections}, we can write $F'=\{x\in \A\mid \underline{\alpha}(x)\ R_{\underline{\alpha}}\ 0,\forall \underline{\alpha}\in \Sigma^{\mathrm{aff}}_C\}$, where $(R_{\underline{\alpha}})\in \{>,=\}^{\Sigma^{\mathrm{aff}}_C}$. Then $\overline{F'}=\{x\in \A\mid \underline{\alpha}(x)\  \overline{R_{\underline{\alpha}}}\ 0,\forall \underline{\alpha}\in \Sigma^{\mathrm{aff}}_C\}$, where $\overline{R_{\underline{\alpha}}}$ is equal to ``$\geq$'' if $R_{\underline{\alpha}}$ is equal to ``$>$'' and equal to ``$=$'' if $R_{\underline{\alpha}}$ is equal to ``$=$''. Therefore if $\overline{F'}\supset F$, the symbol $R_{\underline{\alpha}}$ is equal to ``$>$'' for every $\underline{\alpha}\in \Sigma^{\mathrm{aff}}_C\setminus\cM_F$.   Therefore: \begin{equation}\label{e_description_E'}
    E'=\{x\in \A\mid \underline{\alpha}(x)\geq 0, \forall \underline{\alpha}\in \Sigma^{\mathrm{aff}}_C\setminus\cM_F, \underline{\alpha}(x)>0, \forall \underline{\alpha}\in \cM_F\}.
\end{equation}

Let us prove that 

\begin{equation}\label{e_explicit_description_residue}
E=\bigcap_{\underline{\alpha}\in W^{\mathrm{aff}}_F.\cM_F} \underline{\alpha}^{-1}(\R_{>0}):=E_1.
\end{equation}

If $s_{\underline{\alpha}}$ denotes the orthogonal reflection with respect to $\underline{\alpha}^{-1}(\{0\})$ for $\underline{\alpha}\in \Phi^{\mathrm{aff}}$, then $(W^{\mathrm{aff}},\{s_{\underline{\alpha}}\mid \underline{\alpha}\in \Sigma^{\mathrm{aff}}_C\})$ is a Coxeter group  and $W^{\mathrm{aff}}_F=\langle s_{\underline{\alpha}}\mid \underline{\alpha}\in \Sigma^{\mathrm{aff}}_C\setminus\cM_F\rangle$, according to Theorem~\ref{t_Weyl_group_Coxeter} (3). Thus by \cite[Corollaries 1.4.8 and 1.4.6]{bjorner2005combinatorics} if $w\in W^{\mathrm{aff}}_F$ and $\underline{\alpha}\in \cM_F$, we have $\ell^{\mathrm{aff}}(ws_{\underline{\alpha}})=\ell^{\mathrm{aff}}(w)+1$, where $\ell^{\mathrm{aff}}$ denotes the length on $(W^{\mathrm{aff}},\{s_{\underline{\alpha}}\mid \underline{\alpha}\in \Sigma^{\mathrm{aff}}_C\})$. By \cite[(2.2.8)]{macdonald2003affine}, we deduce $w.\underline{\alpha}\in \Phi^{\mathrm{aff}}_+$, where $\Phi^{\mathrm{aff}}_+=\{\underline{\gamma}\in \Phi^{\mathrm{aff}}\mid \underline{\gamma}(C)\geq 0\}$. Since 
$w\in \langle s_{\underline{\beta}}\mid \underline{\beta}\in \Sigma^{\mathrm{aff}}_C\setminus\cM_F$, we can write 
$w.\underline{\alpha}=\underline{\alpha}+\sum_{\underline{\beta}\in \Sigma^{\mathrm{aff}}_C\setminus\cM_F} n_{\underline{\beta}}\underline{\beta}$, where the  $n_{\underline{\beta}}$ are non-negative integers (by \cite[(1.4) and 4.6 Proposition]{macdonald1971affine}). Therefore $w.\underline{\alpha}(E')>0$. Consequently $E'\subset E_1$ and since $E_1$ is is $W^{\mathrm{aff}}_F$-invariant, we deduce that $E\subset E_1$.

 Let $x\in E_1$. Then by Lemma~\ref{l_positivity_result_affine_root_systems}, there exists $w\in W_F^{\mathrm{aff}}$ such that $\underline{\alpha}(w.x)\geq 0$ for all $\underline{\alpha}\in \Sigma^{\mathrm{aff}}_C\setminus\cM_F$. By definition of $E_1$, we then have $w.x\in E'$, by Eq. \eqref{e_description_E'} and thus $x\in E$, which proves Eq. \eqref{e_explicit_description_residue}, when $\Phi$ is irreducible. 
 
Assume now that $\Phi$ is reducible. Using the notation of Eq. \eqref{e_decomposition_Phi}, with $\Phi$ instead of $\Psi$, we write $\A=\A_1\times \ldots \times \A_\ell$ and  $\Phi=\Phi_1\sqcup \ldots \sqcup \Phi_\ell$, where $\ell \in \Z_{\geq 2}$ and $\Phi_i$ is an irreducible root system in $\A_i$ for $i\in \llbracket 1,\ell\rrbracket$. Write $C=C_1\times \ldots\times C_\ell$ and $F=F_1\times \ldots \times F_\ell$.  Then with obvious notations, we have: \begin{align*}E&=(W_{F_1}^{\mathrm{aff}}\times \ldots \times W_{F_\ell}^{\mathrm{aff}}).(E'_1\times \ldots \times E_\ell')\\
 &=(W_{F_1}^{\mathrm{aff}}.E_1'\times \ldots  \times  W^{\mathrm{aff}}_{F_\ell}.E_{\ell}') \\&=\bigcap_{i=1}^\ell\bigcap_{\underline{\alpha}\in W^{\mathrm{aff}}_{F_i}. \cM_{F_i}} \underline{\alpha}^{-1}(\R_{>0}),\end{align*} 
 which proves the lemma.

\end{proof}

\begin{Lemma}\label{l_barycenter_orbit}
Let $F,F'$ be  faces of $(\A,\cH)$ such that $F'$ dominates $F$. Let $x\in F'$ and $y=\frac{1}{|W^{\mathrm{aff}}_F|}\sum_{w\in W_F^{\mathrm{aff}}} w.x$. Then $y\in F$ and $y$ is the orthogonal projection of $x$ on the support of $F$. 
\end{Lemma}

\begin{proof}
By Lemma~\ref{l_union_alcoves_containing_face}, $y$ belongs to a face $F''$ dominating $F$. We thus have $W^{\mathrm{aff}}_{F''}=W^{\mathrm{aff}}_{\overline{F''}}\subset W^{\mathrm{aff}}_{F}$. Moreover, $W^{\mathrm{aff}}_F$ fixes $y$ and thus by Theorem~\ref{t_Weyl_group_Coxeter}, $W^{\mathrm{aff}}_F$ fixes $F''$. Consequently $W^{\mathrm{aff}}_F=W^{\mathrm{aff}}_{F''}$ and hence $F=F''$, by Lemma~\ref{l_comparison_fixators}. In other words, $y\in F$. 

Let $z\in F$ and set $t=\langle x-y\mid y-z\rangle$. Then for $w\in W^{\mathrm{aff}}F$, we have $t=\langle x-y\mid w^{-1}.(y-z)\rangle=\langle w.x-w.y\mid y-z\rangle=\langle w.x-y\mid y-z\rangle$. Therefore \[|W^{\mathrm{aff}}_F|t=\sum_{w\in W^{\mathrm{aff}}_F} \langle w.x\mid y-z\rangle-|W^{\mathrm{aff}}_F|\langle y\mid y-z\rangle=|W^{\mathrm{aff}}_F|\langle y\mid y-z\rangle-|W^{\mathrm{aff}}_F|\langle y\mid y-z\rangle=0,\] which proves the lemma.
\end{proof}

\begin{Lemma}\label{l_star_vertex_intersected_face}
    Let $F$ be a face of $(\A,\cH)$, $\lambda\in \ve(F)$ and $E$ be the star of $\lambda$. Then $F=E\cap (\lambda+\R_{> 0}(F-\lambda))$.
\end{Lemma}

\begin{proof}
 The inclusion $F\subset E\cap (\lambda+\R_{> 0}(F-\lambda))$ is clear.  Let $x\in E\cap (\lambda+\R_{>0} (F-\lambda))$ and let $F_1$ be the face of $(\A,\cH)$ containing $x$. Then $F_1$ dominates $\lambda$ and $x-\lambda\in \R_{>0}(F_1-\lambda)\cap \R_{>0}(F-\lambda)$. As the faces of $(\A,\cH^0_\lambda)$ form a partition of $\A$, we deduce $\R_{>0}(F_1-\lambda)=\R_{>0}(F-\lambda)$  and thus $F_1=F$, by Proposition~\ref{p_bijection_affine_faces_vectorial_faces_vertices}. Therefore $E\cap (\lambda+\R_{> 0}(F-\lambda))\subset F$ and the lemma follows.
\end{proof}

\color{black}

\subsection{Affine weight polytopes}\label{ss_affine_weight_polytopes}
\subsubsection{Definition of affine weight polytopes}

Let $C_0$\index{c@$C_0$} be the fundamental alcove of $\A$, i.e., the alcove of $\A$ based at $0$ and contained in $C^v_f$.
Let us fix $\bb\in C_0$.
Let $C\in  \Alc(\cH)$. Then there exists $w\in W^{\mathrm{aff}}$ such that $C=w.C_0$. We then set $\bb_C=w.\bb\in C$\index{b@$\bb_C$}. The element $w$ is uniquely determined by $C$, since  $W^{\mathrm{aff}}$ acts simply transitively on the set of alcoves of $(\A,\cH)$. Thus $\bb_C$ is well-defined.

\begin{Definition}\label{d_affine_Weyl_polytope}
For $\lambda\in \ve(\A)$, we set $\overline{\cP(\lambda)}=\overline{\cP_{\bb,\Phi}(\lambda)}=\conv(\{\bb_C \mid C\in \Alc_\lambda(\cH)\})$. Then we call  $\overline{\cP(\lambda)}$ the \textbf{closed affine Weyl polytope associated with $\lambda$} (and $\bb$). The \textbf{(open) affine Weyl polytope associated with $\lambda$} is the interior of $\overline{\cP(\lambda)}$. By Lemma~\ref{l_transitive_actions_element_fixing_face}, \begin{equation}\label{e_affine_Weyl_polytope_1}
\overline{\cP(\lambda)}=\conv(W^{\mathrm{aff}}_\lambda.\bb_C)
\end{equation}
for any $C \in \Alc_\lambda(\cH)$.
\end{Definition}

Note that in order to avoid confusions  we use roman letters to denote the faces of $(\A,\cH)$ and curly letters for the faces of the $\overline{\cP(\lambda)}$.

\subsubsection{Description of the faces of the affine weight polytopes}

Let $\lambda\in \ve(\A)$.  We have $W^v_{\Phi_\lambda}=\langle s_\alpha\mid \alpha\in \Phi_\lambda \rangle\subset W^v_\Phi$.  We denote by $\bt_\lambda$ the translation on $\A$ by the vector $\lambda$ (this is a not a element of $W^{\mathrm{aff}}$ in general).

\begin{Lemma}\label{l_translation_fixator_0}
Let $\lambda\in \ve(\A)$. Then $W_\lambda^{\mathrm{aff}}=\bt_\lambda W^v_{\Phi_\lambda} \bt_\lambda^{-1}$. 
\end{Lemma}

\begin{proof}
Let $\alpha\in \Phi_\lambda$. Then $\bt_\lambda s_\alpha \bt_{\lambda}^{-1}$ is the orthogonal reflection fixing $\lambda+\alpha^{-1}(\{0\})=\alpha^{-1}(\alpha(\lambda))$. Therefore it belongs to $W^{\mathrm{aff}}_\lambda$. By Theorem~\ref{t_Weyl_group_Coxeter}, $W_\lambda^{\mathrm{aff}}$ is generated by the reflections fixing the walls containing  $\lambda$, which proves the lemma. 
\end{proof} 

Let $\lambda\in \ve(\A)$. Then $\overline{\cP(\lambda)}=\conv(\bt_\lambda W^v_{\Phi_\lambda}.\bt^{-1}.\bb_C)=\lambda+\conv(W^v_{\Phi_\lambda}.(-\lambda+\bb_C))$ for $C \in \Alc_\lambda(\cH)$. Moreover, $\alpha(\bb_C)\notin \Z$, for all $\alpha\in \Phi$ and thus $\alpha(-\lambda+\bb_C)\neq 0$ for all $\alpha\in \Phi_\lambda$. Therefore, with the notation of subsection~\ref{ss_vectorial_Weyl_polytope}, we have:  
\begin{equation}\label{e_affine_Weyl_polytope_2}
\overline{\cP(\lambda)}=\lambda+\overline{\cP^0_{-\lambda+\bb_C,\Phi_\lambda}}.
\end{equation}

\begin{Definition}\label{d_F_F}
Let $\lambda\in \ve(\A)$ and $F$ be a face of $(\A,\cH)$ dominating $\lambda$. We choose an alcove $C$ dominating $F$ and we set  $\cF_F=\Int_r(\conv (W^{\mathrm{aff}}_F.\bb_C))$\index{f@$\cF_F$}. This is well-defined by Lemma~\ref{l_transitive_actions_element_fixing_face}. 
\end{Definition}

Note that for $\lambda\in \ve(\A)$, we have $\cF_{\lambda}=\cP(\lambda)$.

\begin{Lemma}\label{l_inclusion_P(lambda)_union_alcoves}
\begin{enumerate}

\item Let $F$ be a face of $(\A,\cH)$. Then $\overline{\cF_F}$ is contained in the star of $F$. In particular, if $\lambda\in \ve(\A)$, then $\overline{\cP(\lambda)}=\overline{\cF_{\{\lambda\}}}$ is contained in the star of $\lambda$.

\item Let $B$ be a bounded subset of $\A$. Then $\{\lambda\in \ve(\A)\mid \overline{\cP(\lambda)}\cap B\neq \emptyset\}$ is finite. 
\end{enumerate}
\end{Lemma}

\begin{proof}
1. Let $E$ be the star of $F$.  We have $\overline{\cF_{F}}=\conv(W^{\mathrm{aff}}_F.\bb_C)$, for any alcove $C$ dominating $F$. As $\bb_C\in E$ and as $E$ is stabilized by $W^{\mathrm{aff}}_F$, we deduce $W^{\mathrm{aff}}_F.\bb_C\subset E$. By Lemma~\ref{l_union_alcoves_containing_face}, $E$ is convex and thus $\overline{\cF_F}\subset E$.

2. As $B$ is bounded, there exists a finite set of closed alcoves $F$ such that $B\subset \bigcup_{\overline{C}\in F} \overline{C}$. Let $F'$ be the set of closed alcoves containing a vertex of an element of $F$. Then by statement (1), if $\lambda\in \ve(\A)$ is such that $B\cap\overline{\cP(\lambda)}$ is non-empty, then $\lambda\in \bigcup_{\overline{C}\in F'}\overline{C}\cap \ve(\A)$, which is finite. This proves  statement (2). 
\end{proof}

For  a polytope $\cQ$ of $\A$, we denote by $\mathrm{Face}(\cQ)$\index{f@$\mathrm{Face(\cQ)}$} its set of faces. We equip $\mathrm{Face}(\cQ)$ with the dominance order $\preceq$.
Recall that, if $E$ is an affine subspace of $\A$, its direction $\di(E)$, is the vector subspace $\{x-x'\mid x,x'\in E\}$ of $\A$. (I prefer to recall the definition since it is far.)

\begin{Proposition}\label{p_description_faces_affine} (see Figure~\ref{f_faces_P_affine})
    Let $\lambda\in \ve(\A)$. Let $\cP^0=\cP(\lambda)-\lambda$.   \begin{enumerate}
        \item The map $\mathrm{Face}(\overline{\cP(\lambda)})\rightarrow \mathrm{Face}(\overline{\cP^0})$ given by $\cF\mapsto \cF-\lambda$ is an isomorphism of partially ordered sets.

        \item Let $C\in \Alc_\lambda(\cH)$ and $C^v=\R_{>0}(C-\lambda)$. Set $\bb_{C^v}^0=\bb_{C}-\lambda$. For each chamber $C_1^v$ of $(\A,\cH^0_\lambda)$, set $\bb_{C_1^v}^0=w.\bb_{C^v}^0$, where $w\in W^v_{\Phi_\lambda}$ satisfies $C_1^v=w.C^v$. For $F^v\in \sF(\cH^0_\lambda)$, set $\cF^0_{F^v}=\Int_r(\conv(W^v_{F^v,\Phi_\lambda}.\bb_{C_1^v}))$, where $C_1^v$ is a chamber of $(\A,\cH^0_\lambda)$ dominating $F^v$. Then for every $F^v\in \sF(\cH^0_\lambda)$, we have  $\cF^0_{F^v}+\lambda=\cF_F$, where $F\in \sF_\lambda(\cH)$ satisfies $F^v=\R_{>0}(F-\lambda)$. 

    \item  The map  $F_1\mapsto \cF_{F_1}$ is an isomorphism between $(\sF_\lambda(\cH),\preceq)$ and $(\mathrm{Face}(\overline{\cP(\lambda)}),\succeq)$. Moreover, $\di(\supp(\cF_{F_1}))=(\di(\supp(F_1)))^\perp$ and  $\dim(\supp(\cF_{F_1}))=\mathrm{codim}_\A(F_1)$ for every $F_1\in \sF_\lambda(\cH)$.

    \end{enumerate} 
\end{Proposition}

 \begin{figure}[h]
 \centering
 \includegraphics[scale=0.35]{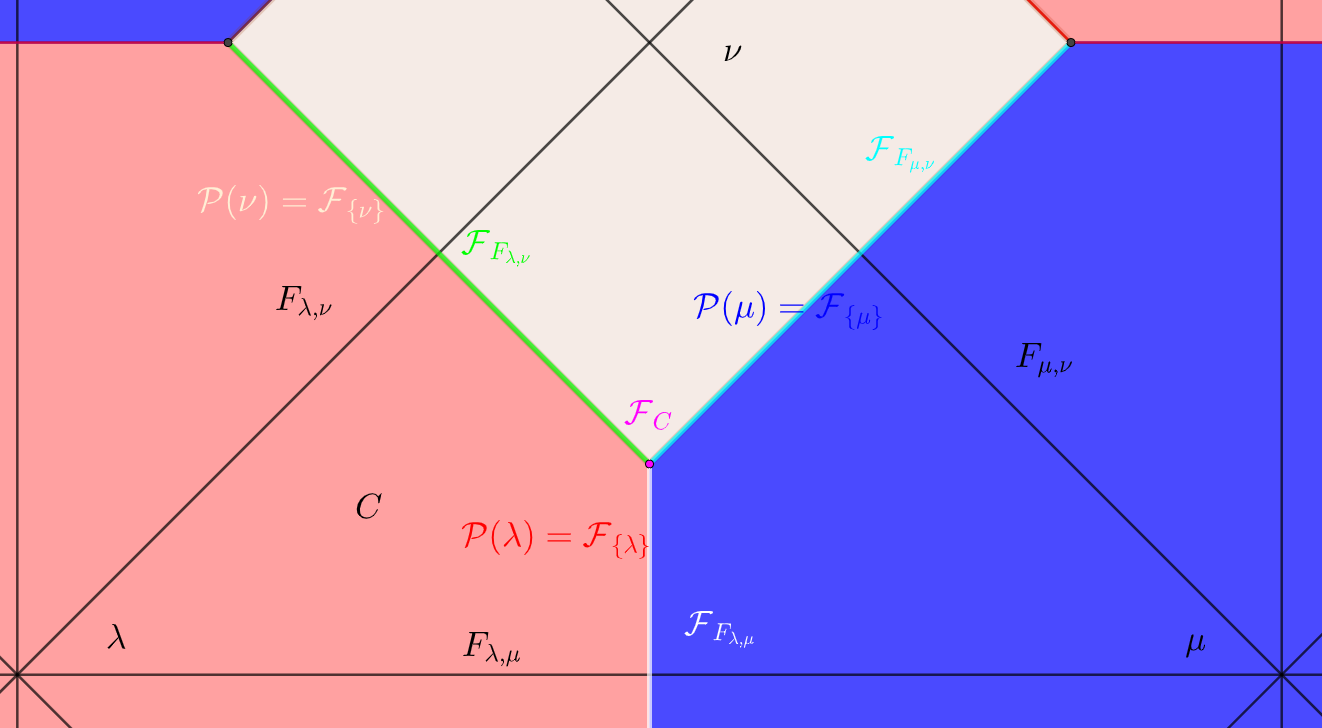}
 \caption{Illustration of Proposition~\ref{p_description_faces_affine}: the case of $\mathrm{B}_2$.}
 \label{f_faces_P_affine}
\end{figure}

\begin{proof}
Statement (1) is clear. 

Now, we prove statement (2). By Proposition~\ref{p_bijection_affine_faces_vectorial_faces_vertices}, $C^v$ is a chamber of $(\A,\cH^0_\lambda)$. By Eq. \eqref{e_affine_Weyl_polytope_2}, we have $\overline{\cP(\lambda)-\lambda}=\overline{\cP^0}=\overline{\cP^0_{\bb^0_{C^v},\Phi_\lambda}}=\conv(W^v_{\Phi_\lambda}.\bb^0_{C^v})$. Let $F^v_1\in \sF(\cH^0_\lambda)$ and let $C_1^v$ be a chamber of $(\A,\cH^0_\lambda)$ dominating $F^v_1$.  Let $\cF=\cF^0_{F_1^v}+\lambda$.   Let $F_1,C_1\in \sF_\lambda(\A)$ be such that $F^v_1=\R_{>0}(F_1-\lambda)$ and $C_1^v=\R_{>0}(C_1-\lambda)$.  

Using Lemma~\ref{l_transitive_actions_element_fixing_face}, we get $C_1=w_1.C$, with $w_1\in W^{\mathrm{aff}}_\lambda$. Then we have: \[ w_1.C-\lambda = C_1- \lambda  =w_1.C-w_1.\lambda  =\vec{w_1}.(C-\lambda),\] where $\vec{w_1}$ is the linear part of $w_1$. By Lemma~\ref{l_translation_fixator_0}, $\vec{w_1}\in W^v_{\Phi_\lambda}$.  Moreover we have:\[\R_{>0} \vec{w_1}.(C-\lambda)=\vec{w_1}.C^v=\R_{>0}(C_1-\lambda)=C_1^v.\] 

Hence: \[\overline{\cF}=\bt_\lambda.\overline{\cF^0_{F_1^v}}=\conv(\bt_\lambda.W^v_{F^v_1,\Phi_\lambda}.\bb^0_{C^v})=\conv\left(W^{\mathrm{aff}}_{\bt_\lambda.F^v_1}.(\lambda+\bb^0_{C^v_1})\right)=\conv(W^{\mathrm{aff}}_{F_1}.(\lambda+\bb_{C^v_1}^0)),\] where $W^v_{F_1^v,\Phi_\lambda}$ is the fixator of $F_1^v$ in $W^v_{\Phi_\lambda}$. Moreover,  $\bb_{C_1}=w_1.\bb_C=w_1.(\lambda+\bb_{C^v}^0)=w_1.\lambda+\vec{w_1}.\bb_{C^v}^0=\lambda+\bb_{C_1^v}^0$, which proves that $\overline{\cF}=\conv(W^{\mathrm{aff}}_{F_1}.\bb_{C_1})=\overline{\cF_{F_1}}$. Therefore $\cF=\Int_r(\overline{\cF_{F_1}})=\cF_{F_1}$, which proves  (2). 

Finally, we prove statement (3). Indeed, by statements (1), (2) and by Propositions~\ref{p_bijection_affine_faces_vectorial_faces_vertices} and \ref{p_description_faces_P0}, the map $F_1\mapsto \cF_{F_1}$ can be decomposed as  \[F_1\mapsto \R_{>0}(F_1-\lambda)\mapsto \cF_{\R_{>0}(F_1-\lambda)}^0\mapsto \cF_{\R_{>0}(F_1-\lambda)}^0+\lambda=\cF_{F_1}\] and then it is an isomorphism between $(\sF_\lambda(\cH),\preceq)$ and $(\mathrm{Face}(\overline{\cP(\lambda)}),\succeq)$. The last assertion follows from Proposition~\ref{p_description_faces_P0} and statement (2).
\end{proof}

\subsubsection{Faces and intersection of affine weight polytopes}

Recall that, for any pair $(\lambda, \mu)\in \A^2$, we write $[\lambda, \mu]=\lbrace t \lambda + (1-t)\mu \mid t \in [0,1] \rbrace$ and $]\lambda, \mu[=\mathrm{Int}_r([\lambda, \mu])=\lbrace t \lambda + (1-t)\mu \mid t \in ]0,1[\rbrace$. 

\begin{Lemma}\label{l_intersection_P(lambda)_P(mu)_adjacent}
Let $\lambda,\mu\in \ve(\A)$ be adjacent. Then $\overline{\cP(\lambda)}\cap \overline{\cP(\mu)}$ is a closed facet of both $\overline{\cP(\lambda)}$ and $\overline{\cP(\mu)}$. More precisely, \[\overline{\cP(\lambda)}\cap \overline{\cP(\mu)}=\conv(W_{[\lambda,\mu]}^{\mathrm{aff}}.\bb_C),\] where $C$ is any alcove dominating $]\lambda,\mu[$.
\end{Lemma}

\begin{proof}
As $\lambda$ and $\mu$ are adjacent, $]\lambda,\mu[$ is a face of some alcove $C$. By Proposition~\ref{p_description_faces_affine}, $\overline{\cF}:=\conv(W_{[\lambda,\mu]}^{\mathrm{aff}}.\bb_C)$ is a closed facet of both $\overline{\cP(\lambda)}$ and $\overline{\cP(\mu)}$. It remains to prove that $\lambda$ and $\mu$ are on opposite sides of  $\cF$. By Lemma~\ref{l_barycenter_orbit}, $\lambda=\frac{1}{|W^{\mathrm{aff}}_\lambda|}\sum_{w\in W^{\mathrm{aff}}_\lambda}w.\lambda\in \overline{\cP(\lambda)}$ and similarly, $\mu\in \overline{\cP(\mu)}$. By Lemma~\ref{l_barycenter_orbit} applied with $F=]\lambda,\mu[$, the orthogonal projection $x$ of $\bb_C$ on the line containing $]\lambda,\mu[$ belongs to $]\lambda,\mu[ \cap \overline{\cF}$. The support of $\cF$ is the hyperplane $H=x+(\R(\mu-\lambda))^\perp$ by Proposition \ref{p_description_faces_affine}(3). Then $\overline{\cP(\lambda)}$ (resp. $\overline{\cP(\mu)}$) is contained in the half-space of $\A$ delimited by $H$ and containing $\lambda$ (resp. $\mu$), which proves the lemma. 
\end{proof}

\begin{Lemma}\label{l_intersection_neighbour_polytopes}
Let $\lambda,\mu\in \ve(\A)$ and assume that there exists an alcove dominating both $\lambda$ and $\mu$. Then $\overline{\cP(\lambda)} \cap \overline{\cP(\mu)}$ is a face of $\overline{\cP(\lambda)}$. Moreover, if $C$ is an alcove dominating $\lambda$ and $\mu$, then $\overline{\cP(\lambda)} \cap \overline{\cP(\mu)}=\conv(W^{\mathrm{aff}}_{[\lambda,\mu]}.\bb_C)=\conv(W^{\mathrm{aff}}_{F}.\bb_C)$, where $F$ is the smallest face dominating $\{\lambda,\mu\}$ (whose existence is provided by Lemma~\ref{l_enclosure_subset_alcove}).

    

\end{Lemma}

\begin{proof}
If $\Phi$ is irreducible, then this is Lemma~\ref{l_intersection_P(lambda)_P(mu)_adjacent}, so we assume that $\Phi$ is reducible. We use the same notation as in Eq. \eqref{e_decomposition_Phi}, with $\Phi$ instead of $\Psi$.  Write $\A=\A_1\oplus\ldots\oplus \A_\ell$ and  $\Phi=\bigsqcup_{i=1}^\ell\Phi_i$. We have  $W^{\mathrm{aff}}= W^{\mathrm{aff}}_{\Phi_1}\times \ldots \times W^{\mathrm{aff}}_{\Phi_\ell}$, where $W^{\mathrm{aff}}_{\Phi_i}$ is the affine  Weyl group of $(\A_i,\Phi_i)$, for $i\in \llbracket 1,\ell\rrbracket$. Write $\bb_C=(\bb_1,\ldots,\bb_\ell)$. We have $\ve(\A)=\ve(\A_1)\times\ldots \times \ve(\A_\ell)$ and if $(\tau_1,\ldots,\tau_\ell)\in \ve(\A)$, we have $\cP_{\bb,\Phi}((\tau_1,\ldots,\tau_\ell))=\cP_{\bb_1,\Phi_1}(\tau_1)\times \ldots \times \cP_{\bb_\ell,\Phi_\ell}(\tau_\ell)$. Write  $\lambda=(\lambda_1,\ldots,\lambda_\ell)$ and $\mu=(\mu_1,\ldots,\mu_\ell)$. Then for all $i\in \llbracket 1,\ell\rrbracket$, either $\lambda_i=\mu_i$, or $\lambda_i$ and $\mu_i$ are adjacent. By Lemma \ref{l_intersection_P(lambda)_P(mu)_adjacent}, we get that:
$$\overline{\cP_{\bb_i,\Phi_i}(\lambda_i)}\cap \overline{\cP_{\bb_i,\Phi_i}(\mu_i)}=\conv(W_{\Phi_i,[\lambda_i,\mu_i]}^{\mathrm{aff}}.\bb_i)$$
and that this intersection is a facet of both $\cP_{\bb_i,\Phi_i}(\lambda_i)$ and $\cP_{\bb_i,\Phi_i}(\mu_i)$. We deduce that:

$$
    \overline{\cP(\lambda)}\cap \overline{\cP(\mu)}=\prod_{i=1}^\ell (\overline{\cP_{\bb_i,\Phi_i}(\lambda_i)}\cap \overline{\cP_{\bb_i,\Phi_i}(\mu_i)})=\prod_{i=1}^\ell \conv(W_{\Phi_i,[\lambda_i,\mu_i]}^{\mathrm{aff}}.\bb_i)
$$
is a face of both $\cP(\lambda)$ and $\cP(\mu)$. The equalities $$\overline{\cP(\lambda)} \cap \overline{\cP(\mu)}=\conv(W^{\mathrm{aff}}_{[\lambda,\mu]}.\bb_C)=\conv(W^{\mathrm{aff}}_{F}.\bb_C)$$ immediately follow from the facts that:
$$
    \overline{\cP(\lambda)}\cap \overline{\cP(\mu)}=\prod_{i=1}^\ell \conv(W_{\Phi_i,[\lambda_i,\mu_i]}^{\mathrm{aff}}.\bb_i)=\conv \left(W_{\prod_{i=1}^\ell  [\lambda_i,\mu_i]}^{\mathrm{aff}}.\bb_C \right),
$$
 that $\overline{F}=\prod_{i=1}^\ell  [\lambda_i,\mu_i]$ and that $W^{\mathrm{aff}}_{[\lambda,\mu]}=W^{\mathrm{aff}}_{F}$.
\end{proof}

\begin{Proposition}\label{p_faces_intersections}
\begin{enumerate}
\item Let $J$ be a non-empty subset of $\ve(\A)$. Then $\bigcap_{\lambda\in J}\overline{\cP(\lambda)}$ is non-empty if and only if there exists an alcove of $(\A,\cH)$ dominating $J$. Moreover, if $F$ is the smallest face dominating $J$ (which exists by Lemma~\ref{l_enclosure_subset_alcove}), then $\bigcap_{\lambda\in J} \overline{\cP(\lambda)}=\conv(W^{\mathrm{aff}}_F.\bb_C)=\overline{\cF_F}=\bigcap_{\lambda\in \ve(F)} \overline{\cP(\lambda)}$, for any alcove $C$ dominating $F$.

\item If $F$ is a face of $(\A,\cH)$, then the set of extreme points of $\overline{\cF}_F$ is $W^{\mathrm{aff}}_F.\bb_C$, for any alcove $C$ dominating $F$. Moreover $\cF_F\cap F=\{x\}$, where $x$ is the barycenter of the extreme points of $\cF_F$.

\item Let $F$ be a face of $(\A,\cH)$ and $x\in \cF_F$. Then $\{\mu\in \ve(\A)\mid x\in \overline{\cP(\mu)}\}=\ve(F)$.

\end{enumerate}
\end{Proposition}

\begin{proof}
1) Let $J \subset \ve(\A)$ be non-empty. Assume  that $\bigcap_{\lambda\in J} \overline{\cP(\lambda)}$ is non-empty. Let $x\in \bigcap_{\lambda\in J}\overline{\cP(\lambda)}$ and let $F$ be the face of $(\A,\cH)$ containing $x$. Then by Lemma~\ref{l_inclusion_P(lambda)_union_alcoves}, $F$ dominates $J$ and, in particular, $J$ is dominated by any alcove dominating $F$. Conversely, if $J$ is dominated by a face $F$ of $(\A,\cH)$, then $W^{\mathrm{aff}}_F.\bb_C\subset \bigcap_{\lambda\in J} \overline{\cP(\lambda)}$, for any alcove $C$ dominating $F$.

Let $C$ be an alcove of $(\A,\cH)$ and $J$ be a subset of $\ve(C)$ with cardinality at least $2$. Let $F$ be the smallest face of $C$ dominating $J$, which exists by Lemma~\ref{l_enclosure_subset_alcove}.  Let $\lambda\in J$. Then \[\overline{\cF_j}:=\bigcap_{\mu\in J} \overline{\cP(\mu)}=\bigcap_{\mu\in J\setminus\{\lambda\}}(\overline{\cP(\lambda)}\cap \overline{\cP(\mu)}).\]
 By Lemma~\ref{l_intersection_neighbour_polytopes} and \cite[Theorem 1.10]{bruns2009polytopes}, $\overline{\cF_J}$ is a face of $\overline{\cP(\lambda)}$. \\

We have:   \begin{align*}
  \Extr(\cF_J)&=\bigcap_{\mu\in J\setminus\{\lambda\}}  \Extr(\overline{\cP(\lambda)}\cap \overline{\cP(\mu)}) &\text{ by Lemma~\ref{l_extreme_points_faces_polytope}},\\
  &=\bigcap_{\mu\in J\setminus\{\lambda\}} W^{\mathrm{aff}}_{[\lambda,\mu]}.\bb_C &\text{ by Lemma~\ref{l_extreme_points_P},}\\
  &=(\bigcap_{\mu\in J\setminus\{\lambda\}} W^{\mathrm{aff}}_{[\lambda,\mu]}).\bb_C&\text{ since }\bb_C\text{ is regular,}\\
  &=W^{\mathrm{aff}}_F.\bb_C &\text{by Lemma~\ref{l_fixator_subset_alcove}}. 
 \end{align*}
 Thus by Krein-Milman theorem, $\cF_J=\conv(W^{\mathrm{aff}}_F.\bb_C)$, which proves statement (1).

2) Let $F\in \sF(\cH)$ and $C\in \Alc(\cH)$, dominating $F$. By 1) and by Lemma \ref{l_extreme_points_P}, the set of extreme points of $\overline{\cF}_F$ is $W^{\mathrm{aff}}_F.\bb_C$. Let now $x=\frac{1}{|W^{\mathrm{aff}}_F|}\sum_{w\in W^{\mathrm{aff}}_F}w.\bb_C$. Then $x\in F$ according to Lemma~\ref{l_barycenter_orbit}. By Proposition~\ref{p_description_faces_affine}, the direction of the support of $F\cap \cF_F$  is a singleton, which proves that $\cF_F\cap F=\{x\}$ and completes the proof of statement (2).

3) Let $F\in \sF(\cH)$ and $x\in \cF_F$. Let $\mu\in \ve(\A)\setminus \ve(F)$. If $\ve(F)\cup \{\mu \}$ is not dominated by any alcove, then $\cF_F\cap \overline{\cP(\mu)}=\emptyset$ according to statement (1) and thus $x\notin \overline{\cP(\mu)}$. Assume that $\mu\cup\ve(F)$ is dominated by an alcove $C$. Then it follows from (1) that $\overline{\cF_F}\cap \overline{\cP(\mu)}=\overline{\cF_{F'}}$, where $F'$ is the smallest face dominating $F\cup \{\mu\}$, and by (2) that $\cF_{F'}\neq \cF_F$. Therefore $\cF_{F'}$ is a proper face of $\cF_F$ and by \cite[Corollary 1.7]{bruns2009polytopes}, $x\notin \cF_{F'}$, which proves that $x\notin \overline{\cP(\mu)}$. Therefore $\{\mu\in \ve(\A)\mid x\in \overline{\cP(\mu)}\}\subset \ve(F)$, and this is actually an equality by 1), which proves statement (3).

\end{proof}

\subsection{Tessellation of $\A$ by the affine weight polytopes}\label{ss_tessellation_A}

We now prove that $\A$ is tessellated by the $\overline{\cP(\lambda)}$, for $\lambda\in \ve(\A)$, as stated in the introduction in Theorem \ref{t_tessellation_intro}.

\begin{Theorem}\label{t_tessellation_A} (see Figures~\ref{f_irregular_tessellation_A2}, \ref{f_regular_tessellation_A2}, \ref{f_tessellation_B2} and \ref{f_tessellation_G2})

Let $\A$ be a finite dimensional vector space and $\Phi$ be a finite root system in $\A$. Let $\cH=\{\alpha^{-1}(\{k\})\mid \alpha \in \Phi,\ k\in \Z \}$, let $\sF$ be the set of faces of $(\A,\cH)$ and let $\ve(\A)$ be the set of vertices of $(\A,\cH)$. We choose a fundamental alcove  $C_0$  of $(\A,\cH)$ and choose $\bb=\bb_{C_0}\in C_0$. If $\lambda$ is a vertex of $(\A,\cH)$, we choose an alcove $C_\lambda$ dominating $\lambda$ and we set $\bb_{C_\lambda}=w.\bb$ where $w$ is the element of the affine Weyl group $W^{\mathrm{aff}}$ of $(\A,\cH)$ sending  $C_0$ to $C$. If $F$ is a face of $(\A,\cH)$, we  write $W_F^{\mathrm{aff}}:=\mathrm{Fix}_{W^{\mathrm{aff}}}(F)$ and we set $\cF_F=\In_r(\conv(W^{\mathrm{aff}}_F.\bb_C))$, for any alcove $C$ dominating $F$. Then $\bigsqcup_{F\in \sF} \cF_F=\A$. In particular, if we set $\overline{\cP(\lambda)}=\overline{\cF_{\{\lambda\}}}$, for $\lambda\in \ve(\A)$, then $\A=\bigcup_{\lambda\in \ve(\A)}\overline{\cP(\lambda)}$.
\end{Theorem}
 \begin{proof}

Let $E=\bigcup_{\lambda\in \ve(\A)} \overline{\cP(\lambda)}$. We prove that $E=\A$ by proving that $E$ is closed and dense in $\A$.  Take $(z_n)\in E^\N$ a converging sequence. Then $(z_n)$ is bounded and by Lemma~\ref{l_inclusion_P(lambda)_union_alcoves} (2), there exists a finite subset $F\subset \ve(\A)$ such that $\{z_n\mid n\in \N\}\subset \bigcup_{\mu\in F} \overline{\cP(\mu)}$. Then as the $\overline{\cP(\mu)}$ are closed, $\lim z_n\in \bigcup_{\mu\in F} \overline{\cP(\mu)}\subset E$. Therefore $E$ is closed.

 Suppose that $E\neq \A$  and fix $x\in \A\setminus E$. Let $\lambda\in \ve(\A)$ and $y\in \cP(\lambda)$. Let $\gamma:[0,1]\rightarrow [x,y]$ be the affine parametrization of $[x,y]$ such that $\gamma(0)=y$ and $\gamma(1)=x$. Let \[t=\inf\{t'\in [0,1]\mid \gamma(t')\notin E\}.\]

As $E$ is closed,  $\gamma(t)\in E$. Let $\nu\in \ve(\A)$ be such that $\gamma(t)\in \overline{\cP(\nu)}$. Then $\gamma(t)\in \Fr(\cP(\nu))$ and thus there exists a facet $\overline{\cF}$ of $\overline{\cP(\nu)}$ containing $\gamma(t)$. By Propositions~\ref{p_description_faces_affine} and \ref{p_faces_intersections}, there exists $\nu'\in \ve(\A)$, adjacent to $\nu$ and such that $\overline{\cP(\nu)}\cap\overline{\cP(\nu')}=\overline{\cF}$.  If $\gamma(t)\in \cF$, then by Lemma~\ref{l_local_description_polytope}, there exists $r\in \R_{>0}$ such that \[B(\gamma(t),r)\cap\overline{\cP(\nu)}=D\cap B(\gamma(t),r)\text{ and }B(\gamma(t),r)\cap\overline{\cP(\nu')}=D'\cap B(\gamma(t),r),\] where $D$ and $D'$ are the two closed half-spaces delimited by $\supp(\cF)$. But then $B(\gamma(t),r)\subset\overline{\cP(\nu)}\cup\overline{\cP(\nu')}$, which is impossible by definition of $t$. Therefore \[\gamma(t)\in \overline{\cF}\setminus \cF.\]

The facet $\cF$ is a polytope of $\supp(\cF)$ and thus its frontier is contained in a union $\bigcup_{V\in F(\nu,\nu')} V$, where $F(\nu,\nu')$ is a finite set of affine subspaces of dimension at most $\dim \A-2$.  For $\nu,\nu'\in \ve(\A)$, set $F(\nu,\nu')=\emptyset$ if $\nu$ and $\nu'$ are not adjacent.  Then we proved: \[\A\setminus E\subset \bigcup_{\nu,\nu'\in \ve(\A)} \bigcup_{V\in F(\nu,\nu')} \supp(V\cup\{x\}).\] As the right hand side of this formula is a countable union of proper affine subspaces of $\A$, it has Lebesgue measure $0$. In particular, $E$ is dense in $\A$. But as $E$ is closed, $E=\A$. Thus we proved \[\A=\bigcup_{\lambda\in \ve(\A)}\overline{\cP(\lambda)}.\] 
As for $\lambda\in \ve(\A)$, $\overline{\cP(\lambda)}$ is the union of the $\cF_F$, where $F$ runs over the faces of $(\A,\cH)$ dominating $\lambda$ (see Proposition~\ref{p_description_faces_affine}), we deduce that $\A=\bigcup_{F\in\sF} \cF_{F}$. Let $F_1,F_2\in \sF$ be such that $\cF_{F_1}\cap \cF_{F_2}$ is non-empty. Then $\overline{\cF_{F_1}}\cap \overline{\cF_{F_2}}\neq \emptyset$ and thus by Proposition~\ref{p_faces_intersections}, there exists an alcove $C$ dominating $\ve(F_1)\cup\ve( F_2)$, and $\cF_{F_1}, \cF_{F_2}$ are faces  of $\overline{\cP(\lambda)}$, for any $\lambda\in \ve(C)$, by Proposition~\ref{p_description_faces_affine} 3). As the faces of $\overline{\cP(\lambda)}$ form a partition $\overline{\cP(\lambda)}$ (by \cite[Theorem 1.10]{bruns2009polytopes}), we deduce that $F_1=F_2$, which concludes the proof of the theorem.
\end{proof}

 \begin{figure}[h]
 \centering
 \includegraphics[scale=0.2]{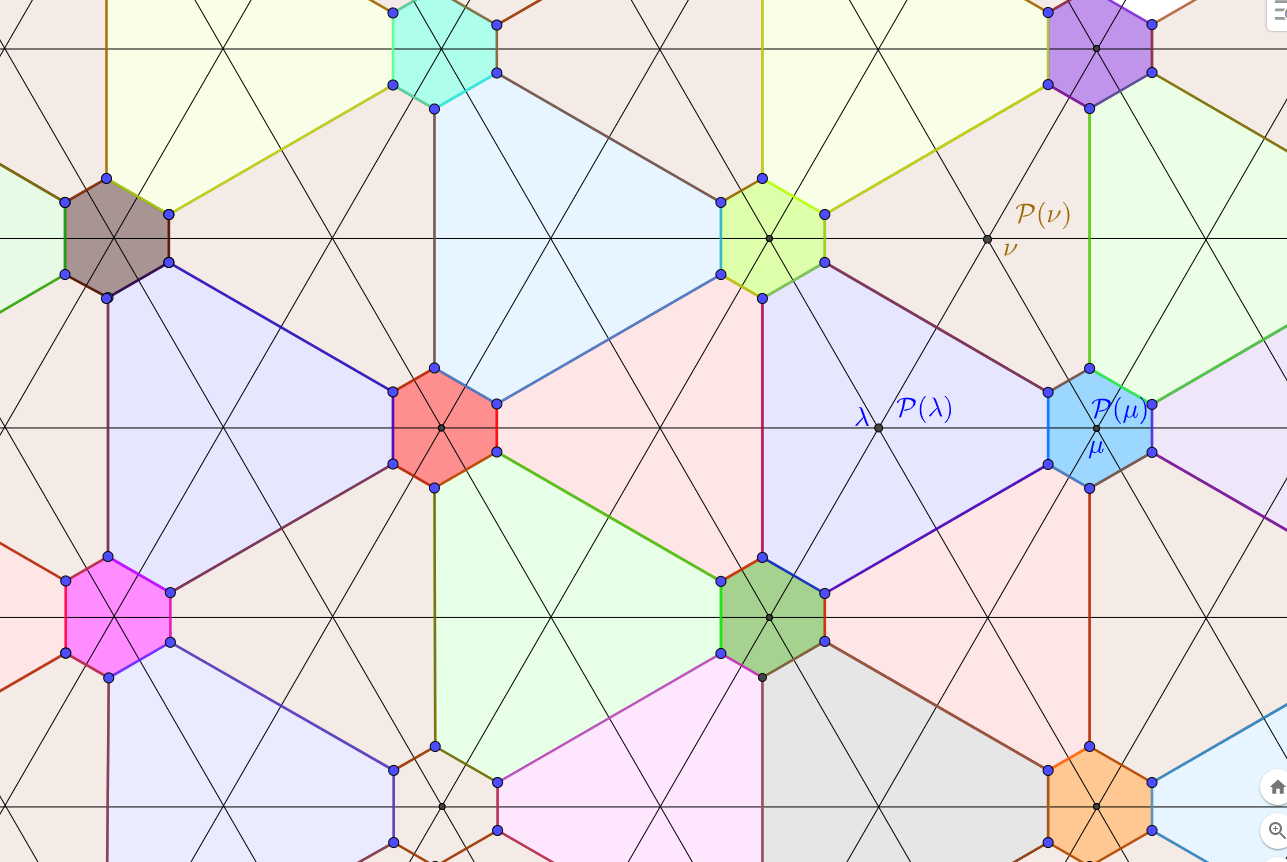}
 \caption{Illustration of Theorem~\ref{t_tessellation_A}: the case of $\mathrm{A}_2$ when $\bb$ is not the barycenter of the fundamental alcove.}\label{f_irregular_tessellation_A2}
\end{figure}

 \begin{figure}[h]
 \centering
 \includegraphics[scale=0.2]{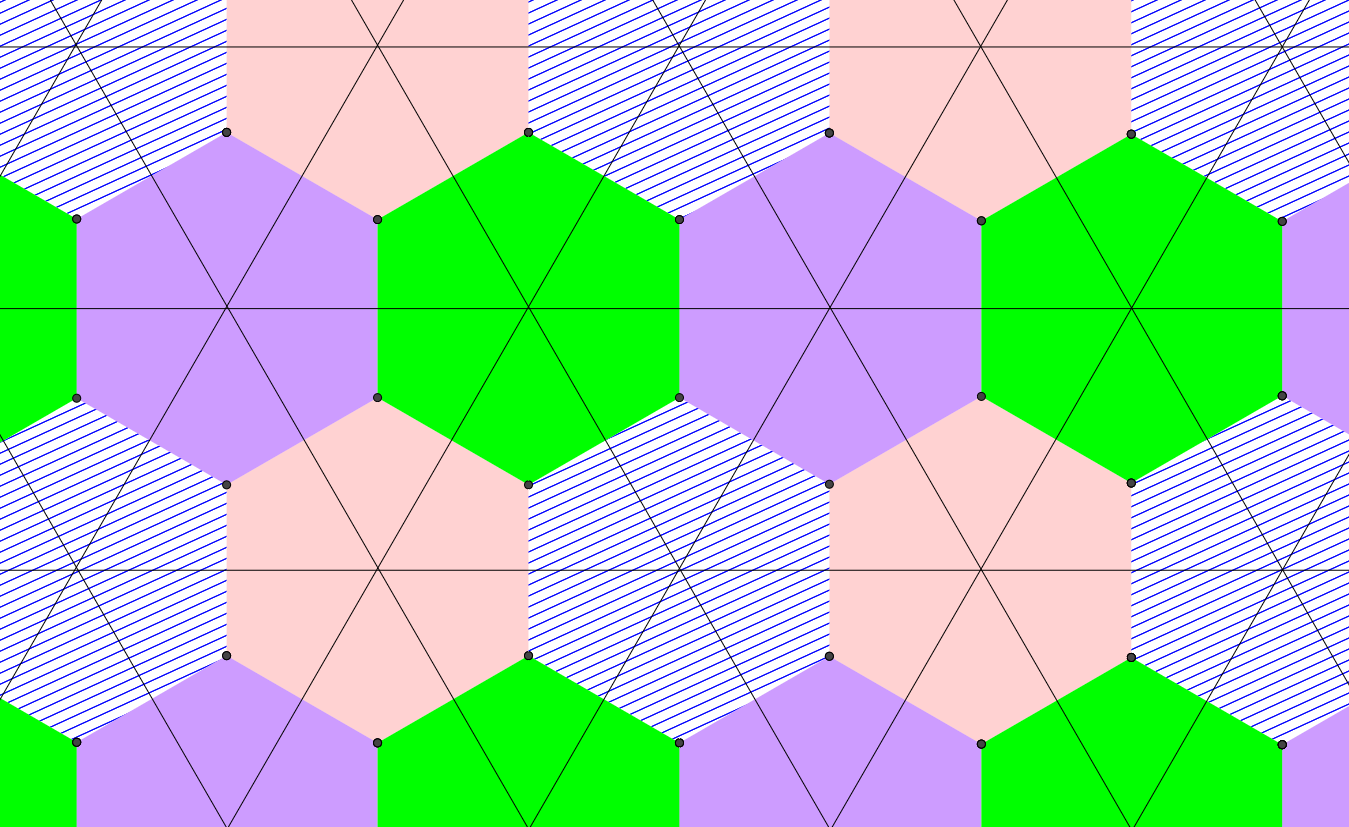}
 \caption{Illustration of Theorem~\ref{t_tessellation_A}: the case of $\mathrm{A}_2$ when $\bb$ is the barycenter of the fundamental alcove.}\label{f_regular_tessellation_A2}
\end{figure}

 \begin{figure}[h]
 \centering
 \includegraphics[scale=0.2]{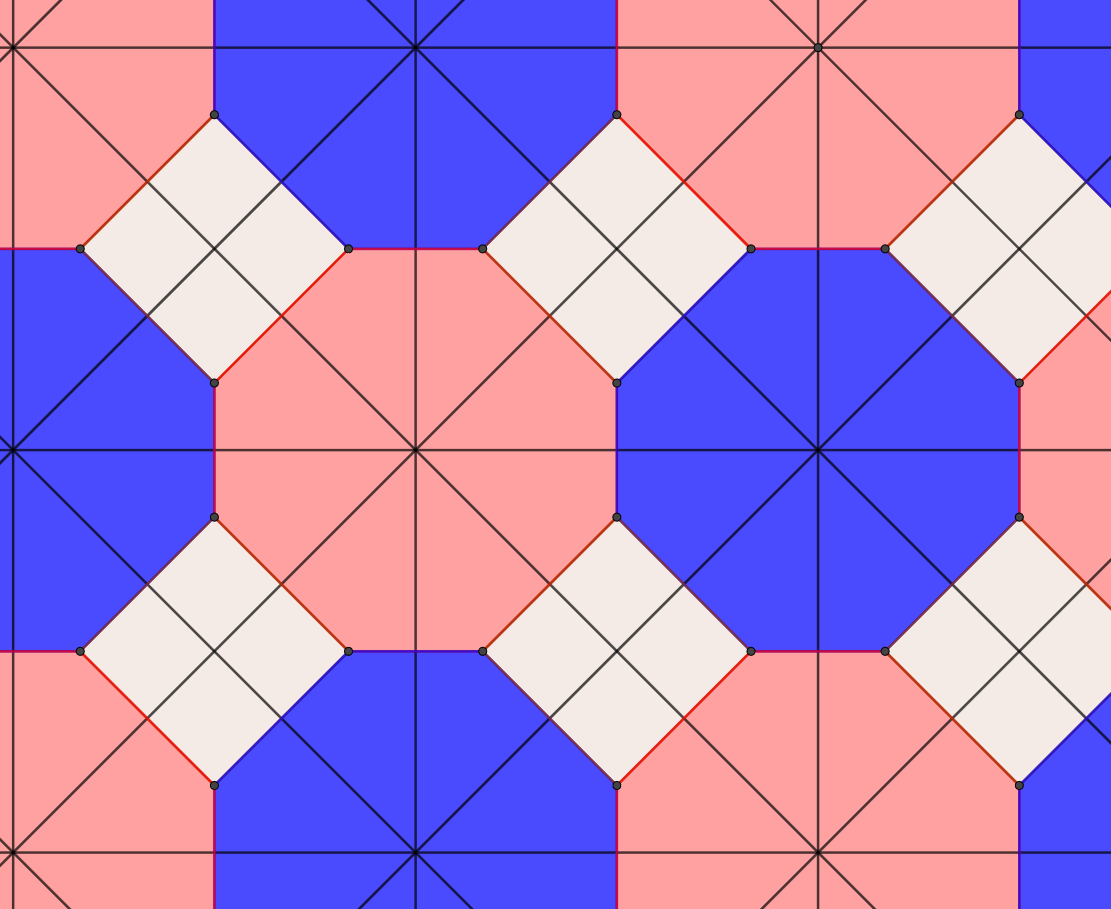}
 \caption{Illustration of Theorem~\ref{t_tessellation_A}: the case of $\mathrm{B}_2$.}\label{f_tessellation_B2}
\end{figure}

 \begin{figure}[h]
 \centering
 \includegraphics[scale=0.2]{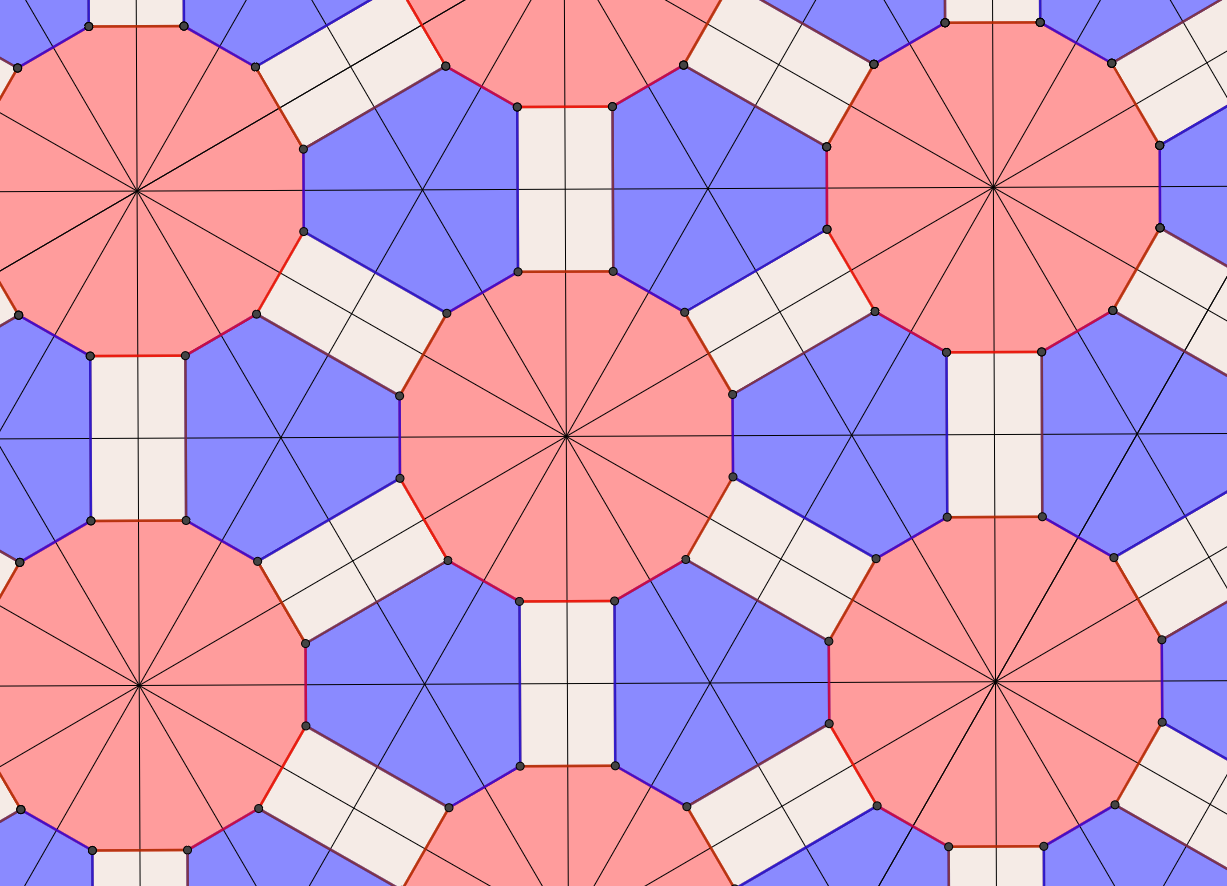}
 \caption{Illustration of Theorem~\ref{t_tessellation_A}: the case of $\mathrm{G}_2$.}\label{f_tessellation_G2}
\end{figure}

\section{Thickened tessellation of $\A$}\label{s_thickened_tessellation}

The aim of this section is as follows. Take $\eta\in ]0,1[$. For $\lambda\in \ve(\A)$, denote by $h_{\lambda,\eta}: \A \to \A, h_{\lambda,\eta}(x)=\lambda+\eta (x-\lambda)$\index{h@$h_{\lambda,\eta}$} the homothety centered at $\lambda$ with ratio $\eta$. Let $x\in h_{\lambda,\eta}(\overline{\cP(\lambda)})$. Write $x=h_{\lambda,\eta}(x')$, with $x' \in \overline{\cP(\lambda)}$.  Let $F$ be the face of $(\A,\cH)$ such that $x'\in \cF_F$. Let $\overline{\tilde{F}_\eta(x)}=x+(1-\eta)(\overline{F}-\lambda)$. This is a copy of $F$ which contains exactly one point of $h_{\mu,\eta}(\overline{\cP(\mu)})$, for each vertex $\mu$ of $F$ (and which is the convex hull of these points).
The main goal of this section is to prove Theorem~\ref{t_thickened_tessellation}, which
 states that $\A=\bigcup_{\lambda\in \ve(\A),x\in h_{\lambda,\eta}(\overline{\cP(\lambda)})} \overline{\tilde{F}_\eta(x)}$ (see Figures~\ref{f_explanation_t_thickened},   \ref{f_explanation_t_thickened_A2} and \ref{f_explanation_t_thickened_G2}). 
Since for all $\mu\in \ve(F)$, we have $\overline{\tilde{F}_\eta(h_{\mu,\eta}(x'))}=\overline{\tilde{F}_\eta(x)}$, in order to describe $\A$ as a disjoint union, we will introduce a new notation by replacing $\overline{\tilde{F}_\eta(x)}$ by $\overline{F_\eta(x')}$ (see Definition~\ref{d_Feta}) We will then have $\A=\bigsqcup_{x'\in \A} \overline{F_\eta(x')}$.

We divide this section in two parts.
Firstly,
in subsection~\ref{ss_projection_polytope_neighbour_polytope}, we prove that if $\lambda$ and $\mu$ are two adjacent vertices of $\A$, then the projection of $\overline{\cP(\mu)}$ on $\overline{\cP(\lambda)}$ is contained in $\overline{\cP(\lambda)}\cap \overline{\cP(\mu)}$ (see Proposition~\ref{p_projection_polytope_neighbour}).
Then,
in subsection~\ref{ss_thickened_tessellation}, we prove Theorem~\ref{t_thickened_tessellation}.

 \begin{figure}[h]
 \centering
 \includegraphics[scale=0.35]{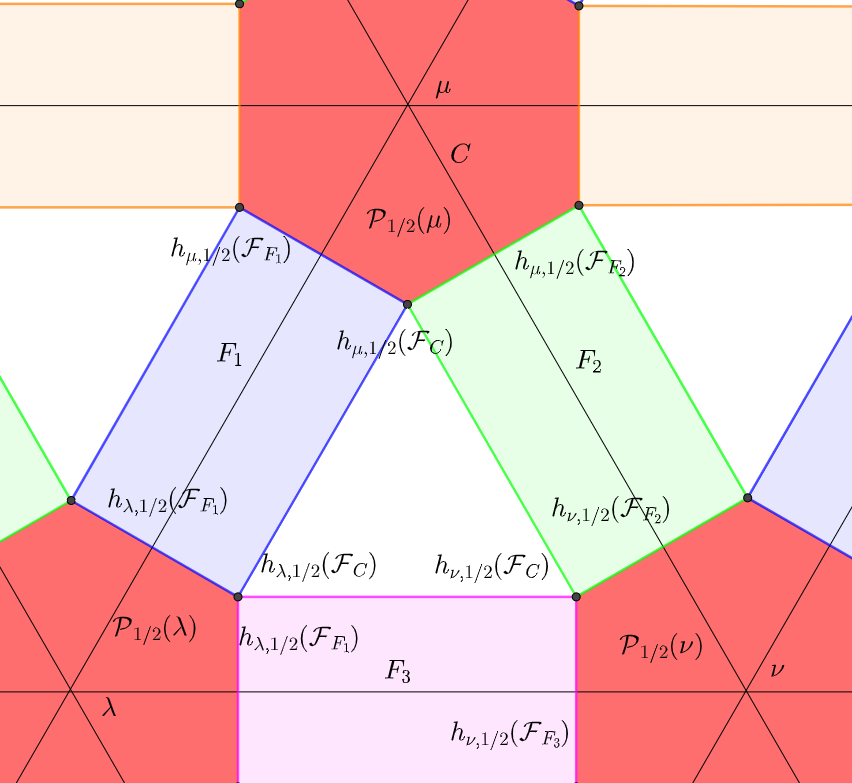}
 \caption{Illustration of Theorem~\ref{t_thickened_tessellation} in the case of $\mathrm{A}_2$. The represented alcove $C$ is the equilateral triangle delimited by the black segments. The hexagons $\cP_{1/2}(\lambda), \cP_{1/2}(\mu)$ and $\cP_{1/2}(\nu)$ are represented in red. The edges of $C$ are the segments $F_1,F_2$ and $F_3$ and the vertices of $C$ are $\lambda,\mu$ and $\nu$.  The blue rectangle is the union of the $\overline{\tilde{F}_{1/2}(x)}$ for $x\in h_{\lambda,1/2}(\cF_{F_1})$ (or equivalently $x\in h_{\mu,1/2}(\cF_{F_1})$). Similarly, the green and pink rectangles correspond to the 
 $\overline{\tilde{F}_{1/2}(x)}$ for $x$ in $h_{\mu,1/2}(\cF_{F_2})$ and $h_{\nu,1/2}(\cF_{F_2})$ respectively. The face $\overline{\cF_C}$ is the (non represented) vertex $\overline{\cP(\lambda)}\cap \overline{\cP(\mu)}\cap \overline{\cP(\nu)}$.  The central white  triangle is  $\overline{\tilde{F}_{1/2}}(z)$, for $z\in h_{\lambda,1/2}(\cF_{C})$ (or equivalently, $z\in h_{\mu,1/2}(\cF_{C})$ or $z \in h_{\nu,1/2}(\cF_{C})$).                   }
 \label{f_explanation_t_thickened}
\end{figure}

\subsection{Projection of a polytope on a neighbour polytope}\label{ss_projection_polytope_neighbour_polytope}

If $F$ is a face of $(\A,\cH)$ and $\lambda\in \ve(F)$, we denote by $\ve_\lambda(F)$\index{v@$\ve_\lambda(F)$} the set of vertices of $F$ which are adjacent  to $\lambda$ (and different from $\lambda$). If $\lambda\in \ve(\A)$, we denote by $\pi_{\lambda}$\index{p@$\pi_\lambda$} instead of $\pi_{\overline{\cP(\lambda)}}$ the orthogonal projection on $\overline{\cP(\lambda)}$. 

In this subsubsection, we prove that if $\lambda$ and $\mu$ are two vertices of $\A$ contained in a common closed alcove, then $\pi_{\lambda}(\overline{\cP(\mu)})\subset \overline{\cP(\lambda)}\cap \overline{\cP(\mu)}$ (see Proposition~\ref{p_projection_polytope_neighbour}).

\begin{Lemma}\label{l_adjacent_vertices}
\begin{enumerate}
    \item Let $F$ be a face of $(\A,\cH)$ and $\lambda\in \ve(F)$.  Then: \[\R_{>0}(F-\lambda)=\sum_{\mu\in \ve_\lambda(F)}\R_{>0}(\mu-\lambda)=\sum_{\mu\in \ve(F)}\R_{>0}(\mu-\lambda).\]

    \item Let $\lambda\in \ve(\A)$ and  $F,F'$ be faces of $(\A,\cH)$ dominating $\lambda$. Then $\ve_\lambda(F)\subset \ve_\lambda(F')$ if and only if $F$ is dominated by $F'$. 
\end{enumerate}
\end{Lemma}

\begin{proof}
1) The map $\mu\mapsto ]\lambda,\mu[$  is a bijection from $\ve_\lambda(F)$ to the set of one dimensional faces of $F$ dominating $\lambda$. By Proposition~\ref{p_bourbaki_vectorial_faces} and Proposition~\ref{p_bijection_affine_faces_vectorial_faces_vertices}, $\R_{>0}(F-\lambda)$ is the sum of its one-dimensional faces and thus we have $\R_{>0}(F-\lambda)=\sum_{\mu \in\ve_\lambda(F)} \R_{>0}(\mu-\lambda)$. Using the notation of Eq. \eqref{e_decomposition_Phi} (with $\ell=1$ if $\Phi$ is irreducible), we write $F=F_1\times \ldots \times F_\ell$ and $\lambda=(\lambda_1,\ldots,\lambda_\ell)$. We have  $\ve_\lambda(F)=\bigsqcup_{i=1}^\ell \{(\lambda_1,\ldots,\lambda_{i-1})\}\times \ve_{\lambda_i}(F_i)\times \{(\lambda_{i+1},\ldots,\lambda_\ell)\}$ and $\ve(F)=\ve(F_1)\times \ldots \times \ve(F_\ell)$ and the second equality of 1) follows.

 2) If $F$ is dominated by $F'$, then $\ve(F)\subset \ve(F')$ and thus $\ve_\lambda(F)\subset \ve_\lambda(F')$. Conversely, if $\ve_\lambda(F)\subset \ve_\lambda(F')$, then $\overline{\R_{>0}(F-\lambda)}=\sum_{\mu\in \ve_\lambda(F)}\R_{\geq 0}(\mu-\lambda)\subset \sum_{\mu\in \ve_\lambda(F')}\R_{\geq 0}(\mu-\lambda)=\overline{\R_{> 0}(F'-\lambda)}$, whence the result  follows. 
\end{proof}

\begin{Lemma}\label{l_description_fibers_pi_affine}
 Let $F$ be a face of $(\A,\cH)$ and $\lambda\in \ve(F)$. Let $x\in \cF_F$. Then $\pi_\lambda^{-1}(\{x\})=x+\R_{\geq 0}(\overline{F}-\lambda)$. 
\end{Lemma}

\begin{proof}
By Proposition~\ref{p_description_faces_affine}, the facets of $\cP(\lambda)$ are the $\cF_{F_1}$ for $F_1\in \sF_{\lambda}(F)$ of dimension $1$.  If $F_1=]\lambda,\mu[$ is such a face, with $\mu\in \ve(\A)$ , then by Proposition~\ref{p_description_faces_affine},  we have $\delta_{\cF_{F_1}}=\R_{\geq 0} (\mu-\lambda)$,  with the notation of subsection~\ref{ss_preliminary_polytopes}.  Using Proposition~\ref{p_description_fibers_pi}, we deduce $\pi_{\lambda}^{-1}(\{x\})=x+\sum_{\mu\in \ve_\lambda(F)}\R_{\geq 0}(\mu-\lambda)$. Then, the result follows from Lemma~\ref{l_adjacent_vertices}.  
\end{proof}

\begin{Lemma}\label{l_intermediate_step_projection_P_mu} 
Let $C$ be an alcove of $(\A,\cH)$ and $F$ be a proper subface of $C$ which is not a vertex. Let $x\in \cF_F$. Let $\lambda,\mu\in\ve(F)$  and $\nu\in \ve(C)\setminus \ve(F)$. Then $\pi_\lambda^{-1}(\{x\})\cap \overline{\cP(\mu)}\cap \overline{\cP(\nu)}=\emptyset$. 
\end{Lemma}

\begin{proof}
Let $F(\mu,\nu)$ be the smallest face of $(\A,\cH)$ dominating both $\mu$ and $\nu$.
Such face exists according to  Lemma~\ref{l_enclosure_subset_alcove}. By Proposition~\ref{p_faces_intersections}, we have:
$$\overline{\cP(\mu)}\cap \overline{\cP(\nu)}=\bigcap_{\nu'\in \ve_\mu(F(\mu,\nu))}\overline{\cP(\mu)}\cap \overline{\cP(\nu')}.$$
 We have $F(\mu,\nu)\not\subseteq \overline{F}$ and thus Lemma~\ref{l_adjacent_vertices} implies the existence of  $\nu'$ belonging to $ \ve_\mu(F(\mu,\nu))\setminus \ve(F)$. Up to replacing $\nu$ by $\nu'$, we may assume that $\mu$ and $\nu$ are adjacent.

Let $\DC(\mu,\nu)$ be the closed half-space delimited by $\overline{\cP(\mu)}\cap \overline{\cP(\nu)}$ and containing $\overline{\cP(\mu)}$. Let $V=x+\sum_{\tau\in \ve_\mu(F)}\R (\tau-\mu)=x+\vect(F-\mu)$. By Lemma~\ref{l_description_fibers_pi_affine}, we have \[\pi_\lambda^{-1}(\{x\})=x+ \R_{\geq 0} (\overline{F}-\lambda)\subset V.\]

Let $C^v=\R_{>0}(C-\mu)$ and $F^v=\R_{>0}(F-\mu)$. Then $C^v$ (resp. $F^v$) is a chamber (resp. face) of $(\A,\cH^0_\mu)$ by Proposition~\ref{p_bijection_affine_faces_vectorial_faces_vertices}. Let $\Sigma$ be the basis of $\Phi_\mu$ associated with $C^v$. Write $\Sigma=\{\alpha_i,i\in I\}$ for some indexing set $I$ satisfying $|I|=|\Sigma|$. Let $(\qp_i)\in (\A^*)^I$ be the dual basis of $(\alpha_i)_{i\in I}$. Let $J=\{i\in I\mid \qp_i\in \overline{F^v}\}$. Then $\R_{>0}(\nu-\mu)$ is a one dimensional face of $C^v$ and thus $\nu-\mu\in \R_{>0}\qp_i$, for some $i\in I\setminus J$. Then $\DC(\mu,\nu)-\mu$ is the half-space of $\A$ containing $\cP^0_{\Phi_\mu,\bb_C-\mu}$ and delimited by the support of  $W^v_{\R_{>0}\qp_i}.(\bb_C-\mu)$, where $W^v_{\R_{>0}\qp_i}$ is the fixator of $\R_{>0}\qp_i$ in $W^v_{\Phi_\mu}$. This half-space is denoted $\DC_i$ in Lemma~\ref{l_faces_affine_subspace}.
Therefore by Lemma~\ref{l_faces_affine_subspace}, we have $\mathring{\DC}(\mu,\nu)-\mu\supset (V-\mu)\cap (\overline{\cP(\mu)}-\mu)$ and hence $\mathring{\DC}(\mu,\nu)\supset V\cap \overline{\cP(\mu)}\supset \pi_\lambda^{-1}(\{x\})\cap \overline{\cP(\mu)}$, which proves the lemma, since $\mathring{\DC}(\mu,\nu)\cap \overline{\cP(\nu)}$ is empty. 
\end{proof}

\begin{Lemma}\label{l_projection_contained_face_affine}
Let $\lambda\in \ve(\A)$ and $F$ be a face of $(\A,\cH)$ dominating $\lambda$. Let $x\in F$. Then $\pi_\lambda(x)\in F$. 
\end{Lemma}

\begin{proof}
We have $x-\lambda\in \R_{>0}(F-\lambda)$. Using Lemma~\ref{l_projection_contained_face},  Eq. \eqref{e_affine_Weyl_polytope_2} and Proposition~\ref{p_bijection_affine_faces_vectorial_faces_vertices} we deduce $\pi_{\overline{\cP(\lambda)}-\lambda}(x-\lambda)=\pi_{\lambda}(x)-\lambda\in \R_{>0}(F-\lambda)$. Consequently, $\pi_{\lambda}(x)\in \lambda+\R_{>0}(F-\lambda)$.

Let $F'$ be the face of $(\A,\cH)$ containing $\pi_\lambda(x)$. By Lemma~\ref{l_inclusion_P(lambda)_union_alcoves}, $F'$ dominates $\lambda$. Therefore $\pi_\lambda(x)\in F'\subset \lambda+\R_{>0} (F'-\lambda)$. We deduce $\R_{>0}(F-\lambda)\cap \R_{>0}(F'-\lambda)\neq \emptyset$. As  the faces of $(\A,\cH_\lambda^0)$ partition $\A$, we deduce $\R_{>0}(F-\lambda)=\lambda+\R_{>0}(F'-\lambda)$ and thus $F=F'$, by Proposition~\ref{p_bijection_affine_faces_vectorial_faces_vertices}. 
\end{proof}

\begin{Lemma}\label{l_intermediate_step_projection_P_mu_2}
Let $\lambda\in \ve(\A)$, $C$ be an alcove of $\A$ dominating $\lambda$ and $\mu\in \ve(C)\setminus\{\lambda\}$. Then $\pi_\lambda(\overline{\cP(\mu)}\cap \overline{C})\subset \overline{\cP(\lambda)}\cap \overline{\cP(\mu)}$.
\end{Lemma}

\begin{proof}
Let $x\in \Fr(\cP(\lambda))\cap \overline{C}$.   Let $F=F(x)$ be the face dominating $\lambda$ and such that $x\in \cF_F$. Then by Proposition~\ref{p_faces_intersections}, $\{\nu\in \ve(\A)\mid x\in \overline{\cP(\nu)}\}=\ve(F)$.
By contradiction, 
we make the following assumption: \begin{equation}\label{e_absurd_assumption}
\exists \mu'\in \ve(C)\setminus\ve(F) \mid (\pi_\lambda)^{-1}(\{x\})\cap \overline{\cP(\mu')}\cap \overline{C}\neq \emptyset.
\end{equation}

Let $\mu'\in \ve(C)\setminus \ve(F)$ be  such that there exists $y\in \pi_{\lambda}^{-1}(\{x\})\cap \overline{\cP(\mu')}\cap \overline{C}$. By Proposition~\ref{p_faces_intersections}, $x\notin \overline{\cP(\mu')}$ and thus $x\neq y$. 
Let $\gamma:[0,1]\rightarrow [x,y]$ be the affine parametrization such that $\gamma(0)=x$ and $\gamma(1)=y$. 
For $t\in ]0,1]$, we have $|y-\gamma(t)|<|y-x|=|y-\pi_\lambda(y)|$ and thus $\gamma(t)\notin \overline{\cP(\lambda)}$ for $t\in ]0,1]$. 

Let $E=\{t\in [0,1]\mid \exists \mu''\in \ve(C)\setminus \ve(F), \gamma(t)\in \overline{\cP(\mu'')}\}$. Then $E$ is non-empty by Eq. \eqref{e_absurd_assumption}. Let $t_0=\inf E$. Then $\gamma(t_0)\in \overline{\cP(\mu'')}$ for some $\mu''\in \ve(C)\setminus \ve(F)$, since the $\overline{\cP(\tau)}$ are closed, for  $\tau\in \ve(\A)$. Therefore $t_0=\min E$ and $t_0>0$. By Theorem~\ref{t_tessellation_A} and Lemma~\ref{l_inclusion_P(lambda)_union_alcoves}, we have $\gamma([0,t_0[)\subset \bigcup_{\nu\in \ve(F)} \overline{\cP(\nu)}$. Let $\nu\in \ve(F)$ be such that for some $\epsilon>0$, we have: \[\gamma(t)\in \overline{\cP(\nu)},\ \forall t\in ]t_0-\epsilon,t_0[.\]

Then $\gamma(t_0)\in \overline{\cP(\nu)}\cap \overline{\cP(\mu')}$. By Lemma~\ref{l_description_fibers_pi_affine}, $\pi_\lambda^{-1}(\{x\})$ is convex and contains $\gamma(t_0)$. This contradicts Lemma~\ref{l_intermediate_step_projection_P_mu}. Therefore, Eq. \eqref{e_absurd_assumption} is wrong. In other words, \begin{equation}\label{e_converse_absurd_assumption}
\forall x\in \overline{C}\cap \overline{\cP(\lambda)}, (\pi_\lambda)^{-1}(\{x\})\cap \overline{C}\subset \overline{C}\setminus \bigcup_{\mu'\in \ve(C)\setminus \ve(F(x))} \overline{\cP(\mu')}.
 \end{equation}

Let now $y\in \overline{\cP(\mu)}\cap \overline{C}$. Then by Lemma~\ref{l_projection_contained_face_affine}, $x:=\pi_\lambda(y)\in \overline{C}$. Let $F\in \sF_\lambda$ be such that $x\in \cF_F$.  Then by Eq. \eqref{e_converse_absurd_assumption}, $y\in \overline{C}\setminus \bigcup_{\mu'\in \ve(C)\setminus \ve(F)}\overline{\cP(\mu')}$ and thus $\mu\in \ve(F)$. Therefore $\overline{\cP(\mu)}\supset \overline{\cF_F}$, which completes the proof of the lemma. 
\end{proof}

\begin{Lemma}\label{l_alcoves_containing_vertex}
Let $C_1\in \Alc(\cH)$ and   $\lambda,\mu\in \ve(C_1)$. Then: \[\bigcup_{C\in \Alc_\mu(\cH)} \overline{C}\subset \bigcup_{C\in \Alc_{\lambda,\mu}(\cH)} \lambda+\R_{\geq 0}(\overline{C}-\lambda),\] where $\Alc_{\lambda,\mu}(\cH)$ is the set of alcoves of $(\A,\cH)$ dominating $\lambda$ and $\mu$.
\end{Lemma}

\begin{proof}
Let $E=\bigcup_{F\in \sF_{\lambda,\mu}}\R_{>0} (F-\lambda)$, where $\sF_{\lambda,\mu}$ is the set of faces dominating $\lambda$ and $\mu$. Then $\lambda+ \overline{E}\subset \lambda+\bigcup_{C\in \Alc_{\lambda,\mu}(\cH)}\R_{\geq 0}(\overline{C}-\lambda)$.  By Proposition~\ref{p_bijection_affine_faces_vectorial_faces_vertices} and Lemma~\ref{l_union_chambers_dominating_face}, there exists $X\subset \Phi_\lambda$ such that $E=\bigcap_{\alpha\in X} \alpha^{-1}(\R_{>0})$. As $E$ is a cone containing $]0,\mu-\lambda[$, it contains $\mu-\lambda$.  Therefore $\lambda+E$ is open and contains $\mu$. Consequently, $\lambda+E$ contains a point of each alcove dominating $\mu$. Moreover $\lambda+\overline{E}=\bigcap_{\alpha\in X} \alpha^{-1}(\R_{\geq \alpha(\lambda)})$ and thus it is enclosed, which proves that it contains $\bigcup_{C\in \Alc_\mu(\cH)} \overline{C}$.
\end{proof}

\begin{Lemma}\label{l_conical_property_alcove}
Let $C\in \Alc(\cH)$ and $\lambda\in \ve(C)$. Let $E'$ be the union of the faces of $C$ dominating $\lambda$ and $x\in E'$. Let $y\in \lambda+\R_{\geq 0}(\overline{C}-\lambda)$. Then $[x,y]\cap E'$ contains a neighbourhood of $x$ in $[x,y]$. 
\end{Lemma}

\begin{proof}
Let $E$ be the union of the faces of $(\A,\cH)$ dominating $\lambda$. Then $E$ is open by Lemma~\ref{l_union_alcoves_containing_face}. Therefore there exists a face $F$ dominating $\lambda$ and $\epsilon>0$ such that $F\supset ]x,x+\epsilon(y-x)[$. We also have $\lambda+\R_{>0}(F-\lambda)\supset F$. Moreover $x,y \in\lambda+\R_{\geq 0} (\overline{C}-\lambda)$  and thus it contains $]x,x+\epsilon(y-x)[$. The closed chamber $\overline{C^v}:=\R_{\geq 0}(\overline{C}-\lambda)$ is the union of the faces of $(\A,\cH^0_\lambda)$ dominated by $\overline{C^v}$. Thus by Proposition~\ref{p_bijection_affine_faces_vectorial_faces_vertices}, we have \[\overline{C^v}=\bigsqcup_{F_1\in \sF_\lambda(C)} \R_{>0}(F_1-\lambda),\] where $\sF_\lambda(C)$ is the set of subfaces of $C$ dominating $\lambda$. Therefore there exists $F_1\in \sF_\lambda(C)$ such that $\R_{>0}(F-\lambda)\cap \R_{>0}(F_1-\lambda)$ is non-empty, and as the faces of $(\A,\cH^0_\lambda)$ partition $\A$, we deduce that $F=F_1$. But $F_1\subset E'$ and the lemma follows.
\end{proof}

\begin{Proposition}\label{p_projection_polytope_neighbour}
Let $\lambda,\mu\in \ve(\A)$ be such that there exists an alcove of $(\A,\cH)$ dominating $\lambda$ and $\mu$. Then $\pi_{\lambda}(\overline{\cP(\mu)})\subset \overline{\cP(\lambda)}\cap \overline{\cP(\mu)}$.
\end{Proposition}

\begin{proof}
Let $y\in \overline{\cP(\mu)}$. By Lemma~\ref{l_inclusion_P(lambda)_union_alcoves} and  Lemma~\ref{l_alcoves_containing_vertex}, there exists an alcove $C$ of $(\A,\cH)$ dominating $\lambda
$ and $\mu$ and  such that $y\in \lambda+\R_{\geq 0}(\overline{C}-\lambda)$. Let $\cQ=\overline{C}\cap \overline{\cP(\lambda)}\cap \overline{\cP(\mu)}$. Let $x=\pi_{\cQ}(y)\in \cQ$.  By Lemmas~\ref{l_inclusion_P(lambda)_union_alcoves} and \ref{l_conical_property_alcove}, there exists $z\in ]x,y]\cap \overline{C}$. By Eq. \eqref{e_characterization_projection}, we have $\pi_{\cQ}(z)=\pi_{\cQ}(y)=x$. 

By Lemmas~\ref{l_projection_contained_face_affine} and \ref{l_intermediate_step_projection_P_mu_2}, we have $\pi_\lambda(z)\in \cQ$ and by definition $|a-z|\geq |a-\pi_\lambda(z)|$ for all $a\in \overline{\cP(\lambda)}$. As $\cQ\subset \overline{\cP(\lambda)}$, we deduce $\pi_\lambda(z)=\pi_{\cQ}(z)$. By Eq. \eqref{e_characterization_projection}, we have $\pi_\lambda(z)=\pi_\lambda(y)$. Therefore $\pi_\lambda(y)\in \overline{\cP(\lambda)}\cap \overline{\cP(\mu)}$, which proves the proposition.
\end{proof}

\subsection{Tessellation of apartments by thickened polytopes}\label{ss_thickened_tessellation}

Recall that, for $\eta\in [0,1]$ and $\lambda\in \ve(\A)$, we have the homothety $h_{\lambda,\eta}:\A\rightarrow \A$ given by $h_{\lambda,\eta}(x)=\lambda+\eta (x-\lambda)$, for $x\in \A$. We write $\pi_{\lambda,\eta}$\index{p@$\pi_{\lambda,\eta}$} instead of  the orthogonal projection $\pi_{h_{\lambda,\eta}(\overline{\cP(\lambda)})}$ on $h_{\lambda,\eta}(\overline{\cP(\lambda)})$. 

Let $C$ be an alcove of $(\A,\cH)$ dominating $\lambda$. Then by \eqref{e_affine_Weyl_polytope_2}, we have $\overline{\cP(\lambda)}=\lambda+\overline{\cP^0_{-\lambda+\bb_C,\Phi_\lambda}}$. Therefore \begin{equation}\label{e_homothetic_polytope}
h_{\lambda,\eta}(\overline{\cP(\lambda)})=h_{\lambda,\eta}(\lambda)+\eta \conv(W^v_{\Phi_\lambda}.(-\lambda+\bb_C))=\lambda+\overline{\cP^0_{\eta(-\lambda+\bb_C),\Phi_\lambda}}.
\end{equation}

\begin{Remark}\label{r_fibers_pi_lambda_eta}
    Let $F\in \sF$ and $a\in \cF_F$. Then similarly as in Lemma~\ref{l_description_fibers_pi_affine}, we have  \[(\pi_{\lambda,\eta})^{-1}(h_{\lambda,\eta}(a))=h_{\lambda,\eta}(a)+\R_{\geq 0}(\overline{F}-\lambda).\] 
\end{Remark}

\begin{Lemma}\label{l_projection_stabilizing_affine_faces}
    Let $F$ be a face of $(\A,\cH)$, $\lambda\in \ve(F)$ and $\eta\in ]0,1]$. Then $\pi_{\lambda,\eta}(F)\subset F$ and $\pi_{\lambda,\eta}(\overline{F})\subset \overline{F}$.
\end{Lemma}

\begin{proof}
    Let $x\in F$. Then $x\in \lambda+\R_{>0}(F-\lambda)$. By Lemma~\ref{l_projection_contained_face} and Eq.~\eqref{e_homothetic_polytope}, $\pi_{\lambda,\eta}(x)\in \lambda+\R_{>0}(F-\lambda)$. Let $E$ be the star of $\lambda$. By Lemma~\ref{l_inclusion_P(lambda)_union_alcoves}, since $h_{\lambda,\eta}$ stabilizes $E$, we have $h_{\lambda,\eta}(\overline{\cP(\lambda)})\subset E$. Therefore by Lemma~\ref{l_star_vertex_intersected_face}, we have $\pi_{\lambda,\eta}(x)\in E\cap (\lambda+\R_{>0}(F-\lambda))=F$. As $\pi_{\lambda,\eta}$ is continuous (by \cite[A. (3.1.6)]{hiriart2012fundamentals}), we deduce the result.
\end{proof}

\begin{Proposition}\label{p_conical_description_faces}
    Let $F$ be a face of $(\A,\cH)$. Then \[F=\bigcap_{\lambda\in \ve(F)} \left(\lambda+\R_{>0}(F-\lambda)\right).\]
\end{Proposition}

\begin{proof}
Let $J=\ve(F)$. Set $E=\bigcap_{\lambda\in J} \left(\lambda+ \R_{> 0} (F-\lambda)\right)$. As $F\subset \lambda+\R_{>0}(F-\lambda)$ for $\lambda\in J$, we have $F\subset E$. 

We now prove the converse inclusion. By Lemma~\ref{l_adjacent_vertices}, we have: 
$$E=\bigcap_{\lambda\in J} \left(\lambda+\bigoplus_{\mu\in J} \R_{> 0} (\mu-\lambda)\right).$$
We first assume that $\Phi$ is irreducible. Then by \cite[V 3.9 Proposition 8]{bourbaki1981elements}, $C$ is a simplex and thus  for every $\lambda\in \ve(C)$, $(\mu-\lambda)_{\mu\in \ve(C)\setminus\{\lambda\}}$ is a basis of $\A$. Fix $\lambda\in J$. Up to changing the $0$ of the affine space $\A$, we may assume $\lambda=0$. Let $x\in E$. Then we can write $x=\sum_{\mu\in J\setminus \{0\}} t_\mu \mu$, where $(t_\mu)\in (\R_{> 0})^{J\setminus \{0\}}$. Let $\nu\in J$. Then $x\in \nu+\R_{> 0} (-\nu)+\bigoplus_{\mu\in J\setminus\{0\}}\R_{> 0}(\mu-\nu)$, so we can write $x=\nu+t'_0(-\nu)+ \sum_{\mu\in J\setminus \{0,\nu\}}t'_\mu (\mu-\nu)$ for some $t_0' \in \mathbb{R}$ and $(t'_{\mu})\in (\R_{> 0})^{J\setminus\{\nu\}}$. Thus we have \[(1-t'_0-\sum_{\mu\in J\setminus\{0,\nu\}} t'_\mu) \nu+\sum_{\mu\in J\setminus\{0,\nu\}} t'_\mu \mu=\sum_{\mu\in J\setminus\{0\}} t_\mu \mu.\] Therefore $t'_\mu=t_\mu$ for $\mu\in J\setminus\{0,\nu\}$, $t_\nu< 1$ and $\sum_{\mu\in J\setminus\{0\}} t_\mu=(1-t_0'- \sum_{\mu\in J\setminus\{0,\nu\}} t'_\mu)+\sum_{\mu\in J\setminus\{0,\nu\}} t'_\mu< 1$. This proves that $E\subset \Int_r(\conv(J))=F$ and thus when $\Phi$ is irreducible, we have $E=F$.

We no longer assume the irreducibility of $\Phi$. In the notation of Eq. \eqref{e_decomposition_Phi}, with $\Phi$ instead of $\Psi$, we write $\A=\A_1\times \ldots\times \A_\ell$. Let $x=(x_i)_{i\in \llbracket 1,\ell\rrbracket}\in \A$ and $i\in \llbracket 1,\ell\rrbracket$. We have $F=F_1\times \ldots \times F_\ell$, where $F_i\in \sF(\A_i,\cH_i)$ for $i\in \llbracket 1,\ell\rrbracket$. 

We decompose $J$ as  $J=J_1\times \ldots  \times J_\ell$, where $J_i= \ve(F_i)$ for $i\in \llbracket 1,\ell\rrbracket$. Let $x\in E$. 
 Then: \begin{align*} x&\in\bigcap_{\lambda\in J_i}\left(\prod_{j=1}^{i-1}\A_j\times (\lambda+\sum_{\mu\in J_i\setminus\{\lambda\}} \R_{> 0} (\mu-\lambda))\times \prod_{j=i+1}^\ell \A_j\right)\\
 &=\prod_{j=1}^{i-1} \A_j\times \bigcap_{\mu\in J_i\setminus\{\lambda\}}\left(\lambda+ \sum_{\mu\in J_i} \R_{> 0} (\mu-\lambda)\right)\times \prod_{j=i+1}^\ell \A_j\end{align*}
  By the first case we deduce $x_i\in \Int_r(\conv(J_i))$. Consequently $x\in \prod_{i=1}^\ell \Int_r(\conv(J_i))=\Int_r(\conv(J))$, which completes the proof of the proposition. 
\end{proof}

\begin{Lemma}\label{l_projection_homothety}
Let $\mu\in \ve(\A)$, $\eta\in [0,1]$ and $x\in \A$. Let $F\in \sF(\cH)$ be such that $\pi_\mu(x)\in \cF_{F}$. Let $\lambda\in \ve(F)$. Then $\pi_{\mu,\eta}(\hh_{\lambda,\eta}(x))=\hh_{\mu,\eta}(\pi_\mu(x))$. 
\end{Lemma}

\begin{proof}
Let $a=\pi_\mu(x)$. By Lemma~\ref{l_description_fibers_pi_affine}, we have $x-a\in  \R_{\geq 0}(\overline{F}-\mu)$. We have: \[
\hh_{\lambda,\eta}(x)=\hh_{\mu,\eta}(a)+(1-\eta)(\lambda-\mu)+\eta(x-a)\in  \hh_{\mu,\eta}(a) +  \R_{\geq 0} (\overline{F}-\mu),\] and we conclude with Remark~\ref{r_fibers_pi_lambda_eta}.

\end{proof}

\begin{Definition}\label{d_Feta}
For $x\in \A$, we denote by $F_{\cP}(x)$\index{f@$F_{\cP}(x)$} the face of $(\A,\cH)$ such that $x\in \cF_{F_\cP(x)}$, which is well-defined by Theorem~\ref{t_tessellation_A}. For $\eta\in [0,1]$, we set $F_\eta(x)=\Int_r(\conv(\{h_{\lambda,\eta}(x)\mid \lambda\in \ve(F_\cP(x))\}))$ and $\overline{F_\eta(x)}= \conv(\{h_{\lambda,\eta}(x)\mid \lambda\in  \ve(F_\cP(x))\})$ (see Figures~\ref{f_Feta_A2} and \ref{f_Feta_A1A1}). We have:  \begin{equation}\label{e_description_F_eta_bar}
    \overline{F_\eta(x)}=\eta x +(1-\eta)\conv\left(\ve\left(F_{\cP}\left(x\right)\right)\right)=\eta x +(1-\eta)\overline{F_\cP(x)}.
\end{equation}
\end{Definition}

 \begin{figure}[h]
 \centering
 \includegraphics[scale=0.3]{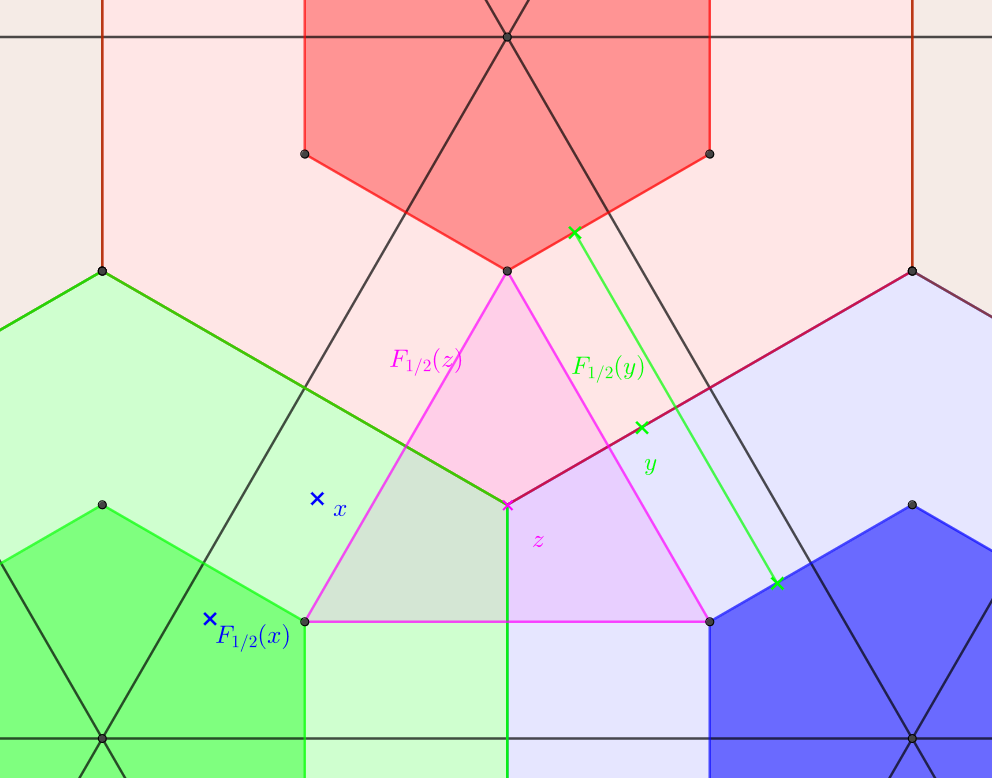}
 \caption{Pictures of $F_\eta$ in the case of $\mathrm{A}_2$.}
 \label{f_Feta_A2}
\end{figure}

 \begin{figure}[h]
 \centering
 \includegraphics[scale=0.35]{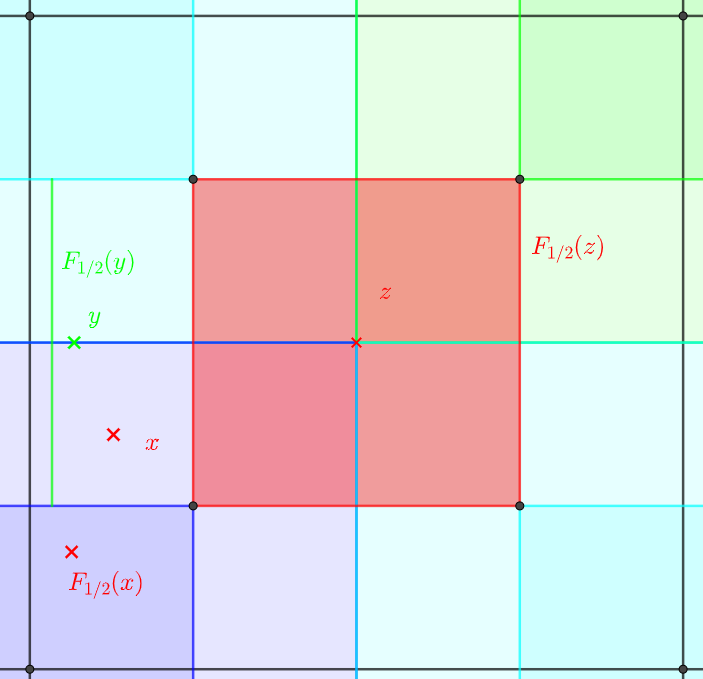}
 \caption{Pictures of $F_\eta$ in the case of $\mathrm{A}_1\times \mathrm{A}_1$.}
\label{f_Feta_A1A1}
\end{figure}

\begin{Lemma}\label{l_projection_F_eta}
Let $C$ be an alcove of $(\A,\cH)$ and $x\in \overline{C}$.  Then: \begin{enumerate}
\item for all $y\in F_\eta(x)$, for all $\mu\in \ve(C)$,  we have $\pi_{\mu,\eta}(y)=h_{\mu,\eta}(\pi_\mu(x))$,

\item  for all $\mu\in \ve(C)$, for all $y\in F_\eta(x)$,  we have  $\ve(F_{\cP}(x))\subset \{\nu\in \ve(C)\mid \pi_{\mu,\eta}(y)\in h_{\mu,\eta}(\overline{\cP(\nu)})\}$ and this is an equality if and only if $\mu\in \ve(F_{\cP}(x))$. 
\end{enumerate}
\end{Lemma}

\begin{proof}
  By Lemma~\ref{l_inclusion_P(lambda)_union_alcoves}, we have $\ve(F_{\cP}(x))\subset \ve(C)$.   Let $\mu\in \ve(C)$.  By Propositions~\ref{p_faces_intersections} and   \ref{p_projection_polytope_neighbour}, we have \begin{equation}\label{e_set_containing_pi_mu}
\pi_\mu(x)\in \bigcap_{\nu\in \ve(F_\cP(x))\cup \{\mu\}} \overline{\cP(\nu)}.
\end{equation}  Let $\lambda\in \ve(F_{\cP}(x))$. By Proposition \ref{p_faces_intersections}, we have $\pi_\mu(x)\in  \overline{\cF_{F_{\cP}(x)}}$. Hence, if $F'$ is the face such that $\pi_\mu(x)\in \cF_{F'}$, we deduce that $\ve(F_\cP(x))\subset F'$. By Lemma~\ref{l_projection_homothety}, we get $\pi_{\mu,\eta}(\hh_{\lambda,\eta}(x))=\hh_{\mu,\eta}(\pi_\mu(x))$. By Eq. \eqref{e_characterization_projection}, $\pi_{\mu,\eta}^{-1}\left(\hh_{\mu,\eta}\left(\pi_\mu\left(x\right)\right)\right)$  is convex. Therefore $\pi_{\mu,\eta}(F_\eta(x))=\{h_{\mu,\eta}(\pi_\mu(x))\}$.

Let $\mu\in \ve(C)$. Then \begin{align*}
\{\nu\in \ve(C)\mid \pi_{\mu,\eta}(y)\in h_{\mu,\eta}(\overline{\cP(\nu)})\}&=\{\nu\in \ve(C)\mid h_{\mu,\eta}(\pi_\mu(x))\in h_{\mu,\eta}(\overline{\cP(\nu)})\}\\
&= \{\nu\in \ve(C)\mid \pi_\mu(x)\in \overline{\cP(\nu)}\}.
\end{align*} 
We conclude with Eq . \eqref{e_set_containing_pi_mu}, since $\pi_\mu(x)=x$, for $\mu\in \ve(F_\cP(x))$. 
\end{proof}

\begin{Lemma}\label{l_projection_parabolic_affine_subspaces_homothety}
Let $F\in \sF(\cH)$.  Let  $f\in \cF_F$ and $C$ be an alcove dominating $f$. Let $\lambda\in \ve(F)$ and $x\in \left(h_{\lambda,\eta}(f)+\vect(F-\lambda)\right)\cap \overline{C}$. Then \[\pi_{\lambda,\eta}(x)\in h_{\lambda,\eta}(f)+\vect(F-\lambda)\] and \[\{\mu\in \ve(\A)\mid \pi_{\lambda,\eta}(x)\in h_{\lambda,\eta}(\overline{\cP(\mu)})\}\subset \ve(F).\]   
\end{Lemma}

\begin{proof}
    We have \begin{align*}
        &h_{\lambda,\eta}(f)+\vect(F-\lambda)\cap \overline{C}\\ 
        &\subset \left(h_{\lambda,\eta}(f)+\vect(F-\lambda)\right)\cap \left(\lambda+\R_{\geq 0}(\overline{C-\lambda)}\right)\\
        &= h_{\lambda,\eta}\left(f+\vect(F-\lambda)\right)\cap h_{\lambda,\eta}\left(\lambda+\R_{\geq 0}(\overline{C}-\lambda)\right)\\
        &=h_{\lambda,\eta}\left(\left(f+\vect(F-\lambda)\right)\cap \left(\lambda+\R_{\geq 0}(\overline{C}-\lambda)\right)\right).
    \end{align*}

    By Lemma~\ref{l_projection_homothety} applied to $\mu=\lambda$, we have $\pi_{\lambda,\eta}\circ h_{\lambda,\eta}=h_{\lambda,\eta}\circ \pi_\lambda$. Therefore if $x\in \left(h_{\lambda,\eta}(f)+\vect(F-\lambda)\right)\cap \overline{C}$, we have \[\pi_{\lambda,\eta}(x)=h_{\lambda,\eta}(\pi_\lambda(y)),\] for some $y\in  \left(f+\vect(F-\lambda)\right)\cap \left(\lambda+\R_{\geq 0}(\overline{C}-\lambda)\right)$. Using Proposition~\ref{p_projection_preserving_affine_spaces}, we have $\pi_\lambda(y)\in f+\vect(F-\lambda)$ and we deduce $\pi_{\lambda,\eta}(x)\in h_{\lambda,\eta}(f)+\vect(F-\lambda)$.

    Now we have $\pi_\lambda(y)=\pi_{\overline{\cP(\lambda)}}(y)=\pi_{\overline{\cP(\lambda)}-\lambda}(y-\lambda)+\lambda$. By Eq. \eqref{e_affine_Weyl_polytope_2}, we have $\overline{\cP(\lambda)}-\lambda=\overline{\cP^0_{\bb_C-\lambda}}$. By Proposition~\ref{p_description_faces_affine} and Proposition~\ref{p_bijection_affine_faces_vectorial_faces_vertices}, the map $\cF_{F_1}\mapsto \cF^0_{\R_{>0}(F_1-\lambda)}$ is a bijection between the faces of $\overline{\cP(\lambda)}$ and those of $\overline{\cP(\lambda)}-\lambda$. Hence  we have $\cF^0_{\R_{>0}(F_1-\lambda)}=\cF_{F_1}-\lambda$, for $F_1\in \sF_\lambda(\cH)$. By Proposition~\ref{p_projection_preserving_affine_spaces}~(2), $\pi_{\overline{\cP(\lambda)}-\lambda}(y-\lambda)\in \cF^0_{\R_{>0}(F_1-\lambda)}$, for some subface $F_1$ of $F$ dominating $\lambda$. Therefore $\pi_\lambda(y)\in \cF_{F_1}$, and we get:
    \[\{\mu\in \ve(\A)\mid \pi_{\lambda,\eta}(x)\in h_{\lambda,\eta}(\overline{\cP(\mu)})\}=\{\mu\in \ve(\A)\mid \pi_{\lambda}(y)\in \overline{\cP(\mu)}\}=\ve(F_1),\] which proves the lemma. 
\end{proof}

\begin{Lemma}\label{l_convex_hull_projections}
Let $\eta\in ]0,1]$. Let $C\in \Alc(\cH)$ and $x\in \overline{C}$. For $\lambda\in \ve(C)$, set $J_\lambda:=J_{\lambda,\eta}(x):=\{\nu\in \ve(C)\mid \pi_{\lambda,\eta}(x)\in h_{\lambda,\eta}(\overline{\cP(\nu)})\}$. Let $N_C(x)=N_C:=\min \{|J_{\mu}|\mid \mu \in \ve(C)\}$\index{n@$N_C=N_C(x)$} and:
\[J=\{\lambda\in \ve(C) \mid |J_{\lambda}|= N_C \}.\] Then:\begin{enumerate}
    \item There exists a face $F$ of $(\A,\cH)$ such that $J=\ve(F)$. There exists $\ff_x\in \overline{C}\cap \cF_{F}$ such that $\pi_{\lambda,\eta}(x)=h_{\lambda,\eta}(\ff_x)$ for all $\lambda\in J$. More generally, for every $\mu\in \ve(C)$, we have $\pi_{\mu,\eta}(x)=h_{\mu,\eta}(\pi_\mu(\ff_x))$.

    \item For every $\mu\in \ve(C)$, we have $J_\mu\supset J$; in the particular case when $\mu \in J$, there is an equality $J_\mu=J$.

    \item We have   $x\in \overline{F_\eta(\ff_x)}$. 
\end{enumerate}   
\end{Lemma}

\begin{proof}
 Let $\lambda_0\in J$. Let $F$ be the face of $(\A,\cH)$ dominating $\lambda_0$ and such that $\pi_{\lambda_0,\eta}(x)\in h_{\lambda_0,\eta}(\cF_{F})$.  By Lemma~\ref{l_projection_stabilizing_affine_faces}, we have $\pi_{\lambda_0,\eta}(x)\in \overline{C}$. Let $F_1$ be the face of $\overline{C}$ containing $\pi_{\lambda_0,\eta}(x)$. By Lemma~\ref{l_inclusion_P(lambda)_union_alcoves}, we have that $\cF_F$ is contained in the star $E$ of $F$. As $\lambda_0\in \ve(F)$ and $\eta\in ]0,1]$, the set  $E$ is stabilized by $h_{\lambda_0,\eta}$ and hence  $h_{\lambda_0,\eta}(\cF_F)\subset E$. Therefore $F_1$ dominates $F$.
 
 Write $\pi_{\lambda_0,\eta}(x)=h_{\lambda_0,\eta}(\ff_x)$, with $\ff_x\in \cF_{F}\cap \overline{C}$.
We have: \begin{align*}
    J_{\lambda_0}&= \{\nu\in \ve(C)\mid \pi_{\lambda_0,\eta}(x)\in h_{\lambda_0,\eta}(\overline{\cP(\nu)})\}\\
    &=  \{\nu\in \ve(C)\mid h_{\lambda_0,\eta}(f_x)\in h_{\lambda_0,\eta}(\overline{\cP(\nu)})\}\\
    &= \{\nu\in \ve(C)\mid f_x\in \overline{\cP(\nu)}\}=\ve(F),
\end{align*}
 by Proposition~\ref{p_faces_intersections}. By Remark~\ref{r_fibers_pi_lambda_eta}, we have \[x\in h_{\lambda_0,\eta}(\ff_x)+\mathrm{vect}(F-\lambda_0).\]

Let $\mu\in \ve(F)$. Then \[x\in h_{\mu,\eta}(\ff_x)+(1-\eta)(\lambda_0-\mu)+\mathrm{vect}(F-\lambda_0)=h_{\mu,\eta}(\ff_x)+\mathrm{vect}(F-\mu).\] By Lemma~\ref{l_projection_parabolic_affine_subspaces_homothety}, we have $\pi_{\mu,\eta}(x)\in h_{\mu,\eta}(\ff_x)+\mathrm{vect}(F-\mu)$ and \[J_\mu=\{\nu\in \ve(\A)\mid \pi_{\mu,\eta}(x)\in h_{\mu,\eta}(\overline{\cP(\nu)})\}\subset \ve(F)=J_{\lambda_0}.\] 
By minimality of $J_{\lambda_0}$, we deduce \begin{equation}\label{e_J_mu_subset_J}
J_\mu=J_{\lambda_0}=\ve(F)\subset J. 
\end{equation}

By the previous computations and by Proposition~\ref{p_faces_intersections}, we have $h_{\mu,\eta}(\ff_x),\pi_{\mu,\eta}(x)\in h_{\mu,\eta}(\cF_F)\cap (x+\vect(F-\mu))$. It follows from Proposition~\ref{p_description_faces_affine} that   $\di(\supp(h_{\mu,\eta}(\cF_F)))=\di(\supp(\cF_F))=\di(\supp(F))^\perp$ and thus $h_{\mu,\eta}(\cF_F)\cap (x+\vect(F-\mu))$ is either empty or a singleton. Therefore  \[h_{\mu,\eta}(\ff_x)=\pi_{\mu,\eta}(x).\]

Let $\lambda \in \ve(F)$. By Lemma~\ref{l_description_fibers_pi_affine}, we have \[x\in \pi_{\lambda,\eta}(x)+\R_{\geq 0} (\overline{F}-\lambda)=\hh_{\lambda,\eta}(\ff_x)+\R_{\geq 0} (\overline{F}-\lambda).\] Using Proposition~\ref{p_conical_description_faces}, we deduce that $x$ belongs to: 
\begin{align*}
    \bigcap_{\lambda\in \ve(F)}\left(\hh_{\lambda,\eta}(\ff_x)+\R_{\geq 0}(\overline{F}-\lambda)\right)&= \eta\ff_x+(1-\eta)\bigcap_{\lambda\in \ve(F)}(\lambda+\R_{\geq 0} (\overline{F}-\lambda))\\&=\eta \ff_x+(1-\eta)\overline{F}.
\end{align*}
  Using \eqref{e_description_F_eta_bar} we deduce statement 3): \begin{equation}\label{e_convex_hull_hLambda_f_x}
        x\in \conv(\{h_{\lambda,\eta}(f_x)\mid \lambda\in \ve(F)\}).
    \end{equation}

Let $\mu\in \ve(C)$.  By Propositions~\ref{p_faces_intersections} and   \ref{p_projection_polytope_neighbour}, we have \begin{equation}\label{e_pi_mu_f_x_intersection}
    \pi_\mu(\ff_x)\in \bigcap_{\nu\in \ve(F)\cup \{\mu\}} \overline{\cP(\nu)}.
\end{equation} Let $F_2$ be the face of $(\A,\cH)$ such that $\pi_\mu(f_x)\in \cF_{F_2}$. Then by Proposition~\ref{p_faces_intersections}, we have $\ve(F_2)\supset \ve(F).$  Let $\lambda\in \ve(F)$. By Lemma~\ref{l_projection_homothety}, we have $\pi_{\mu,\eta}(\hh_{\lambda,\eta}(\ff_x))=\hh_{\mu,\eta}(\pi_\mu(\ff_x))$. By Eq.~\eqref{e_characterization_projection}, $(\pi_{\mu,\eta})^{-1}\left(\hh_{\mu,\eta}\left(\pi_\mu\left(\ff_x\right)\right)\right)$  is convex. Using Eq.~\eqref{e_convex_hull_hLambda_f_x} we deduce \[\pi_{\mu,\eta}(x)=\hh_{\mu,\eta}\left(\pi_\mu\left(\ff_x\right)\right).\]  Using Eq.~\eqref{e_pi_mu_f_x_intersection} we deduce that $J_\mu\supset \ve(F)\cup\{\mu\}=J_{\lambda_0}\cup \{\mu\}$. Consequently for $\mu\in \ve(C)\setminus J_{\lambda_0}$, we have $|J_\mu|\geq |J_{\lambda_0}|+1=N_C+1$ and $\mu\notin J$. Therefore  $J\subset J_{\lambda_0}$ and hence by \eqref{e_J_mu_subset_J}, we have \[J=J_{\lambda_0}=\ve(F),\] which completes the proof of the lemma.

\end{proof}

\begin{Lemma}\label{l_projection_separation}
Let $C\in \Alc(\cH)$ and $x,x'\in  \overline{C}$. Then $x=x'$ if and only if (for all $\lambda\in \ve(C)$, $\pi_\lambda(x)=\pi_\lambda(x')$).
\end{Lemma}

\begin{proof}
Assume that $\pi_\lambda(x)=\pi_\lambda(x')$ for every $\lambda\in \ve(C)$. For $\lambda\in \ve(C)$, set \[J_\lambda=\{\mu\in \ve(C)\mid \pi_{\lambda}(x)\in \overline{\cP(\lambda)}\cap \overline{\cP(\mu)}\}.\] Let $J=\{\lambda\in \ve(C)\mid |J_\lambda|=N_C(x)   \}$. By Lemma~\ref{l_convex_hull_projections} applied with $\eta=1$, we have $\{x,x'\}\subset \bigcap_{\lambda\in J} \overline{\cP(\lambda)}$. Take $\lambda\in J$. Then $\pi_\lambda(x)=x=\pi_\lambda(x')=x'$. 
\end{proof}

\begin{Lemma}\label{l_inclusion_F_eta_face}
Let $x\in \A$, $\eta\in ]0,1]$ and $F$ be the face of $(\A,\cH)$ containing $x$. Then $\overline{F_\eta(x)}\subset F$. 
\end{Lemma}

\begin{proof}
    Let $F_1\in \sF(\cH)$ be  such that $x\in \cF_{F_1}$. Then by Lemma~\ref{l_inclusion_P(lambda)_union_alcoves}, we have $F_1\subset \overline{F}$. For $\lambda\in \ve(F_1)$, we thus have $]\lambda,x]\subset F$ and $h_{\lambda,\eta}(x)\in F$. Consequently $\overline{F_\eta(x)}=\conv(\{h_{\lambda,\eta}(x)\mid \lambda\in \ve(F_1\})\subset F$. 
\end{proof}

\begin{Theorem}\label{t_thickened_tessellation}
Let $\A$ be a finite dimensional vector space and $\Phi$ be a finite root system in $\A$. Let $\cH=\{\alpha^{-1}(\{k\})\mid \alpha \in \Phi, \,\, k\in \Z\}$. We choose a fundamental alcove $C_0$  of $(\A,\cH)$ and choose $\bb_0\in C_0$. If $\lambda$ is a vertex of $(\A,\cH)$, we choose an alcove $C_\lambda$ dominating $\lambda$ and  we set $\bb_{C_\lambda}=w.\bb_0$ if $w$ is the element of the affine Weyl group $W^{\mathrm{aff}}$ of $(\A,\cH)$ sending $C_0$ to $C$. We set $\overline{\cP(\lambda)}=\conv(W^{\mathrm{aff}}_\lambda.\bb_{C_\lambda})$, where $W^{\mathrm{aff}}_\lambda$ is the fixator of $\lambda$ in $W^{\mathrm{aff}}$. If $x\in \A$, we define $F_{\cP}(x)$ as the unique face of $(\A,\cH)$ such that $x\in \cF_{F_{\cP}(x)}:=\Int_r(\conv(W^{\mathrm{aff}}_{F_{\cP}(x)}.\bb_C))$, where $\bb_C$ is an alcove dominating $F_{\cP}(x)$ and $W^{\mathrm{aff}}_{F_\cP(x)}$ is the fixator of $F_{\cP}(x)$ in $W^{\mathrm{aff}}$. If $\eta\in ]0,1]$, we set $F_\eta(x)=\eta x+(1-\eta)F_{\cP}(x)$.    Then we have: \[\forall \eta\in ]0,1],\ \A=\bigsqcup_{x\in \A} \overline{F_\eta(x)}.\]
\end{Theorem}

\begin{proof}
Let $C$ be an alcove of $(\A,\cH)$ and $y\in \overline{C}$. For $\lambda\in \ve(C)$, set  $J_{\lambda,\eta}(y)=\{\mu\in \ve(C)\mid \pi_{\lambda,\eta}(y)\in h_{\lambda,\eta}(\overline{\cP(\mu)})\}$. Let
$N_{C,\eta}=\min \{|J_{\mu,\eta}(y)|\mid \mu\in \ve(C)\}$, and 
\[\tilde{J}=\{\lambda\in \ve(C)\mid |J_{\lambda,\eta}(y)|=N_{C,\eta}\}.\] By Lemma~\ref{l_convex_hull_projections}, there exists  $\ff_y\in \Int_r\left(\bigcap_{\nu\in\tilde{J}}\overline{\cP(\nu)}\right)$ such that $y\in \overline{F_\eta(\ff_y)}$. Therefore $\bigcup_{x\in \overline{C}} \overline{F_\eta(x)}=\overline{C}$.

Let $x,x'\in \A$ be such that  $\overline{F_\eta}(x)\cap \overline{F_{\eta}}(x')\neq \emptyset$ and let $y\in \overline{F_\eta(x)}\cap \overline{F_{\eta}(x')}$. Let $F$ and $F'$ be the faces of $(\A,\cH)$ containing $x$ and $x'$ respectively. Then by Lemma~\ref{l_inclusion_F_eta_face}, $y\in F\cap F'$, which proves that $F=F'$. Let $C$ be an alcove dominating $F$.   By Lemma~\ref{l_projection_F_eta}, we have $\pi_{\mu,\eta}(y)=\hh_{\mu,\eta}(\pi_\mu(x))=\hh_{\mu,\eta}(\pi_\mu(x'))$ for every $\mu\in \ve(C)$. Therefore $\pi_\mu(x)=\pi_\mu(x')$ for every $\mu\in \ve(C)$ and $x=x'$ by Lemma~\ref{l_projection_separation}.
\end{proof}

\begin{Corollary}
  Let $\eta\in ]0,1]$.  We have: \[\A=\bigsqcup_{\lambda\in \ve(\A)}h_{\lambda,\eta}(\cP(\lambda))\sqcup \bigsqcup_{x\in \bigcup_{\lambda\in \ve(\A)}\Fr(\cP(\lambda))} \overline{F_\eta(x)}.\]
\end{Corollary}

\begin{figure}[h]
 \centering
 \includegraphics[scale=0.35]{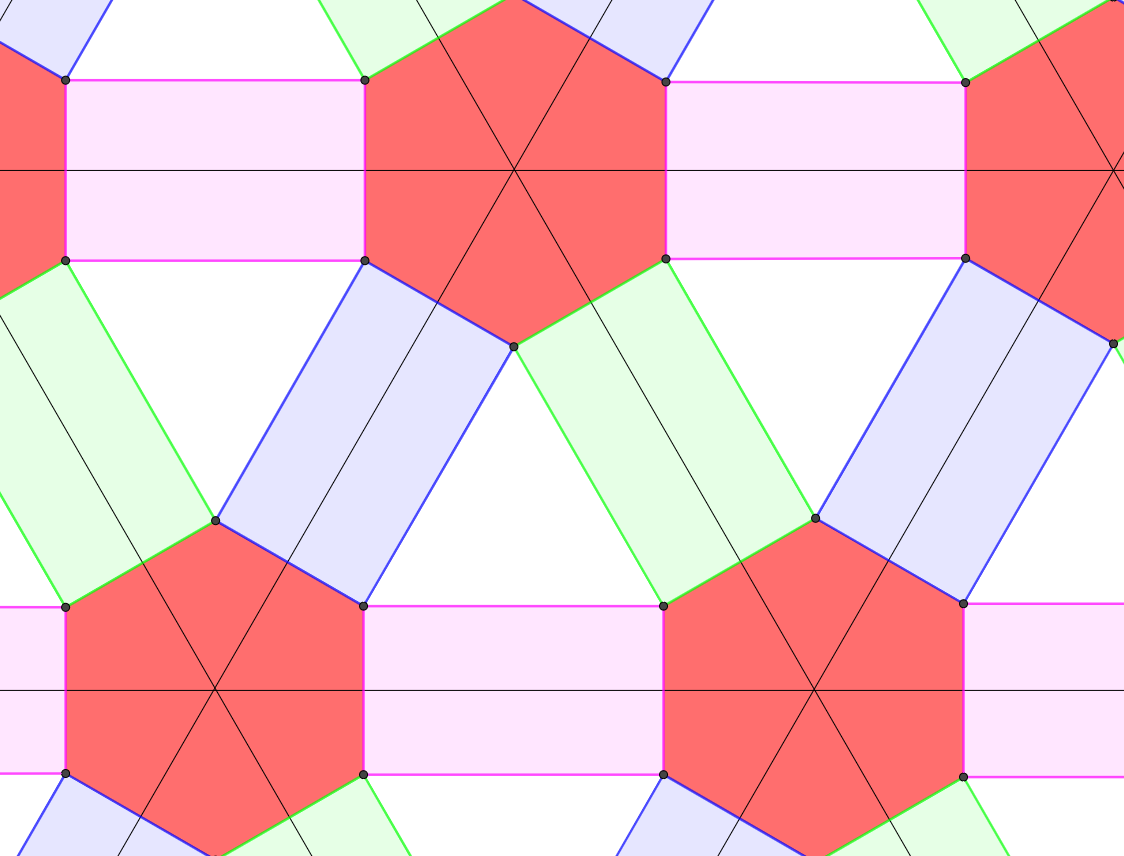}
 \caption{Illustration of Theorem~\ref{t_thickened_tessellation} in the case of $\mathrm{A}_2$ and $\eta=1/2$.}
 \label{f_explanation_t_thickened_A2}
\end{figure}

\begin{figure}[h]
 \centering
 \includegraphics[scale=0.35]{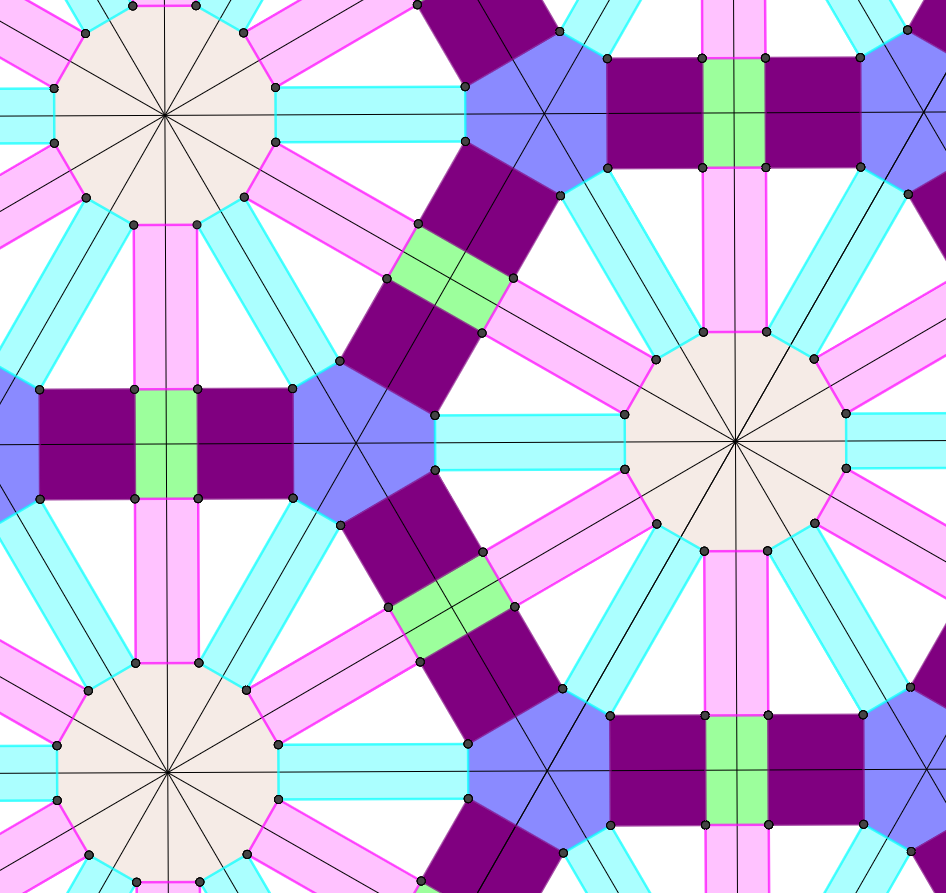}
 \caption{Illustration of Theorem~\ref{t_thickened_tessellation} in the case of $\mathrm{G}_2$ and $\eta=1/2$.}
 \label{f_explanation_t_thickened_G2}  
\end{figure}

\begin{Remark}
\begin{enumerate}
\item Note that when $\eta=0$, we have $F_0(\bb_C)=\overline{C}$ for each alcove $C$. Therefore we still have $\A=\bigcup_{x\in \A} F_0(x)$, but this is no longer a disjoint union according to Lemma \ref{l_intersection_P(lambda)_P(mu)_adjacent}.

\item Let $F$ be a face of $(\A,\cH)$. Let $x\in \cF_F$ and $\eta\in ]0,1[$. Then for every vertex $\lambda$ of $F$, we have $F_\eta(x)=\eta x+(1-\eta) F=h_{\lambda,\eta}(x)+(1-\eta)(F-\lambda)$. Indeed, we have $h_{\lambda,\eta}(x)=\lambda+\eta(x-\lambda)=(1-\eta)\lambda+\eta x$.

\end{enumerate}

\end{Remark}

\printindex

\bibliography{bibliographie.bib}

\end{document}